\documentclass[letterpaper]{article}
\usepackage{lmodern}

\usepackage[T1]{fontenc}
\usepackage[latin9]{inputenc}
\usepackage{refstyle}
\usepackage{float}
\usepackage{enumitem}
\usepackage{amsmath}
\usepackage{amssymb}
\usepackage{stmaryrd}
\PassOptionsToPackage{normalem}{ulem}
\usepackage{ulem}
\usepackage[pdfusetitle,
 bookmarks=true,bookmarksnumbered=false,bookmarksopen=false,
 breaklinks=false,pdfborder={0 0 1},backref=false,colorlinks=false]
 {hyperref}

\makeatletter

\global\long\def\N{\mathbb{N}}%
\global\long\def\Z{\mathbb{Z}}%
\global\long\def\Q{\mathbb{Q}}%

\global\long\def\Lang{\mathcal{L}}%

\global\long\def\isnotwd{\mathord{\uparrow}}%

\global\long\def\upto{\mathord{\upharpoonright}}%
\global\long\def\ZFniptc{\mathsf{ZF}\mathord{-}\mathsf{inf}\mathord{+}\mathsf{TC}}%
\global\long\def\ZFfinptc{\mathsf{ZFfin}\mathord{+}\mathsf{TC}}%

\global\long\def\llor{\bigvee\mathchoice{\hspace{-1.1em}}{\hspace{-0.7em}}{\hspace{-0.7em}}{\hspace{-0.6em}}\bigvee}%
\global\long\def\llorc{\llor^{\mathrm{c}}}%

\AtBeginDocument{\providecommand\thmref[1]{\ref{thm:#1}}}
\AtBeginDocument{\providecommand\rmkref[1]{\ref{rmk:#1}}}
\AtBeginDocument{\providecommand\defref[1]{\ref{def:#1}}}
\AtBeginDocument{\providecommand\propref[1]{\ref{prop:#1}}}
\AtBeginDocument{\providecommand\secref[1]{\ref{sec:#1}}}
\AtBeginDocument{\providecommand\enuref[1]{\ref{enu:#1}}}
\AtBeginDocument{\providecommand\lemref[1]{\ref{lem:#1}}}
\AtBeginDocument{\providecommand\corref[1]{\ref{cor:#1}}}
\AtBeginDocument{\providecommand\questionref[1]{\ref{question:#1}}}
\pdfpageheight\paperheight
\pdfpagewidth\paperwidth

\RS@ifundefined{subsecref}
  {\newref{subsec}{name = \RSsectxt}}
  {}
\RS@ifundefined{thmref}
  {\def\RSthmtxt{theorem~}\newref{thm}{name = \RSthmtxt}}
  {}
\RS@ifundefined{lemref}
  {\def\RSlemtxt{lemma~}\newref{lem}{name = \RSlemtxt}}
  {}

\usepackage[amsmath,thmmarks]{ntheorem}
\newref{thm}{refcmd={Theorem~\ref{#1}}}
\theoremstyle{plain}
\theorembodyfont{\upshape}
\theoremheaderfont{\bfseries}
\theoremseparator{.}
\theoremnumbering{arabic}
\theoremsymbol{}
\newtheorem{thm}{Theorem}[subsection]
\newref{def}{refcmd={Definition~\ref{#1}}}
		\newtheorem{definition}[thm]{Definition}
\newref{rmk}{refcmd={Remark~\ref{#1}}}
		\newtheorem{remark}[thm]{Remark}
\theoremstyle{nonumberplain}
\theoremheaderfont{\itshape}
\theorembodyfont{\upshape}
\theoremseparator{:}
\theoremsymbol{\ensuremath{\blacksquare}}
\newtheorem{proof}{Proof}

\theoremstyle{plain}
\theorembodyfont{\upshape}
\theoremheaderfont{\bfseries}
\theoremseparator{.}
\theoremnumbering{arabic}
\theoremsymbol{}
\theoremstyle{nonumberplain}
\theorembodyfont{\upshape}
\theoremheaderfont{\bfseries}
\theoremseparator{.}
\renewtheorem{thm*}{Theorem}
\theoremstyle{plain}
\theorembodyfont{\upshape}
\theoremheaderfont{\bfseries}
\theoremseparator{.}
\theoremnumbering{arabic}
\theoremsymbol{}
\theoremstyle{nonumberplain}
\theorembodyfont{\upshape}
\theoremheaderfont{\bfseries}
\theoremseparator{.}
\newtheorem{question**}{Question}
\theoremstyle{plain}
\theorembodyfont{\upshape}
\theoremheaderfont{\bfseries}
\theoremseparator{.}
\theoremnumbering{arabic}
\theoremsymbol{}
\theoremstyle{nonumberplain}
\theorembodyfont{\upshape}
\theoremheaderfont{\bfseries}
\theoremseparator{.}
		\renewtheorem{definition*}{Definition}
\theoremstyle{plain}
\theorembodyfont{\upshape}
\theoremheaderfont{\bfseries}
\theoremseparator{.}
\theoremnumbering{arabic}
\theoremsymbol{}
\newref{prop}{refcmd={Proposition~\ref{#1}}}
		\newtheorem{prop}[thm]{Proposition}
\theoremstyle{nonumberplain}
\theorembodyfont{\upshape}
\theoremheaderfont{\bfseries}
\theoremseparator{.}
		\renewtheorem{prop*}{Proposition}
\theoremstyle{plain}
\theorembodyfont{\upshape}
\theoremheaderfont{\bfseries}
\theoremseparator{.}
\theoremnumbering{arabic}
\theoremsymbol{}
\newref{lem}{refcmd={Lemma~\ref{#1}}}
		\newtheorem{lemma}[thm]{Lemma}
\newref{cor}{refcmd={Corollary~\ref{#1}}}
		\newtheorem{corollary}[thm]{Corollary}
\newref{question}{refcmd={Question~\ref{#1}}}
		\newtheorem{question}{Question}

\usepackage{listings}
\usepackage{fullpage}
\usepackage{braket}

\usepackage[cal=euler]{mathalpha}

\setlist[enumerate,2]{ref=\theenumi.\theenumii}
\setlist[enumerate,3]{ref=\theenumi.\theenumii.\theenumiii}
\setlist[enumerate,4]{ref=\theenumi.\theenumii.\theenumiii.\theenumiv}

\newref{sec}{refcmd={Section \ref{#1}}}

\usepackage[style=numeric]{biblatex}
\DeclareFieldFormat{postnote}{#1}

\makeatother

\usepackage[style=numeric]{biblatex}
\begin{document}
\title{Escaping Tennenbaum's Theorem and a Strong Jump Inversion Theorem}
\author{Duarte Maia}

\maketitle

\section{Introduction}

\subsection{Background -- Tennenbaum's Theorem}

\nocite{lutz_walsh_tennenbaum}

It is a well-known result of Stanley Tennenbaum (see e.g.\ \cite{kaye_book})
that $\mathsf{PA}$, the first-order theory of Peano Arithmetic, does
not admit any non-standard computable model. In fact, in every non-standard
model of $\mathsf{PA}$, neither addition nor multiplication are computable.
Generalizations of this theorem can be found in the literature, including
versions of it for weaker fragments of $\mathsf{PA}$ \cite{wilmers_PA},
alternative operations \cite{schmerl_tennenbaum_recursive_reducts},
and finite set theory \cite{mancini_zambella_set_theories,computable_quotient_arithmetic_set_theory},
to name a few. However, all of these approaches and variations appear
to depend heavily on the specificities of the chosen signature. This
was observed by Fedor Pakhomov, who in 2022 \cite{pakhomov} constructed
a theory that is definitionally equivalent to $\mathsf{PA}$ (the
meaning of which will be described shortly, but roughly means ``it's
$\mathsf{PA}$ but with a different choice of signature'') for which
there is a computable nonstandard model. This shows that Tennenbaum's
theorem really is reliant on the choice of signature.

Let us elaborate on the way in which this theory is ``$\mathsf{PA}$
with a different signature'', by means of an example. In the literature,
the choice of signature used to define $\mathsf{PA}$ is not set in
stone. Some authors choose to include a symbol for successor, while
others choose to omit it. This is not a problem, as long as a symbol
for $1$ and a symbol for $+$ remain in the signature, because the
successor of $x$ can always be rewritten as $x+1$. In a similar
vein, some authors choose to omit a symbol for $1$, opting instead
to write it as the successor of zero. Thus, we have two distinct axiomatizations
of $\mathsf{PA}$, in two signatures neither of which contains the
other, but which are intuitively seen to both be able to define the
missing symbols from their counterpart. This idea of defining new
symbols in terms of previous ones underlies the notion of \emph{definitional
extension}, and the way in which ``$\mathsf{PA}$ without successor''
and ``$\mathsf{PA}$ without $1$'' are `kind of the same theory'
corresponds to the notion of \emph{definitional equivalence}.
\begin{definition}[Definitional Extension]
Let $T$ be a theory over a signature $\Lang$. A \emph{definitional
extension} of $T$ is a theory $S\supseteq T$, over a language $\Lang'\supseteq\Lang$,
such that
\begin{itemize}
\item Every theorem of $S$ that only uses symbols from the language $\Lang$
is also a theorem of $T$, and
\item Every symbol of $\Lang'\setminus\Lang$ is $S$-definable in terms
of symbols in $\Lang$. For example, if $P$ is any propositional
symbol in $\Lang'$, there is an $\Lang$-formula $\varphi$ such
that
\[
S\vdash\forall_{\vec{x}}(P(\vec{x})\leftrightarrow\varphi(\vec{x})),
\]
and a similar statement holds for function symbols.
\end{itemize}
\end{definition}
\begin{definition}[Definitional Equivalence]
Let $T_{1}$ and $T_{2}$ be two theories, whose signatures $\Lang_{1}$
and $\Lang_{2}$ respectively are assumed to be disjoint with no loss
of generality. We say that $T_{1}$ and $T_{2}$ are \emph{definitionally
equivalent} if there is a theory $T$ which is a definitional extension
of both.
\end{definition}
The notion of definitional equivalence is related, but not identical,
to a more common notion of ``equal power of theories'' called bi-interpretability:
definitional equivalence is strictly stronger, though these two notions
coincide for most natural examples. The interested reader will find
a thorough comparison of these two notions in \cite{biinterpretabilityvssynonymy}.

A nontrivial example of two definitionally equivalent theories is
that of $\mathsf{PA}$ and a specific version of finite set theory,
which we will refer to as $\ZFfinptc$. This theory, and a closely
related one, shall play an important role in the sequel, so we formally
introduce them to the reader.
\begin{definition}
The theory $\ZFniptc$ consists of the usual axioms of Zermelo-Fraenkel
set theory, with the removal of the axiom of infinity, and with the
addition of the so-called ``axiom of transitive closure'', which
states that any set is contained in a transitive set.
\end{definition}
\begin{remark}
\label{rmk:vonneumann}Under the remaining axioms of $\mathsf{ZF}$
(minus infinity), the axiom of transitive closure is equivalent to
the validity of $\in$-induction. It is also equivalent (under the
same circumstances) to the claim that every set is in some level of
the von~Neumann hierarchy $V_{\alpha}$. This is, in fact, the principal
way in which this axiom will be necessary for our work.
\end{remark}
\begin{definition}
The theory $\ZFfinptc$ consists of $\ZFniptc$, together with the
negation of the axiom of infinity.
\end{definition}
\begin{remark}
The usual theory of $\mathsf{ZF}$ consists of $\ZFniptc$, together
with the axiom of infinity -- the axiom of transitive closure adds
no deductive power in this case. In other words, the axiom of transitive
closure can be proven in $\mathsf{ZF}$. This is because the transitive
closure of a set can be obtained by iterating the operation $x\mapsto x\cup\bigcup x$
countably many times, but to define this countable iteration requires
recourse to $\omega$. This recourse cannot be avoided, and indeed
it turns out that $\mathsf{ZF}-\mathsf{inf}+\neg\mathsf{TC}$ is consistent;
see \cite{mancini_zambella_set_theories} or \cite{kaye_wong} for
a proof.
\end{remark}
\begin{thm}
\label{thm:pazffineqv}The theories $\mathsf{PA}$ and $\ZFfinptc$
are definitionally equivalent.
\end{thm}
\begin{proof}
See \cite{kaye_wong}.
\end{proof}
Besides serving as an interesting example of definitional equivalence,
\thmref{pazffineqv} also underlies the entirety of Pakhomov's construction,
as it is more convenient to work with a set-theoretical framework
than it would be to work directly with $\mathsf{PA}$. It also gives
us a little bit more flexibility, allowing us to construct computable
models of theories definitionally equivalent to $\mathsf{ZF}$ or
$\mathsf{ZFC}$, though we will not explore this avenue in this document.
We now briefly present the idea behind Pakhomov's construction, before
elaborating on our original contributions.

\subsection{Pakhomov's Construction and New Results}

\global\long\def\Tboth{T_{\mathrm{both}}}%
In Fedor Pakhomov's paper ``How to Escape Tennenbaum's Theorem''
\cite{pakhomov}, the author defines a theory $T_{S}$, definitionally
equivalent to $\ZFniptc$, which axiomatizes a certain ternary predicate
$S(x,y,z)$. This ternary predicate acts as ``set membership with
witnesses'', and is initially defined in $\ZFniptc$ by transfinite
recursion in the von~Neumann hierarchy (see also \rmkref{vonneumann})
in such a way as to satisfy the rule
\begin{equation}
\ZFniptc\vdash\forall_{x,y}(x\in y\leftrightarrow\forall_{z}S(x,y,z)).\label{eq:Sdef}
\end{equation}
Since $S$ is initially defined in the context of set theory, it is
definable in terms of set membership, and Equation (\ref{eq:Sdef})
ensures that set membership is itself definable in terms of $S$,
which is why these two theories are definitionally equivalent.
\begin{definition}
\label{def:tboth}Henceforth, $\Tboth$ denotes the set of all sentences
provable from $\ZFniptc$ plus the definition of $S$ in the signature
containing both $S$ and set inclusion. Moreover, $T_{S}$ denotes
the theory of $S$, that is, $\Tboth$ restricted to the sentences
containing only the predicate $S$. 
\end{definition}
The main property of this predicate $S$ is that it is endowed with
a certain `flexibility' or `genericity'. As we will later make
precise, if we are attempting to build a model of $T_{S}$ and make
some mistakes along the way, any extra `trash' elements we might
have added on accident can still be reused as `real' elements added
later on. Pakhomov exploits this flexibility to prove:
\begin{thm}
\label{thm:premain}If $T$ is a consistent c.e.\ theory extending
$T_{S}$, then $T$ admits a computable model.
\end{thm}
Pakhomov proves this using an explicit ``Henkin construction'' type
of argument. We reframe his work in terms of the following stronger
result:
\begin{thm*}[\ref{thm:basic_pakhomov_jump_inversion}]
If $D$ is a $0'$-computable model of $\Tboth$, the reduct $D\upto S$
admits a computable copy.
\end{thm*}
We provide a brief sketch of how \thmref{basic_pakhomov_jump_inversion}
implies \thmref{premain}. If $T$ is a consistent c.e.\ theory extending
$T_{S}$, consider the consistent c.e.\ theory $\bar{T}=T\cup T_{\mathrm{both}}$.
By the computable completeness theorem, there is a $0'$-computable
model $M$ of $\bar{T}$, whose computable reduct to $S$ is a computable
model of $T$ by \thmref{basic_pakhomov_jump_inversion}.

We broadly generalize the methods used by Pakhomov to answer a question
posed by him in the affirmative:
\begin{question**}[Pakhomov \cite{pakhomov}]
Fix a value of $n$. Are there theories definitionally equivalent
to ``$\mathsf{PA}$ plus all true $\Pi_{n}$ sentences'' that have
computable non-standard models?
\end{question**}
The motivation for this question is that, as Pakhomov proved in \cite{pakhomov},
this statement is not true `in the limit'. More precisely, Pakhomov
showed that any nonstandard model of a theory that is definitionally
equivalent to true arithmetic cannot be computable. This leads to
the question of whether a partial result can be recovered.

We answer Pakhomov's question via the following improvement on the
construction of Pakhomov's theory $T_{S}$:
\begin{thm*}[\ref{thm:nestedpakhomov}]
There is a nested sequence of consistent c.e.\ theories $\ZFniptc\subseteq T^{0}\subseteq T^{1}\subseteq T^{2}\subseteq\cdots$
satisfying the following properties:
\begin{itemize}
\item Each $T^{n}$ is in the signature containing $\in$ and predicates
$S^{0}$, $\dots$, $S^{n}$, with each $S^{i}$ being an $(i+2)$-ary
predicate symbol,
\item All of these extensions are conservative over $\ZFniptc$, in the
sense that they contain no additional theorems in the predicate $\in$,
\item The predicates $\in$ and $S^{n}$ are definable in terms of the other
within $T^{n}$, and
\item Given a $0'$-computable model $D$ of $T^{n}$ restricted to the
language containing $S^{i}$, $\dots$, and $S^{n}$, there is a computable
copy $M$ of $D\upto(S^{i+1},\dots,S^{n})$.
\end{itemize}
The last item relativizes: Given an $X'$-computable model of $T^{n}$
implementing the last $k$ predicates, there is an $X$-computable
copy of its reduct to the last $k-1$ predicates.

\end{thm*}
\begin{remark}
The second and third bullet point imply that, for every $n$, $\ZFniptc$
and $T^{n}\upto S^{n}$ are definitionally equivalent.
\end{remark}
From \thmref{nestedpakhomov}, we conclude one of the main results
of this paper:
\begin{thm*}[\ref{thm:pakhomov_answer}]
For every $n$, there is a theory definitionally equivalent to ``$\mathsf{PA}$
plus all true $\Pi_{n}$ sentences'' that admits a computable non-standard
model.
\end{thm*}
\begin{proof}
First, note that the theory $T$ containing $\mathsf{PA}$ plus all
true $\Pi_{n}$ sentences is a $0^{(n)}$-c.e.\ theory. By the results
of Kaye and Wong \cite{kaye_wong}, $\mathsf{PA}$ is definitionally
equivalent to $\ZFfinptc$, and therefore $T$ is definitionally equivalent
to a $0^{(n)}$-c.e.\ extension of the theory $T^{n+1}$ from \thmref{nestedpakhomov}.
Say $\bar{T}$ is this theory, and let $\bar{T}^{+}$ be the same
with an added constant $c$ and an axiom schema ensuring that $c$
is a nonstandard element. Let $D$ be a $0^{(n+1)}$-computable model
of $\bar{T}^{+}$, in the signature containing the $n+2$ predicates
$S^{0}$ to $S^{n+1}$, which exists by the computable completeness
theorem (see \cite{harizanov}). Then, apply the relativized version
of \thmref{nestedpakhomov} $n+1$ times, obtaining a computable model
of $\bar{T}\upto S^{n+1}$. This shows that $\bar{T}\upto S^{n+1}$
is the theory we sought: It is definitionally equivalent to ``$\mathsf{PA}$
plus all true $\Pi_{n}$ sentences'', and it admits a computable
nonstandard model.
\end{proof}
As an aside, we note that Pakhomov posed another question in his original
paper, which has since been answered by Lutz and Walsh.
\begin{question**}[Pakhomov \cite{pakhomov}]
Is there a c.e.\ theory $T$ such that no definitionally equivalent
theory $T'$ admits a computable model?
\end{question**}
\begin{thm}[Lutz--Walsh \cite{lutz_walsh_tennenbaum}]
There is a c.e.\ theory $T$ such that every model of every theory
$T'$ definitionally equivalent to $T$ is noncomputable.
\end{thm}

\subsection{A Strong Jump Inversion Theorem}

In the process of answering Pakhomov's question, it happened that
certain parts of our construction were not specific to the particularities
of arithmetic, and so we were able to extract a general-purpose theorem
for strong jump inversion.\footnote{Patrick Lutz was the first person to bring the idea of a broad generalization
to my attention. He also provided a working conjecture that was very
useful in developing the theory.}

For our purposes, a strong jump inversion type theorem is one of the
following form: ``If {[}structure of a certain type with a lot of
data{]} admits a copy computable in $X'$, then {[}reduct of the same
structure with fewer data{]} admits a copy computable in $X$.''
(See \rmkref{jumpinversion} below for a much more detailed description.)
Here are a few examples. We state them for the computable case, but
they all relativize uniformly in the obvious way.
\global\long\def\eqvrel{\simeq}%

\begin{thm}[{Folklore, {\cites[Theorem 9.1.4]{cst_downey_melnikov}}}]
\label{thm:JIeqv}Suppose that $(E,\eqvrel)$ is an equivalence relation
with infinitely many infinite equivalence classes, and for $n\in\N$
let $P_{\geq n}(x)$ be the predicate ``the equivalence class of
$x$ has at least $n$ elements''. If $(E,\eqvrel,\{P_{\geq n}\}_{n\in\N})$
admits a $0'$-computable copy, then $(E,\eqvrel)$ admits a computable
copy.
\end{thm}
\begin{thm}[Knight--Stob \cite{knight_stob_computable_boolean_algebras}, Downey--Jockusch
\cite{downey_jockusch_boolean}]
\label{thm:JIbool}Suppose that $B$ is a Boolean algebra, and suppose
that the structure $(B,\land,\lor,\mathrm{isAtom})$ admits a $0'$-computable
copy. Then, the structure $(B,\land,\lor)$ admits a computable copy.
\end{thm}
\begin{thm}[{Fellner, {\cites[VII.26]{montalban_cst}}}]
\label{thm:JIlinear}Suppose that $L$ is a linear order for which
every element admits a successor and predecessor, and suppose that
the structure $(L,<,S)$ admits a $0'$-computable copy, where $S(x,y)$
is the successor relation. Then, the structure $(L,<)$ admits a computable
copy.
\end{thm}
It would be interesting to find a broad theorem from which Theorems
\ref{thm:JIeqv}, \ref{thm:JIbool}, and \ref{thm:JIlinear} could
follow as corollaries. We have had some amount of success: Our general-purpose
\thmref{main_final} is (with some work) able to subsume Theorems
\ref{thm:JIeqv} and \ref{thm:JIlinear} (respectively in Sections
\ref{subsec:eqvrelsinf} and \ref{subsec:linord}), as well as some
other theorems in the literature (see Section \ref{sec:prevs_applications}).
We have not yet had success in obtaining \thmref{JIbool}, but we
do not believe this to be impossible -- rather, it happens that Boolean
algebras are more complex than most structures we work with in this
paper, and would require more significant effort to apply \thmref{main_final}
to.

We are not the first to seek such a theorem -- the 2018 paper \cite{strong_jump_inversion}
also contains a general-purpose theorem (Theorem 2.5) from which results
related to the above are partially recovered. Our theorem seems to
follow similar themes as the theorem found in \cite{strong_jump_inversion},
but our assumptions have a different flavor and our conclusion goes
in a different direction. Notably, our result requires some amount
of creativity in its application -- one must, in most cases, construct
an intermediate structure to apply our \thmref{main_final} to in
order to obtain a jump inversion result, while Theorem 2.5 from \cite{strong_jump_inversion}
is relatively straight-forward in its application to examples.
\begin{remark}
\label{rmk:jumpinversion} For a general framework on the meaning
of ``jump inversion theorem'', we refer the reader to Section 1.1
of \cite{strong_jump_inversion}. For completeness and conciseness,
we provide a very brief summary. The content and nomenclature of this
remark is mostly taken from \cite{montalban_cst}.

First, let us remark that there is a standard notion of ``the jump
of a structure'', see e.g.\ \cites[Definition IX.1]{montalban_cst},
which is known to subsume several other different notions of ``jump
of a structure'' that have arisen over the decades; see Historical
Remark IX.2 in \cite{montalban_cst} and Remark 1.11 in \cite{strong_jump_inversion}.
The notion of ``jump of a structure'' under discussion consists
of adding a certain predicate, or rather, countable family of predicates,
called Kleene's r.i.c.e.\ complete relation, to the language of a
structure.

While the definition of Kleene's r.i.c.e.\ complete relation is completely
generic and applies to any structure in the usual sense, it is markedly
unnatural, in the sense that it has surely never arisen outside the
context of computability. However, for some particular types of structures,
Kleene's relation may be replaced (up to ``relative Turing equivalence'',
which is a natural type of equivalence of relations on a structure
which compares their intrinsic computable complexity) by much more
natural families of relations. For example, on an arbitrary linear
ordering $(L,\leq)$, the adjacency relation together with $0'$ (encoded
as a countable family of zero-ary predicates) is relatively Turing
equivalent to Kleene's r.i.c.e.\ complete relation \cites[Example II.29]{montalban_cst},
and thus it is reasonable to identify the jump of a linear ordering
$(L,\leq)$ with the structure $(L,\leq,S)$ plus an oracle for $0'$.\footnote{\label{fn:adjvsnegadj}We may instead prefer the negation of the adjacency
relation, because checking non-adjacency in a linear order is c.e.\ while
checking adjacency is co-c.e.}

The fact that this oracle for $0'$ appears, even when the rest of
the ``computable information'' of the jump has been nicely encompassed
by the natural relation $S$, may be seen as a flaw. This gives rise
to the notion of ``the/a structural jump of a structure''. We refer
to Definition 1.12 in \cite{strong_jump_inversion}, which we paraphrase
in slightly different but equivalent terms:
\begin{definition*}
A structural jump of a structure $\mathcal{A}$ is an expansion $\mathcal{A}'=(\mathcal{A},(R_{i})_{i\in\N})$
such that the relations $R_{i}$ are uniformly r.i.c.e.\ in $\mathcal{A}$
(this is the natural analogue of the notion of c.e.\ in this context),
and such that Kleene's r.i.c.e.\ complete relation $K^{\mathcal{A}}$
is relatively intrinsically computable from $\mathcal{A}'\oplus0'$
(this the natural analogue of computable from the oracle $\mathcal{A}'\oplus0'$).
\end{definition*}
It should be noted that the jump of a structure is always a structural
jump of said structure, but the benefit of this definition is that,
in some cases, there exist natural structural jumps, where natural
means ``has been studied by non-logicians''. The following three
examples should seem familiar:
\begin{itemize}
\item If $(E,\eqvrel)$ is an equivalence relation, the expansion $(E,\eqvrel,\{P_{\geq n}\}_{n\in\N})$
from \thmref{JIeqv} is a structural jump of $(E,\eqvrel)$ \cites[Exercise II.45.(b)]{montalban_cst},
\item If $B$ is a Boolean algebra, the expansion $(B,\mathrm{isAtom})$
is a structural jump of $B$ \cite{strong_jump_inversion},
\item If $(L,\leq)$ is a linear ordering, the expansion $(L,\leq,\neg S)$
is a structural jump of $(L,\leq)$ \cite{strong_jump_inversion}.\footnote{The authors claim that ``{[}to form the structural jump,{]} the relation
$\mathit{succ}(x,y)$ is sufficient for linear orders'', but this
is a typo, in letter if not in spirit, because as mentioned in Footnote
\ref{fn:adjvsnegadj}, the successor relation is co-r.i.c.e. and not
r.i.c.e.}
\end{itemize}
There is an established analogue of the Friedberg Jump Inversion theorem
in this context, due to Soskov and Soskova \cite{soskova_soskov_jump_inversion},
and later independently rediscovered by Montalban \cite{montalban_notes_jump_structure}.
Two versions thereof, slightly different to each other, may be found
in \cites[Theorem IX.9]{montalban_cst} and \cites[Theorem 1.13]{strong_jump_inversion}.
We state only a particular case:
\begin{thm*}[Soskov--Soskova \cite{soskova_soskov_jump_inversion}, Montalban
\cite{montalban_notes_jump_structure}]
If there is a $0'$-computable copy of a structural jump of a structure,
then there is a low copy of the original structure.
\end{thm*}
If this is the theorem that we may refer to by the name of ``jump
inversion'', then it stands to reason that we might make the following
definition: A \emph{strong jump inversion theorem} is one of the type:
If there is a $0'$-computable copy of the structural jump of a {[}insert
class of structures here{]}, then there is a \uline{computable} copy
of the original structure.

We prefer to slightly weaken this ``definition'', so that we are
not beholden every time to verifying that certain relations make up
a structural jump of a structure, leading to a wider class of possible
results:
\begin{definition*}
A \emph{strong jump inversion theorem} is one of the type: If there
is a $0'$-computable copy of a {[}insert class of structures here{]},
say $\mathcal{A}$, together with {[}additional structure{]}, then
there is a computable copy of $\mathcal{A}$.
\end{definition*}
We finish this remark with the following fact, which further justifies
how we've chosen to weaken the notion of strong jump inversion theorem.
Suppose that we've shown that, for structures $\mathcal{A}$ in some
class, if there is a $0'$-computable copy of $\mathcal{A}$ together
with additional relations $(R_{i})_{i\in\N}$, then there is a computable
copy of $\mathcal{A}$. Then, under some relatively weak assumptions
on the relations $(R_{i})_{i\in\N}$ (e.g.\ that they are uniformly
r.i.c.e., or more generally that they're relatively intrinsically
computable from the jump of $\mathcal{A}$) we can actually conclude
that if there's a $0'$-computable copy of \emph{any} structural jump
of $\mathcal{A}$ then there is a computable copy of $\mathcal{A}$.
This is because $0'$ plus any structural jump of $\mathcal{A}$ suffices
to uniformly compute the relations $(R_{i})_{i\in\N}$.
\end{remark}
\medskip{}

We will now present our main result, \thmref{main_final}, but first
we must make some necessary definitions. There are two important components
to \thmref{main_final}: A computable-structure-theoretic definition
(\defref{cetypedpresentation}), and an almost entirely model-theoretic
condition (\defref{qetp}).

Henceforth, we restrict our attention to relational structures.

\begin{definition}
\label{def:atomicformula}For our purposes, an \emph{atomic formula}
in the variables $\vec{x}$ is a formula obtained by applying a single
predicate to some of the variables in the tuple $\vec{x}$ (possibly
repeated and reordered). Notably, under our definition, the atomic
formulas are \emph{not} closed under negation -- over the signature
$\Lang=\{\leq\}$ of a linear order, $x\leq y$ is considered an atomic
formula, but $\neg(x\leq y)$ is not.
\end{definition}
\begin{definition}
\label{def:atomictype}An \emph{atomic type} in the variables $\vec{x}$
is simply a set of atomic formulas in $\vec{x}$. For a tuple of elements
$\vec{b}$ in a structure, the \emph{atomic type of $\vec{b}$} consists
of the set of atomic formulas satisfied by $\vec{b}$.
\end{definition}
\begin{remark}
It is important to note that, while atomic types and quantifier-free
types contain the same data in a lot of contexts, the same shall not
be the case in our scenario. We will work with c.e.\ indices for
atomic types, from which it is not generally possible to obtain a
c.e.\ index for the corresponding quantifier-free type (which, indeed,
may not even be c.e.).
\end{remark}
\begin{definition*}[\ref{def:cetypedpresentation}]
Let $\Lang'$ be a computable relational signature.
\global\long\def\tp{\mathrm{tp}}%
 A\emph{ $0'$-computable c.e.-typed structure} $D$ consists of a
$0'$-computable function $t$, whose domain is the finite power-set
of an initial segment of $\N$ (this initial segment we call the \emph{domain}
of $D$), which takes as input the strong index of a finite set of
natural numbers $\vec{n}=\{n_{0}<n_{1}<\dots<n_{k-1}\}$ and outputs
the c.e.\ index $t(\vec{n})$ for an atomic $\Lang'$-type $\tp(\vec{n})$
in the variables $(x_{n_{0}},\dots,x_{n_{k-1}})$, with the compatibility
condition that if $\vec{n}\subseteq\vec{m}$ then $\tp(\vec{n})\subseteq\tp(\vec{m})$.
\end{definition*}
\begin{remark}
A $0'$-computable c.e.-typed structure $D$ induces a natural $\Lang'$
structure on the domain of $D$. Thus, in a slight abuse of notation
we often identify $D$ with the $\Lang'$-structure obtained in this
manner, allowing us to speak, for example, of an isomorphism between
a $0'$-computable c.e.-typed structure and other structures. In particular,
if $D_{0}$ is a pre-existing $\Lang'$-structure we define a \emph{$0'$-computable
c.e.-typed copy of $D_{0}$} to mean a $0'$-computable c.e.-typed
structure $D$ which is isomorphic to $D_{0}$.
\end{remark}
\begin{remark}
\defref{cetypedpresentation} is very much not symmetric with regard
to negation of predicates. It can make a stark difference as to whether
our language includes, for every predicate $P$, its negation $\neg P$
or not.
\end{remark}
In order to introduce our model-theoretic definition, we need to recall
some concepts from computable structure theory.

\begin{definition*}[\ref{def:comp_inf_lang}]
We shall work in the Computable Infinitary Language, as described
in \cites[Chapter III]{montalban_cst2}. We will only work with a
very small fragment of this language, which we now describe.

Given a relational signature $\Lang'$, we define the set of \emph{positive
finitary conjunctive formulas} as the set of formulas of the type
$\varphi(\vec{x})\equiv P_{1}\land\dots\land P_{n}$, where each $P_{j}$
denotes a relation in $\Lang'$ being applied to a tuple of variables
from $\vec{x}$, in any order, with the possibility of repeated or
omitted variables. We also allow for $P_{j}$ to denote $\top$, the
``true'' predicate, and $\bot$, the ``false'' predicate.

We define the set of\emph{ positive computable quantifier-free disjunctive
formulas}, denoted by $\llorc$, as the set of formulas in finitely
many variables of the type
\[
\varphi(\vec{x})\equiv\llor_{i<\omega}\varphi_{i}(\vec{x}),
\]
where $\{\varphi_{i}\}_{i<\omega}$ is a computable sequence of positive
finitary conjunctive formulas.

Finally, we define a third set of formulas, denoted by $\llorc_{2}$,
as those formulas in finitely many variables of the type
\[
\varphi(\vec{x})\equiv\llor_{i<\omega}\left(\psi_{i}(\vec{x})\land\neg\varphi_{i}(\vec{x})\right),
\]
where $\{\psi_{i}\}_{i<\omega}$ is a computable sequence of positive
finitary conjunctive formulas, and $\{\varphi_{i}\}_{i<\omega}$ is
a computable sequence of $\llorc$ formulas.

One can imagine iterating this definition to obtain a hierarchy of
quantifier-free formulas $\left\{ \llorc_{n}\right\} _{n\in\N}$ in
the computable infinitary language, but these first two levels will
be enough for our work.
\end{definition*}
Given a formula in the computable infinitary language, say $\varphi(\vec{x})$,
a structure $D$, and a tuple of elements $\vec{b}$ of $D$, we assign
a truth value to the expression $D\vDash\varphi(\vec{b})$ in the
obvious way. See \cites[Chapter III]{montalban_cst2} for more details.

We are now ready to make our model-theoretic definition. It is quite
technical, so we first provide some intuition and motivation.
\begin{remark}
\label{rmk:qetpmotivation}Many strong jump inversion type results
follow a pattern as follows. We are given a $0'$-computable structure.
Using ideas similar to Shoenfield's Limit Lemma, imagine this structure
as being given by a computable process that sometimes makes mistakes
(particularly, sometimes adds ``fake'' elements and later realizes
this error). Now, we construct a computable copy of our structure
as follows: Copy the limit-computable process, and when you realize
that an element was incorrectly added, repurpose it to serve as another
element that you somehow know exists -- oftentimes, the structure
you are working with admits a class of elements that can be reliably
used to ``dispose of trash''. For example, in \thmref{JIeqv}, any
element added by accident can be relegated to an infinite equivalence
class. This will be the role of the $\tau$ predicate defined in \defref{qetp}.
(The letter $\tau$ was chosen to stand for ``trash'', in the sense
of ``thing to be disposed of''.)

Sometimes, we do not immediately know what kind of trash a given element
will be relegated to. As an example, \propref{eqvfin} below regards
a certain type of equivalence relation with only finitely many infinite
equivalence relations. Then, if an element $x$ is meant to be disposed
of, we cannot relegate it to an infinite equivalence class -- instead,
we relegate it to a ``large enough'' equivalence class. If the set
of sizes of equivalence classes is c.e., this is easy enough, but
\propref{eqvfin} also encompasses some situations in which the set
of sizes of equivalence classes is not c.e.\ (or even $0'$-c.e.).
Thus, we provide relatively weak assumptions on the computability
of $\tau$.

There are a few more considerations. Some mistakes are too costly
to recover from. For example, back to \thmref{JIeqv}, we cannot tolerate
adding an element that breaks the assumption that $\eqvrel$ is an
equivalence relation -- e.g.\ an element $x$ with $x\eqvrel y$
and $x\not\eqvrel z$, where $y$ and $z$ are two previously-added
elements with $y\eqvrel z$. In most concrete cases, it is easy to
assume without loss of generality that such pathological cases never
show up, but this must be explicitly handled when creating a general-purpose
theorem. This will be the role of the $\chi$ predicate defined in
\defref{qetp}.

For technical reasons, we also introduce a third predicate, which
we call
\global\long\def\etau{\mathord{\varepsilon\tau}}%
 $\etau$ (where the $\varepsilon$ stands for ``exists''). The
motivation for why this predicate is necessary is a little too elaborate
to explain shortly in words, but very roughly, it has to do with the
fact that, if an element has been relegated to be trash by $\tau$,
but it has been relegated to an `incorrect' or `fake' trash element,
the feasibility of adding more elements (satisfying some prescription
$Q$) as measured by $\chi$ is unreliable. As such, we need a proxy
for $\chi(\vec{x},y)\equiv\exists_{\vec{z}}Q(\vec{x},y,\vec{z})$
that doesn't rely on $y$, which is thus the role of $\etau(\vec{x})\equiv\exists_{y}\exists_{\vec{z}}\left[\tau(\vec{x},y)\land Q(\vec{x},y,\vec{z})\right]$.
\end{remark}
\begin{definition*}[\ref{def:qetp}]
Let $\Lang\subseteq\Lang'$ be computable signatures, and let $D$
be a structure over $\Lang'$. We say that $D$ satisfies the \emph{Computable
Positive Quantifier Elimination and Trash Existence Property}, abbreviated
to \emph{QETP,} if the following properties hold:
\begin{enumerate}[label=(\alph*)]
\item $\Lang$ is finite, closed under negation, and includes the predicate
$\neq$,
\item For every quantifier-free $\Lang$ formula $q(\vec{x},\vec{y})$,
there is a $\llorc$\emph{ }$\Lang'$-formula $\chi_{q}(\vec{x})$,
computable from $q$, such that
\[
D\vDash\forall_{\vec{x}}\left[\chi_{q}(\vec{x})\leftrightarrow\exists_{\vec{y}}q(\vec{x},\vec{y})\right],
\]
\item There is a $0'$-computable \emph{partial} function denoted $\tau$,
which takes as input a pair $(q(\vec{x},y,\vec{z}),\varphi(\vec{x}))$,
where $q(\vec{x},y,\vec{z})$ is a quantifier-free $\Lang$-type and
$\varphi(\vec{x})$ is a \emph{finitary} quantifier-free $\Lang'$-formula,
whose output is a $\llorc_{2}$ formula denoted $\tau_{q\varphi}(\vec{x},y)$
that satisfies the following two properties:
\begin{itemize}
\item For every tuple $(\vec{x},y,\vec{z})$ of elements in $D$, of $\Lang$-type
$q(\vec{x},y,\vec{z})$, there is a formula $\varphi(\vec{x})$ satisfied
by $\vec{x}$ such that $(q,\varphi)$ is in the domain of $\tau$.
Equivalently, if $\vec{b}$ is a tuple of elements in $D$ such that
$D\vDash\chi_{q}(\vec{b})$, there is $\varphi(\vec{x})$ such that
$(q,\varphi)$ is in the domain of $\tau$ and $D\vDash\varphi(\vec{b})$.
Moreover,
\item The following formula holds in $D$:
\[
D\vDash\forall_{\vec{x}}\left[\chi_{q}(\vec{x})\land\varphi(\vec{x})\rightarrow\exists_{y}\tau_{q\varphi}(\vec{x},y)\right].
\]
\end{itemize}
\item Given a pair $(q(\vec{x},y,\vec{z}),\varphi(\vec{x}))$ in the domain
of $\tau$ as above, together with a quantifier-free $\Lang$-type
$Q(\vec{x},y,\vec{z},\vec{w})$ extending $q(\vec{x},y,\vec{z})$,
there is a $\llorc$ formula $\etau_{q\varphi Q}(\vec{x})$, computable
from the triple $(q,\varphi,Q)$, such that
\[
D\vDash\forall_{\vec{x}}\left[\exists_{y}\left[\chi_{Q}(\vec{x},y)\land\tau_{q\varphi}(\vec{x},y)\right]\leftrightarrow\etau_{q\varphi Q}(\vec{x})\right].
\]
\item \label{enu:qetp3-1}Finally, for all $q$, $\varphi$, $Q$, we require
\[
D\vDash\forall_{\vec{x},y}\left[\etau_{q\varphi Q}(\vec{x})\land\tau_{q\varphi}(\vec{x},y)\rightarrow\chi_{Q}(\vec{x},y)\right].
\]
\end{enumerate}
\end{definition*}
\medskip{}

We now are now ready to state the central result of this paper.
\begin{thm*}[\ref{thm:main_final}]
If $D$ is a $0'$-computable c.e.-typed structure over the signature
$\Lang'\supseteq\Lang$ satisfying the QETP, then the reduct $D\upto\Lang$
admits a computable copy $M$. There is a $0'$-computable isomorphism
between $D\upto\Lang$ and $M$. This result is uniform in the index
for $D$ and an index for witnesses that $D$ satisfies the QETP,\footnote{This includes: An index for which predicates in $\Lang$ are the negation
of each other, an index for the computable functions $\chi$ and $\etau$,
and an index for the $0'$-computable function $\tau$.} and relativizes uniformly.
\end{thm*}
As previously mentioned, we have been able to obtain pre-existing
results from the literature as corollaries of \thmref{main_final},
some examples of which are found in \secref{prevs_applications}.
We will now briefly describe another one of these applications as
a toy example, to give the reader a flavor of the type of argument
that we use.
\begin{definition}
If $T$ is a tree and $\sigma$ is a node, we use the term ``parent
of $\sigma$'' to mean the unique, if it exists, node that has $\sigma$
as a child. We use the term ``$a$-th parent of $\sigma$'' to mean,
if it exists, the parent of the parent of ... (repeated $a$ times)
of the parent of $\sigma$. In particular, the $0$-th parent of $\sigma$
is $\sigma$ itself.
\end{definition}
\begin{prop*}[{\ref{prop:trees}, {\cites[Proposition 3.8]{strong_jump_inversion}}}]
\global\long\def\leaf{\mathop{\mathrm{leaf}}}%
\global\long\def\Shat#1#2{S_{#1\hat{\phantom{i}}#2}}%
Let $T$ be an infinite tree in $\omega^{<\omega}$ (i.e.\ a nonempty
subset of $\omega^{<\omega}$ closed under prefixes) such that every
node of $T$ is either a leaf or admits infinitely many children,
of which finitely many are leaves. Define on $T$ the predicates:
\begin{itemize}
\item $\leaf(x)$, meaning $x$ has no children,
\item $\Shat ab(x,y)$, meaning that $x$ admits an $a$-th parent, and
$y$ admits a $b$-th parent, and that these two are the same,
\item $R_{a}(x)$, meaning that $x$ admits the root of the tree as an $a$-th
parent,
\end{itemize}
which have as particular cases the predicates $S(x,y)\equiv\Shat 01(x,y)$,
meaning that $y$ is a child of $x$, and $R_{0}(x)$, meaning that
$x$ is the root.

Then, if the structure $(T,\leaf,S,R_{0})$ admits a $0'$-computable
copy, the structure $(T,S,R_{0})$ admits a computable copy.

\end{prop*}
\begin{proof}
We will only briefly sketch the proof. A more detailed proof can be
found in Section \ref{subsec:trees}.

First, from a $0'$-computable copy of $(T,\leaf,S,R_{0})$ we uniformly
compute a $0'$-computable c.e.-typed copy of $(T,\neg\leaf,\{\Shat ab,\neg\Shat ab\}_{a,b\in\N},\{R_{a},\neg R_{a}\}_{a\in\N})$,
and thus it suffices to show that the structure $T$ satisfies the
QETP, over the signatures
\begin{itemize}
\item $\Lang=(S,\neg S,R_{0},\neg R_{0},=,\neq)$, and
\item $\Lang'$ equals $\Lang$, plus $\neg\leaf$ and all predicates $\Shat ab$,
$\neg\Shat ab$, $R_{a}$, and $\neg R_{a}$.
\end{itemize}
To this effect, we need to construct the functions $\chi$, $\tau$,
and $\etau$.
\begin{itemize}
\item To construct $\chi$: This is tantamount to proving a quantifier-elimination
result for this type of structure. We relegate the (standard for quantifier-elimination
type results) details to Section \ref{subsec:trees}.
\item To construct $\tau$: In our construction, the domain of $\tau$ is
the set of pairs $(q,\top)$. In other words. $\varphi(\vec{x})$
will always be $\top$ here.

Suppose that we are given the quantifier-free $\Lang$-type $q(\vec{x},y,\vec{z})$.
Our goal in defining $\tau$ is to define what it means for an element
$y$ to be ``trash'' or ``generic'' relative to the remaining
data. In this, there are two cases:
\begin{itemize}
\item If $q$ already prescribes that $y$ is related to some $x_{i}$ (e.g.\ it
includes the data $S(x_{1},y)$, or $S(x_{1},z_{1})\land S(z_{1},y)$),
the position of $y$ in the tree relative to $\vec{x}$ is `forced'.
In this case, we demand that $y$ is not a leaf. Thus, in this case,
$\tau_{q\top}(\vec{x},y)\equiv\chi_{q}(\vec{x},y)\land\neg\leaf(y)$.
\item If $q$ does not prescribe the position of $y$ relative to any $x_{i}$,
we would like to set $y$ to be an arbitrary child of the root. This
will not necessarily work, because perhaps $q$ includes enough data
to say that $y$ has some number of direct predecessors. Thus, we
let $N$ be the least number of predecessors that $q$ guarantees
$y$ has, and set $\tau_{q\top}(\vec{x},y)$ to mean ``$y$ is not
a leaf, has exactly $N$ predecessors, and no common ancestors with
any $x_{i}$ other than the root''. Note that we don't have any predicate
$Q(x_{i},y)$ representing that two elements only have the root as
common ancestor, but this can be expressed as a $\llorc$ sentence
in $\Lang'$ as: for some $a,b\in\N$, $x_{i}$ has depth $a$, $y$
has depth $b$, and $\neg\Shat{a-1}{b-1}(x,y)$.
\end{itemize}
\item To construct $\etau$: Given a quantifier-free $\Lang$-type $q(\vec{x},y,\vec{z})$
and another, $Q(\vec{x},y,\vec{z},\vec{w})$, extending it, we wish
to determine, in terms of $\vec{x}$, whether there exists $y$ satisfying
$\chi_{Q}(\vec{x},y)$ and $\tau_{q\top}(\vec{x},y)$. This consists
of (metatheoretically) inspecting the demands that are made of $y$
by $\chi_{Q}$ and by $\tau_{q\top}$ and verifying that none of them
interfere with each other. If none do, set $\etau_{q\top Q}(\vec{x})\equiv\chi_{Q}(\vec{x})$;
otherwise, set $\etau_{q\top Q}(\vec{x})\equiv\bot$.
\end{itemize}
This concludes the construction. A tedious model-theoretic argument
(which, of course, requires one to have done the details) will show
that these formulas satisfy all the conditions of the QETP, and thus
\thmref{main_final} applies, whence we obtain a computable copy of
the structure $(T,S,R_{0})$. In lieu of the details, we provide an
explanation for why we might expect these conditions to hold.

Properties \ref{enu:qetp0}, \ref{enu:qetp1}, and \ref{enu:qetpetau}
all hold by deliberate construction. Property \ref{enu:qetp2} requires,
in this context, that if $\vec{x}$ satisfies $\chi_{q}(\vec{x})$
then there exists $y$ such that $\tau_{q\top}(\vec{x},y)$, and indeed
our $\tau$ only makes ``generic'' demands, such as ``$y$ is not
a leaf'' or ``$y$ has no common ancestor with $\vec{x}$ except
for the root'', which can always be satisfied by infinitely many
elements. Finally, Property \ref{enu:qetp3} is saying, in words:
Given $\vec{x}$, if there is some $y$ that satisfies $\chi_{Q}(\vec{x},y)\land\tau_{q\top}(\vec{x},y)$,
then any $y$ that satisfies $\tau_{q\top}(\vec{x},y)$ will do so.
This holds, because the only thing that might distinguish any two
such $y$ is the pattern of leaves succeeding them (e.g.\ perhaps
one such $y$ has two leaves as a successor, while another $y$ has
infinitely many), but $\chi_{Q}(\vec{x},y)$ cannot demand anything
on the number of leaves succeeding $y$ (or leaves succeeding nodes
of $y$, etc.) -- although it \emph{can }demand that $y$ is not
a leaf, which is why we added $\neg\leaf(y)$ to the definition of
$\tau$ -- and thus any two $y$ that satisfy $\tau_{q\top}(\vec{x},y)$
will be indistinguishable from the perspective of $\chi_{Q}(\vec{x},y)$.
\end{proof}

\begin{remark}
We would like to highlight two points about \thmref{main_final}.
First, we note that it reverses the usual way in which jump inversion
constructions go: Generally, one starts with the $0'$-computable
structure, builds a computable structure out of it, and proves that
this computable structure is isomorphic to the original. On the other
hand, applications of \thmref{main_final} start by constructing a
$0'$-computable c.e.-typed structure out of it, proving that this
structure is isomorphic to the original, and only afterwards constructing
the computable copy.

Another interesting thing about \thmref{main_final} is that, in all
applications we've found, it can be used to black-box all mentions
of injury from a proof. This does not necessarily make the proof \emph{easier};
the verifications necessary to apply \thmref{main_final} are often
more burdensome, if not necessarily more difficult, than a standard
injury proof would be. Regardless, we believe that \thmref{main_final}
encapsulates a common theme among many different finite injury proofs
of jump inversion theorems.
\end{remark}

\section*{Acknowledgements}

I would like to thank Denis Hirschfeldt, Miles Kretschmer, and Patrick
Lutz for their valuable feedback, help, and advice, without which
this work would not have been possible.

\section{A General Jump Inversion Theorem}

As explained in the introduction, there are two main concepts that
need to be introduced before we may state and prove the main theorem
of this paper, \thmref{main_final}. We will now dedicate some time
to defining, motivating, and explaining these concepts.

\subsection{\protect\label{subsec:cetypedstructures}C.e.-typed Structures}

In computable structure theory, the most standard notion of what is
meant by a computable structure over the finite/computable signature
$\Lang$ is: A computable set $A$, with an interpretation for the
symbols from $\Lang$, such that the atomic diagram $D(A)=\{$quantifier-free
formulas with parameters from $A$ that hold true in $A$$\}$ is
computable. This admits an obvious relativization to an arbitrary
oracle $X$, which induces the standard meaning of ``$X$-computable
structure''. However, we've found this notion to be poorly-behaved
in some respects, especially when it comes to the interplay between
infinite signatures (or function symbols) and limit-computability.

To be more precise, let us consider the case where the signature $\Lang$
is finite and relational, let $A$ be a $0'$-computable structure.
Then, given a tuple of elements $\vec{a}$ from $A$, we have, at
each instant in time, a guess for what the quantifier-free $\Lang$-type
of $\vec{a}$ is, and this guess will change a finite number of times,
because the guess for each individual relation (of which there are
finitely many) will change a finite number of times. However, if the
signature $\Lang$ is finite, the matter becomes more difficult: We
can no longer be sure that, eventually, we'll have the right guess
for the type of $\vec{a}$, only that \emph{for any particular predicate},
eventually our guess for that predicate will be correct. This can
pose an issue when trying to perform a finite-injury construction:
Suppose, for example, that we wish to perform task $A$ if some element
$x$ satisfies one of countably many predicates $P_{0}(x)$, $P_{1}(x)$,
etc., and we wish to perform task $B$ instead if it does not. Then,
in a finite injury construction we would perform task $B$ until (unless)
we saw $x$ satisfies one of the predicates $P_{i}$, in which case
we would undo task $B$ and perform task $A$ instead. However, if
all we have are \emph{guesses} for $P_{i}$, we can no longer guarantee
finite injury: It could happen that $x$ satisfies none of the $P_{i}$,
but that at arbitrarily large times our guess for some $P_{N}$ is
that $P_{N}(x)$ holds. We could only guarantee finite injury if there
was a time past which our guesses for \emph{all} the $P_{i}$ remain
the same. This, we claim, makes $0'$-computable structures over infinite
signatures poorly-behaved, and motivates the definition of what we
might call \emph{$0'$-computable computably typed structure}, as
one for which there is a $0'$-computable function which, given a
tuple of elements, outputs a \emph{computable index} for their quantifier-free
$\Lang$-type.

This already motivates us half-way to the main definition of this
section, but we generalize a bit further. There are some scenarios
in which demanding a computable index is far too much: Either because
the types of the elements of a structure may be too computably complex
to represent as a computable index (see e.g.\ \thmref{khisamiev},
which regards abelian groups with divisibility predicates -- even
in a computable group, the set of integers that divide a given element
may be complex enough to encode $0'$), or even if the types are themselves
simple, $0'$ may be unable to determine a computable type in finite
time (see e.g.\ \thmref{JIlinear-1}, which regards linear orders
with ``$n$-th successor of'' predicates -- if $0'$ could compute
a computable index for the set of $n$ such that $x$ and $y$ are
$n$-th successors, then it could also determine whether $x$ and
$y$ are a finite distance apart (this is called the block relation),
but one can show that $0''$ could be encoded into the block relation
of a linear order, so $0'$ must be unable to compute it in general).
However, we've found in both of these cases that a c.e.\ index provides
a workable middle ground, both in the sense that it's a reasonable
demand to make of our structures, and in the sense that it's sufficient
to prove results of interest.

Of course, the c.e.\ index for the quantifier-free type of a tuple
yields exactly the same information as a computable index -- simply
wait until either $\varphi(\vec{x})$ or $\neg\varphi(\vec{x})$ is
enumerated into the type to determine which holds. Thus, in order
to make any gain from demanding c.e.\ instead of computable, we need
to take care around negations.

\begin{definition}
\label{def:cetypedpresentation}Let $\Lang'$ be a computable relational
signature.
\global\long\def\tp{\mathrm{tp}}%
 A\emph{ $0'$-computable c.e.-typed structure} $D$ consists of a
$0'$-computable function $t$, whose domain is the finite power-set
of an initial segment of $\N$ (this initial segment we call the \emph{domain}
of $D$), which takes as input the strong index of a finite set of
natural numbers $\vec{n}=\{n_{0}<n_{1}<\dots<n_{k-1}\}$ and outputs
the c.e.\ index $t(\vec{n})$ for an atomic\footnote{See Definitions \ref{def:atomicformula} and \ref{def:atomictype}
in the introduction for the distinction between atomic type and quantifier-free
type, which is important in this context.} $\Lang'$-type $\tp(\vec{n})$ in the variables $(x_{n_{0}},\dots,x_{n_{k-1}})$,
with the compatibility condition that if $\vec{n}\subseteq\vec{m}$
then $\tp(\vec{n})\subseteq\tp(\vec{m})$.
\end{definition}
\begin{remark}
As a trivial but noteworthy edge case in \defref{cetypedpresentation},
by considering the empty set, we also have access to a c.e.\ index
for the atomic $\Lang'$-type of the empty tuple, which corresponds
to an enumeration of the true zero-ary $\Lang'$-predicates in $D$.
\end{remark}
\begin{remark}
\defref{cetypedpresentation} is not symmetric with regard to negation
of predicates. It can make a stark difference whether our language
includes, for every predicate $P$, its negation $\neg P$ or not.
\end{remark}

\subsection{\protect\label{subsec:qetp}The QETP}

\begin{definition}
\label{def:comp_inf_lang}We shall work in the Computable Infinitary
Language, as described in \cites[Chapter III]{montalban_cst2}. We
will only work with a very small fragment of this language, which
we now describe.

Given a relational signature $\Lang'$, we define the set of \emph{positive
finitary conjunctive formulas} as the set of formulas of the type
$\varphi(\vec{x})\equiv P_{1}\land\dots\land P_{n}$, where each $P_{j}$
denotes a relation in $\Lang'$ being applied to a tuple of variables
from $\vec{x}$, in any order, with the possibility of repeated or
omitted variables. We also allow for $P_{j}$ to denote $\top$, the
``true'' predicate, and $\bot$, the ``false'' predicate.

We define the set of\emph{ positive computable quantifier-free disjunctive
formulas}, denoted by $\llorc$, as the set of formulas in finitely
many variables of the type
\[
\varphi(\vec{x})\equiv\llor_{i<\omega}\varphi_{i}(\vec{x}),
\]
where $\{\varphi_{i}\}_{i<\omega}$ is a computable sequence of positive
finitary conjunctive formulas.

Finally, we define a third set of formulas, denoted by $\llorc_{2}$,
as those formulas in finitely many variables of the type
\[
\varphi(\vec{x})\equiv\llor_{i<\omega}\left(\psi_{i}(\vec{x})\land\neg\varphi_{i}(\vec{x})\right),
\]
where $\{\psi_{i}\}_{i<\omega}$ and $\{\varphi_{i}\}_{i<\omega}$
are a computable sequences of $\llorc$ formulas.
\end{definition}
\begin{remark}
\label{rmk:llorc}The main properties of the $\llorc$ formulas are
the following:
\begin{itemize}
\item There is a canonical way to associate indices to $\llorc$ formulas,
so that a computer program may manipulate them, and
\item Given a c.e.\ index for the atomic $\Lang'$-type of a tuple of elements
$\vec{b}$ in some structure, and a $\llorc$ formula $\varphi(\vec{x})$
satisfied by $\vec{b}$, we can computably confirm in finite time
that the tuple $\vec{b}$ satisfies the formula $\varphi$.
\end{itemize}
\end{remark}
\begin{remark}
\label{rmk:llorc2}The main properties of the $\llorc_{2}$ formulas
are the following:
\begin{itemize}
\item There is a canonical way to associate indices to $\llorc_{2}$ formulas,
\item There is a canonical way to turn the index for a $\llorc$ formula
into the index for a logically equivalent $\llorc_{2}$ formula, and
\item Given a tuple of elements of a structure and a c.e.\ index for their
type, it is $0'$-c.e.\ to determine whether a $\llorc_{2}$ formula
holds of the elements.
\end{itemize}
Let us briefly elaborate on that last bullet point. If $\varphi(\vec{x})\equiv\llor_{i}\left(\psi_{i}(\vec{x})\land\neg\varphi_{i}(\vec{x})\right)$
is a $\llorc_{2}$ formula, and $\vec{b}$ is a tuple of elements
of a structure for which we have a c.e.\ index for the atomic type
of $\vec{b}$, we can computably create a sequence of guesses for
whether $\varphi(\vec{b})$ holds: First, continuously guess `no',
whilst dovetailing over the values of $i$ until we find a value for
which $\psi_{i}(\vec{b})$ holds. If such a value is found, continuously
guess `yes' until and unless you discover that $\varphi_{i}(\vec{b})$
holds. At that point, guess `no' once, and go back to finding a
(new) value $j$ for which $\psi_{j}(\vec{b})$ holds, and so on.
This sequence of guesses will have the following behavior:
\begin{itemize}
\item If $\varphi(\vec{b})$ holds, we will guess `yes' all but finitely
many times, and
\item If $\varphi(\vec{b})$ does not hold, we will guess `no' infinitely
many times.
\end{itemize}
\end{remark}
\begin{remark}
Compare the properties in \rmkref{llorc} to the following properties
of a related class of computable infinitary formulas. The $\Sigma_{1}^{c}$
formulas are those formulas in the Computable Infinitary Language
of the form:
\[
\llor_{i<\omega}\exists_{\vec{y}_{i}}\varphi_{i}(\vec{x},\vec{y}_{i}),
\]
where the sequence of formulas $\{\exists_{\vec{y}_{i}}\varphi_{i}(\vec{x},\vec{y})\}_{i\in\N}$
is computable, and each $\varphi_{i}$ is a finite Boolean combination
of predicates applied to variables in the tuple $(\vec{x},\vec{y})$.
\begin{itemize}
\item The $\Sigma_{1}^{c}$ formulas also admit a canonical way to associate
indices, and
\item Given a tuple $\vec{b}$ of elements in a computable structure $D$
that satisfy the $\Sigma_{1}^{c}$ formula $\varphi(\vec{x})$, we
can confirm in finite time that the tuple $\vec{b}$ satisfies the
formula $\varphi$.
\end{itemize}
Despite the similarities, the $\llorc$ formulas shall be more useful
in our treatment, because to check the truth of a $\Sigma_{1}^{c}$
formula, we need access to the entire model (so that we may brute-force
the tuples $\vec{y}_{i}$), whilst we shall be working with $0'$-computable
c.e.-typed structures, which may contain ``phantom elements'' that
could otherwise lead to false positives. Moreover, since we only have
a c.e.\ index for the positive information about elements, we cannot
computably check arbitrary finite Boolean combinations of predicates
-- hence why we restrict to positive formulas.
\end{remark}
\begin{remark}
In the sequence, we shall have two nested signatures, $\Lang\subseteq\Lang'$.
We shall follow the following conventions:
\begin{itemize}
\item Unless indicated otherwise, all formulas are assumed to be quantifier-free,
\item All named $\Lang$-formulas will be quantifier-free finitary first-order
formulas, and will be denoted by lower-case Roman letters,
\item All named $\Lang'$-formulas will be $\llorc$ and will be denoted
by lower-case Greek letters, except for the formulas named $\tau$,
which will be $\llorc_{2}$. Conversely, all $\llorc$ or $\llorc_{2}$
formulas in the sequence are assumed to be over the signature $\Lang'$.
\end{itemize}
\end{remark}
We now present the main definition of this section. It is a rather
technical definition. We recall that we provided some motivation for
this definition in the introduction; see \rmkref{qetpmotivation}.

\begin{definition}[QETP]
\label{def:qetp}Let $\Lang\subseteq\Lang'$ be computable signatures,
and let $D$ be a structure over $\Lang'$. We say that $D$ satisfies
the \emph{Computable Positive Quantifier Elimination and Trash Existence
Property}, abbreviated to \emph{QETP,} if the following properties
hold:
\begin{enumerate}[label=(\alph*)]
\item \label{enu:qetp0}$\Lang$ is finite, closed under negation, and
includes the predicate $\neq$,
\item \label{enu:qetp1}For every quantifier-free $\Lang$-formula $q(\vec{x},\vec{y})$,
there is a $\llorc$\emph{ }$\Lang'$-formula $\chi_{q}(\vec{x})$,
computable from $q$, such that
\[
D\vDash\forall_{\vec{x}}\left[\chi_{q}(\vec{x})\leftrightarrow\exists_{\vec{y}}q(\vec{x},\vec{y})\right],
\]
\item \label{enu:qetp2}There is a $0'$-computable \emph{partial} function
denoted $\tau$, which takes as input a pair $(q(\vec{x},y,\vec{z}),\varphi(\vec{x}))$,
where $q(\vec{x},y,\vec{z})$ is a quantifier-free $\Lang$-type\footnote{Note the lack of ``vector arrow'' on the variable $y$. This denotes
a \emph{single} variable, not a tuple of them.} and $\varphi(\vec{x})$ is a \emph{finitary} $\Lang'$-formula, whose
output is a $\llorc_{2}$ $\Lang'$-formula denoted $\tau_{q\varphi}(\vec{x},y)$
that satisfies the following two properties:
\begin{itemize}
\item For every tuple $(\vec{x},y,\vec{z})$ of elements in $D$, of $\Lang$-type
$q(\vec{x},y,\vec{z})$, there is a formula $\varphi(\vec{x})$ satisfied
by $\vec{x}$ such that $(q,\varphi)$ is in the domain of $\tau$.
Equivalently, if $\vec{b}$ is a tuple of elements in $D$ such that
$D\vDash\chi_{q}(\vec{b})$, there is $\varphi(\vec{x})$ such that
$(q,\varphi)$ is in the domain of $\tau$ and $D\vDash\varphi(\vec{b})$.
Moreover,
\item The following formula holds in $D$:
\[
D\vDash\forall_{\vec{x}}\left[\chi_{q}(\vec{x})\land\varphi(\vec{x})\rightarrow\exists_{y}\tau_{q\varphi}(\vec{x},y)\right].
\]
\end{itemize}
\item \label{enu:qetpetau}Given a pair $(q(\vec{x},y,\vec{z}),\varphi(\vec{x}))$
in the domain of $\tau$ as above, together with a quantifier-free
$\Lang$-type $Q(\vec{x},y,\vec{z},\vec{w})$ extending $q(\vec{x},y,\vec{z})$,
there is a $\llorc$ $\Lang'$-formula $\etau_{q\varphi Q}(\vec{x})$,
computable from the triple $(q,\varphi,Q)$, such that
\[
D\vDash\forall_{\vec{x}}\left[\exists_{y}\left[\chi_{Q}(\vec{x},y)\land\tau_{q\varphi}(\vec{x},y)\right]\leftrightarrow\etau_{q\varphi Q}(\vec{x})\right].
\]
The function $(q,\varphi,Q)\mapsto\etau_{q\varphi Q}$ may be defined
in some cases even if the pair $(q,\varphi)$ is not in the domain
of $\tau$. In this scenario, we do not make any requirement of $\etau_{q\varphi Q}$.

Note: As a corollary of \ref{enu:qetp2}, we obtain $D\vDash\forall_{\vec{x}}\left[\chi_{q}(\vec{x})\land\varphi(\vec{x})\rightarrow\etau_{q\varphi q}(\vec{x})\right]$
for $(q,\varphi)$ in the domain of $\tau$.
\item \label{enu:qetp3}Finally, for all $q$, $\varphi$, $Q$, we require
\[
D\vDash\forall_{\vec{x},y}\left[\etau_{q\varphi Q}(\vec{x})\land\tau_{q\varphi}(\vec{x},y)\rightarrow\chi_{Q}(\vec{x},y)\right].
\]
\end{enumerate}
\end{definition}
\begin{remark}
\label{rmk:tauwork}We will work with Item \ref{enu:qetp2} of the
QETP in the following manner: We envision a computable process that
enumerates triplets $(q,\varphi,\tau)$, and sometimes deletes them,
in such a manner that, if $\tau_{q\varphi}$ is well-defined, the
triple $(q,\varphi,\tau_{q\varphi})$ will eventually be enumerated
and never removed, and every triplet that is not of the form $(q,\varphi,\tau_{q\varphi})$
for $(q,\varphi)$ in the domain of $\tau$ will eventually be removed.
\end{remark}
\begin{remark}
As a consequence of Item \ref{enu:qetp1} of Definition \ref{def:qetp},
if $D$ is a $0'$-computable c.e.-typed structure satisfying the
QETP, we can enumerate the quantifier-free $\Lang$-types that are
realized in $D$. This is because, to tell if a type $q(\vec{y})$
is realized in $D$, we inspect the truth value of the zero-ary positive
sentence $\chi_{q}$, which will consist of a positive combination
of zero-ary predicates of $D$. By inspecting the type of the zero-ary
tuple of elements of $D$ we can therefore enumerate the types $q(\vec{y})$
that are realized.

\label{rmk:weakerchi}We make an additional technical assumption on
the formulas $\chi_{q}$, albeit with no loss of generality. Let $q(\vec{x},\vec{y},\vec{z})$
be a quantifier-free $\Lang$ formula. We can assign a $\chi$-formula
to $q$ in (at least) two different ways: By grouping the variables
as $(\vec{x},[\vec{y},\vec{z}])$, or by grouping them as $([\vec{x},\vec{y}],\vec{z})$.
This leads to two different formulas, which, when necessary to distinguish,
we notate respectively as $\chi_{q(\vec{x},\cdot,\cdot)}(\vec{x})$
and $\chi_{q(\vec{x},\vec{y},\cdot)}(\vec{x},\vec{y})$. Now, \textbf{we
assume without loss of generality} that, in such a scenario, $\chi_{q(\vec{x},\vec{y},\cdot)}(\vec{x},\vec{y})$
contains $\chi_{q(\vec{x},\cdot,\cdot)}(\vec{x})$ as a subformula.

The reason why we can make such an assumption is that, given a quantifier-free
$\Lang$-formula $q(\vec{x},\vec{y})$, we may redefine $\bar{\chi}_{q}(\vec{x})$
as
\[
\bar{\chi}_{q}(\vec{x})=\bigwedge_{\vec{x}=(\vec{x}_{1},\vec{x}_{2})}\chi_{q(\vec{x}_{1},\cdot,\cdot)}(\vec{x}_{1}),
\]
where the conjunction ranges over all decompositions of the tuple
$\vec{x}$ into two tuples of variables (not necessarily in order).
Implicitly, we use the fact that finite conjunctions of $\llorc$
formulas are themselves $\llorc$. This new definition of $\bar{\chi}$
still serves as a witness to the QETP, if the previous $\chi$ did,
but has this additional technical benefit that will be necessary in
our proof.
\end{remark}
\begin{remark}
\label{rmk:assumeq}When verifying the QETP for a specific structure,
we must construct a formula $\chi_{q}(\vec{x})$ for any quantifier-free
$\Lang$-formula $q(\vec{x},\vec{y})$. This can be quite burdensome.
Thus, we present a simplification step, common in quantifier-elimination
proofs, which we will henceforth use without comment: When constructing
$\chi_{q}$ we may assume, without loss of generality, that $q(\vec{x},\vec{y})$
is a quantifier-free $\Lang$-type all of whose elements are distinct.
This is because any formula $q(\vec{x},\vec{y})$ can be written as
a disjunction $\lor_{i}\left(q(\vec{x},\vec{y})\land t_{i}(\vec{x},\vec{y})\right)$,
where $t_{i}$ varies over the quantifier-free $\Lang$-types, each
$q(\vec{x},\vec{y})\land t_{i}(\vec{x},\vec{y})$ is logically equivalent
to either $\bot$ or to $t_{i}(\vec{x},\vec{y})$ (call the result
$q_{i}(\vec{x},\vec{y})$), and moreover $\exists_{\vec{y}}\lor_{i}q_{i}(\vec{x},\vec{y})$
is logically equivalent to $\lor_{i}\exists_{\vec{y}}q_{i}(\vec{x},\vec{y})$,
so it suffices to define $\chi$ on each $q_{i}$. Moreover, we may
assume that all variables are distinct, because if $q_{i}$ contains
$y_{a}=y_{b}$ or $y_{a}=x_{b}$ we can simply perform the substitution
and erase the quantification in $y_{a}$, and if $q_{i}$ contains
$x_{a}=x_{b}$, we can get rid of one of these two variables by using
the fact that $\exists_{\vec{y}}q_{i}(\vec{x},\vec{y})$ is logically
equivalent (in this case) to $x_{a}=x_{b}\land\exists_{\vec{y}}\hat{q}_{i}(\vec{x}_{\hat{a}},\vec{y})$,
where $\hat{q}_{i}$ consists of the reduct of $q_{i}$ to the tuple
of variables from which we removed $x_{a}$.
\end{remark}
\begin{remark}
In truth, \rmkref{assumeq} is actually completely unnecessary: By
inspection of the proof of \thmref{main_final}, one will find that
indeed $\chi_{q}$ is only ever calculated when $q$ is a quantifier-free
$\Lang$-type of distinct elements. Nevertheless, we preferred to
phrase the QETP as it is, even if it is stronger than necessary in
appearance (though not in fact), to emphasize the fact that the $\chi$
part of the definition is no more than an effective quantifier-elimination
result.
\end{remark}

\subsection{Main Result}
\begin{thm}
\label{thm:main_final}If $D$ is a $0'$-computable c.e.-typed structure
over the signature $\Lang'\supseteq\Lang$ satisfying the QETP, then
the reduct $D\upto\Lang$ admits a computable copy $M$. There is
a $0'$-computable isomorphism between $D\upto\Lang$ and $M$. This
result is uniform in the index for $D$ and an index for witnesses
that $D$ satisfies the QETP,\footnote{This includes: An index for which predicates in $\Lang$ are the negation
of each other, an index for the computable functions $\chi$ and $\etau$,
and an index for the $0'$-computable function $\tau$.} and relativizes uniformly.
\end{thm}
\begin{proof}
We shall build a computable copy of $D\upto\Lang$ via a finite injury
argument. First, we lay down the foundation for the functioning of
our algorithm:
\begin{itemize}
\item We visualize the structure $D$ as being given by a computable process,
outputting elements and c.e.\ indices for the atomic $\Lang'$-type
of the elements output thus far, and sometimes erasing previously
constructed elements. We shall be executing this construction in parallel
in the background of ours.
\item We shall build the $\Lang$-structure $M\cong D\upto\Lang$ via a
computable process that outputs a sequence of elements, and with each
element, the quantifier-free $\Lang$-type of the elements output
thus far.
\item We well-order the elements of $M$ and the elements of $D$ by time
of addition. This includes ``fake'' elements of $D$. When we refer
to ``the first element of $M$ (or $D$) that satisfies such-and-such
property'', we mean the one that was added earliest.
\item We shall imagine an assortment of workers, referred to as ``worker
$i$'' for $i=0,1,2,\dots$, each of which will be responsible for
one step in a back-and-forth construction of an isomorphism between
$D$ and the structure $M$ that we are constructing.
\item We assume that only one worker is active at each time. When worker
$i$ is done with its task, it will activate worker $i+1$. When an
injury occurs (the means by which this may be caused will be discussed
later), we may divert execution back to a prior worker, say worker
$j$, and delete the state of all workers past worker $j$.
\item The internal state of this construction, that is, the information
we shall be keeping track of, consists of the following data (plus
whatever ancillary information is necessary to keep track of the worker
mechanism):
\begin{itemize}
\item The current status of the structures $M$ and $D$,
\item One-to-one matchings between some elements of $M$ and some elements
of $D$, the latter of which may or may not be real,
\item To some elements $m$ of $M$, we match a triplet of formulas $(q(\vec{x},y,\vec{z}),\varphi(\vec{x}),\tau(\vec{x},y))$
that have been output by the process from \rmkref{tauwork}, together
with a decomposition of the elements of $M$ at the time that this
matching was created, into $\vec{x}$, $y=m$, and $\vec{z}$.
\item To each matching of the previous two kinds, we keep track of which
worker created it and when.
\end{itemize}
\item The internal state of this construction may be modified via the following
operations:
\begin{itemize}
\item We may add a new element to $M$, in which case we immediately match
it to an element of $D$,
\item We may create a matching between an element of $M$ and an element
of $D$,
\item We may match a triplet of formulas $(q,\varphi,\tau)$ to an element
$y\in M$, together with the decomposition $M_{\mathrm{current}}=(\vec{x},y,\vec{z})$,
where $\vec{x}$ denotes the elements of $M$ that have a match in
$D$ at this moment (note that in particular we assume $y$ does not
have a match at this moment),
\item In the event of injury, we may remove matches of the two kinds above.
\end{itemize}
\item The construction will be injured if either of the following three
events occur:
\begin{itemize}
\item An element of $D$ which admits a match to an element of $M$ is deleted,
\item There is an element of $M$, with an assigned triplet of formulas
$(q,\varphi,\tau)$, for which the triplet of formulas has been deleted
from the enumeration of the process of \rmkref{tauwork},
\item There is an element of $M$, with an assigned triplet $(q,\varphi,\tau)$
and a matched element $d\in D$, such that our guess about whether
$\tau$ holds of this element has changed to `no'.
\end{itemize}
In either of these cases, all matches created after the ``faulty
match'' will be deleted.
\end{itemize}
Now, we establish some assumptions about the internal state of the
construction, which we will be careful to uphold at every step. Together
with each assumption, we point out the operations that we need to
be careful in exercising -- the remaining operations will always
preserve the given assumption. Moreover, the assumptions will not
be broken with the passage of time unless injury happens, in which
case the injury-handling part of the algorithm will rectify any issues.
\begin{enumerate}[label={Assumption (\Alph*)}, leftmargin=*]
\item \label{enu:asmchi}If $M=(\vec{x},\vec{y})$, where $\vec{x}$ denotes
the elements of $M$ that currently admit matches in $D$, $\vec{b}$
denotes these matches, and $q(\vec{x},\vec{y})$ is the type of the
elements of $M$, we demand that $D\vDash\chi_{q}(\vec{b})$.

We need to be careful to ensure that this is true when we add new
elements to $M$, and when we create matches between elements of $M$
and elements of $D$. We would also need to be careful in the event
that matches are deleted, but this is where \rmkref{weakerchi} becomes
relevant, as it ensures that this is not a concern.
\item \label{enu:asmphi}If $y\in M$ is matched with a triplet of formulas
$(q,\varphi,\tau)$ and decomposition $(\vec{x},y,\vec{z})\subseteq M$,
we demand that $M\vDash q(\vec{x},y,\vec{z})$ and that all elements
of $\vec{x}$ admit matches $\vec{b}$ in $D$, which satisfy $D\vDash\varphi(\vec{b})$.

We need to be careful to ensure that this is true when we assign a
triplet to an element. We don't need to be careful when removing matches
or elements of $D$, because our injury process will, in that case,
unmatch the formulas assigned to $y$.
\item \label{enu:asmutau}At most one unmatched element $m$ of $M$ will
have an assigned triplet at any given time, and in this event, the
variables $\vec{x}$ correspond to the elements of $M$ currently
matched to elements of $D$, and the next action to be taken shall
be to match $m$ to an element of $D$.
\item \label{enu:asmtau}If $y$ is an element of $M$ with assigned triplet
$(q(\vec{x},y,\vec{z}),\varphi(\vec{x}),\tau(\vec{x},y))$ and $y$
admits a match $d$ in $D$, if $\vec{b}$ is the tuple of elements
matched with $\vec{x}$, it must be the case that $D\vDash\tau(\vec{b},d)$
\emph{at the current approximation} (in the sense of \rmkref{llorc2}).

We need to be careful to ensure that this is true when we create a
match between an element of $M$ with an assigned triplet $(q,\varphi,\tau)$
and an element of $D$. We will not need to be careful when we assign
a triplet $(q,\varphi,\tau)$ to an element of $M$, because (it turns
out) we will always match such triplets to unmatched elements. We
also need to be careful if the current approximation (in the sense
of \rmkref{llorc2}) ever changes, which is why `the approximation
of the value of $\tau(\vec{b},d)$ changes' is a cause of injury.
\item \label{enu:asmetau}If $m$ is an element with assigned formulas $(q(\vec{x},y,\vec{z}),\varphi(\vec{x}),\tau(\vec{x},y))$,
with $\vec{b}$ denoting the matches of the $\vec{x}$, and at this
stage $M$ is decomposed as $M=(\vec{x},y,\vec{z},\vec{w})$ with
quantifier-free $\Lang$-type $Q(\vec{x},y,\vec{z},\vec{w})$, it
must be the case that, if the $\vec{b}$ are real elements of $D$,
we have $D\vDash\etau_{q\varphi Q}(\vec{b})$.

We need to be careful to ensure that this is true when new elements
are added to $M$, and when we assign triplets $(q,\varphi,\tau)$
to elements of $M$.
\end{enumerate}
The process starts with the initialization of worker $0$.

Now, let us describe the algorithm that will be followed by worker
$i$. Upon initialization:
\begin{itemize}
\item To begin, find the first element of $M$, if any, that does not have
a match in $D$. If there is none, skip to the end of this bullet
point. If there is one, call it $m$.

If $m$ does not have an assigned formula: Let $(\vec{x},y,\vec{z})$
be the current elements of $M$, with $\vec{x}$ corresponding to
the elements that currently have a match in $D$ -- call them $\vec{b}$
--, $y$ corresponding to $m$, and $\vec{z}$ corresponding to the
remaining elements. Look through the triplets $(q,\varphi,\tau)$
enumerated as per \rmkref{tauwork} until you find such a triplet
that satisfies $M\vDash q(\vec{x},y,\vec{z})$ and $D\vDash\varphi(\vec{b})$.
If all elements of $\vec{b}$ are real, such a triple will be found
in finite time by \enuref{asmchi} in conjunction with Item \ref{enu:qetp2}
of the QETP. If not all are real, injury will occur. Once such a triple
is found, we assign to $m$ the triplet $(q,\varphi,\tau)$ with the
decomposition $(\vec{x},y,\vec{z})$ above.

We need to make sure that \ref{enu:asmphi} and \ref{enu:asmetau}
are preserved.
\begin{itemize}
\item \ref{enu:asmphi} -- By design.
\item \ref{enu:asmetau} -- We need to guarantee that, if the $\vec{b}$
are all real, we have $D\vDash\etau_{q\varphi q}(\vec{b})$. This
is a consequence of the note at the end of Item \ref{enu:qetpetau}
of the QETP.
\end{itemize}
Let $(q,\varphi,\tau)$ be the formula assigned to $m$. The formula
$q(\vec{x},y,\vec{z})$ will not, in general, be the quantifier-free
$\Lang$-type of the current elements of $M$, but it is definitely
the quantifier-free $\Lang$-type of some subset of elements of $M$.
By \ref{enu:asmutau}, the elements corresponding to $\vec{x}$ are
exactly the ones that currently admit matches in $D$. Let $\vec{b}$
denote their matches.

We let $Q(\vec{x},y,\vec{z},\vec{w})$ denote the quantifier-free
type of the current elements of $M$, where $\vec{x}$, $y$, and
$\vec{z}$ are the same as above.

Look through the elements of $D$ until you find an element $d$ satisfying
$D\vDash\tau(\vec{b},d)$ to current approximation (in the sense of
\rmkref{llorc2}).

{\small Technical note: This search must be done in some systematic
way so that, if all $\vec{b}$ are real, we eventually find a }{\small\emph{real}}{\small{}
element $d$ satisfying this formula in such a way that the approximation
nevermore changes. An example of such a systematic way: Let's say
we give the $n$-th element added to $D$, say $d_{n}$, a penalty
of $n$ stages in the time it takes to compute whether $D\vDash\tau(\vec{b},d_{n})$,
and select the $d$ such that $D\vDash\tau(\vec{b},d)$ is verified
(and not disproven until now) in the least amount of stages. This
will be important when proving finite injury.}{\small\par}

We claim that, unless an injury occurs, this search will eventually
terminate. By this we mean: Assume that all $\vec{b}$ are real, and
that the triple $(q,\varphi,\tau)$ will never be removed from the
enumeration from \rmkref{tauwork}. Then, by \enuref{asmchi} it must
be the case that $D\vDash\chi_{Q}(\vec{b})$ and by \enuref{asmetau}
it must be the case that $D\vDash\etau_{q\varphi Q}(\vec{b})$. Thus,
by definition of $\etau$, it must be the case that $D\vDash\exists_{y}\tau_{q}(\vec{b},y)$,
and any witness to this statement will serve as the $d$ above.

Once such a $d$ is found, we match $m$ to $d$.

We need to make sure that \enuref{asmchi} and \enuref{asmtau} hold.
\begin{itemize}
\item \enuref{asmchi} -- If $d$ is real, this is direct by Item \ref{enu:qetp3}
of the QETP.
\item \enuref{asmtau} -- By construction
\end{itemize}
\item Next, find the first element of $D$ -- call it $d$ -- that does
not have a match in $M$. If there is no such element as of right
now, proceed to the next bullet point.

Let $q(\vec{x},\vec{y})$ be the quantifier-free $\Lang$-type of
the elements of $M$, and $\vec{b}$ the matches of the elements of
$M$ that admit matches. Let $r(\vec{x},z)$ be the quantifier-free
$\Lang$-type of the tuple $(\vec{b},d)$.\footnote{Here, we use Item \ref{enu:qetp0} of the QETP to computably obtain
(a strong index for) $r(\vec{x},z)$ from a c.e.\ index for the atomic
$\Lang'$-type of $(\vec{b},d)$. More generally, we use the fact
that Item \ref{enu:qetp0} of the QETP allows us to computably obtain
a strong index for the $\Lang$-type of a tuple from a c.e.\ index
for its $\Lang'$-type.} In parallel, look through the realized quantifier-free $\Lang$ existential
types $\exists_{\vec{x}}\exists_{\vec{y}}\exists_{z}Q(\vec{x},z,\vec{y})$
until you find such a $Q(\vec{x},z,\vec{y})$ extending $q(\vec{x},\vec{y})\land r(\vec{x},z)$
such that:
\begin{itemize}
\item $D\vDash\chi_{Q(\vec{x},z,\cdot)}(\vec{b},d)$, and
\item For every $m$ with an assigned triplet $(q_{0}(\vec{x}_{0},y_{0},\vec{z}_{0}),\varphi(\vec{x}_{0}),\tau(\vec{x}_{0},y_{0}))$,
if $\vec{b}_{0}$ are the elements of $M$ corresponding to $\vec{x}_{0}$,
$D\vDash\etau_{q_{0}\varphi Q}(\vec{b}_{0})$.
\end{itemize}
We argue that either such a $Q$ is found in finite time or otherwise
an injury will occur. Indeed, suppose that both $d$ and the elements
of $\vec{b}$ are all real, and that every triplet $(q_{0},\varphi,\tau)$
will nevermore be injured. Now, since we assume $D\vDash\chi_{q}(\vec{b})$,
there must be a tuple of real elements $\vec{c}$ in $D$ such that
$D\vDash q(\vec{b},\vec{c})$. Now, we claim that the type $Q$ of
the tuple $(\vec{b},d,\vec{c})$ satisfies the desired requisites:
\begin{itemize}
\item Since $D\vDash Q(\vec{b},d,\vec{c})$, we must have $D\vDash\exists_{\vec{y}}Q(\vec{b},d,\vec{y})$
and hence $D\vDash\chi_{Q}(\vec{b},d)$, and
\item If $m$ has the triplet $(q_{0},\varphi,\tau)$ assigned, with decomposition
$(\vec{x}_{0},y_{0},\vec{z}_{0})\subseteq M$, with $\vec{b}_{0}$
being the elements of $D$ corresponding to $\vec{x}_{0}$ and $d_{0}$
being the element corresponding to $m$, we have $D\vDash\chi_{Q}(\vec{b}_{0},d_{0})$
as a consequence of $D\vDash\chi_{Q}(\vec{b},d)$ as a consequence
of $D\vDash Q(\vec{b},d,\vec{c})$, and moreover $D\vDash\tau(\vec{b}_{0},d_{0})$
by \enuref{asmtau}, hence $D\vDash\exists_{y}\left[\chi_{Q}(\vec{b}_{0},y)\land\tau_{q\varphi}(\vec{b}_{0},y)\right])$,
and thus $D\vDash\etau_{q_{0}\varphi Q}(\vec{b}_{0})$ by Item \ref{enu:qetpetau}
of the QETP.
\end{itemize}
Once such a $Q(\vec{x},z,\vec{y})$ is found, one of two things occurs.
Either it includes $z=y_{j}$ for some value of $j$, in which case
we match $d$ to the element corresponding to $y_{j}$, or it includes
the information that all its variables represent distinct elements,
in which case we add a new element to $M$, whose relations to the
previous elements are those dictated by $Q$, and match it to $d$.

We need to make sure that \enuref{asmchi} and \enuref{asmetau} hold.
\begin{itemize}
\item \enuref{asmchi} -- By construction, since we ensured $D\vDash\chi_{Q}(\vec{b},d)$.
\item \enuref{asmetau} -- By construction.
\end{itemize}
\item Once both of the above steps are done, worker $i$ shall pause its
execution, and initialize worker $i+1$.
\end{itemize}
We now describe the process of injury. There are three means of injury:
\begin{itemize}
\item Suppose that an element of $D$ is seen to be removed. If this element
does not have a match in $M$, no action must be taken. On the other
hand, suppose that this element $d$ is matched to an element $m$
in $M$. In that event, suppose that worker $i$ is the one that made
this match, and halt the execution of all workers with index greater
than $i$, erasing matches and assigned formulas created past the
moment that $m$ and $d$ were matched. If the match between $d$
and $m$ was created by the first bullet point, restart the execution
of worker $i$. If the match between $d$ and $m$ was created by
the second bullet point, initialize worker $i+1$ (and do not restart
worker $i$).
\item Suppose that an element $m\in M$ has an assigned triplet $(q,\varphi,\tau)$,
which is seen to be deleted from the enumeration described in \rmkref{tauwork}
at a given stage. Suppose worker $i$ is the one who has assigned
this formula. Then, we halt the execution of all workers with index
greater than $i$, and erase all matches and assigned formulas created
past the moment that this formula was assigned to $m$. Then, we run
worker $i$.
\item Suppose that an element $m\in M$ has an assigned triplet $(q(\vec{x},y,\vec{z}),\varphi(\vec{x}),\tau(\vec{x},y))$
and a matched element $d$. Suppose worker $i$ is the one who has
created these matches. Suppose that the elements $\vec{x}$ are matched
to the tuple $\vec{b}$, and suppose that the current approximation
of $D\vDash\tau(\vec{b},d)$ (in the sense of \rmkref{llorc2}) has
changed to `no'. Then, we halt the execution of all workers with
index greater than $i$, erase all matches and assigned formulas created
past this point (including the match between $m$ and $d$, but we
retain the assigned triplet $(q,\varphi,\tau)$), and run worker $i$.
\end{itemize}
This concludes the construction. It remains to show that the resulting
structure $M$ is isomorphic to $D\upto\Lang$, and this is done via
a finite injury argument. We show that each worker will be injured
finitely many times. This will show that every element in $D$ is
eventually matched to an element of $M$ (the $i$-th element of $D$
(counting both real and fake elements) will be matched by, at worst,
the $i$-th worker past the moment when this element is added to $D$),
and that every element of $M$ is eventually matched to an element
of $D$ (given an orphan element of $M$, inductively consider the
first worker activated after all previous elements of $M$ have been
matched), and since these matches all preserve the predicates in $\Lang$,
the resulting map will be an isomorphism.

Suppose, for the sake of induction, that all workers prior to worker
$i$ are injured finitely many times. We show that worker $i$ itself
also suffers finite injury. We do this individually for each of the
bullet points:
\begin{itemize}
\item First, we show that the matching created by the first orphan element
of $m$ at this moment, if there is any, is injured finitely many
times. Indeed, per the notation of \rmkref{tauwork} and using Item
\ref{enu:qetp2} of the QETP, eventually a triplet $(q,\varphi,\tau)$
will be enumerated and never removed, such that the relevant elements
of $D$ satisfy $\chi_{q}\land\varphi$. Past this point, after all
previously-added false triplets have been removed, this triplet (or
another) will be permanently assigned to $m$. Afterward, we know
(see the technical note in the first bullet point of the description
of the worker process) that after enough steps we match $m$ with
a real element $d$ of $D$ that satisfies $\tau(\vec{b},d)$ at the
current approximation, which will nevermore change. Thus, past this
point, this match will never be injured.
\item Now, assuming that we've reached a time where the matching created
by the previous bullet point or previous workers will never suffer
injury, suppose that there is an unmatched element in $D$ for worker
$i$ to match with an element of $M$. If this element is real, no
injury will be suffered. If it is not, eventually it shall be erased,
and execution will move to worker $i+1$, with worker $i$ nevermore
suffering injury from that point.
\end{itemize}
This concludes the proof.
\end{proof}

\section{\protect\label{sec:prevs_applications}Previously Known Applications}

In this section, we present a list of some consequences of \thmref{main_final}.
We claim originality over none of the results themselves -- our goal
here is merely to show how they all follow from the same general principle,
and we present prior references to where the results may be found
in the literature.

All of the following results are bidirectional, but due to the direction
of this paper, we will only be focusing on the most relevant of the
two directions in each case, which will generally be arguably the
most difficult one.

\subsection{\protect\label{subsec:eqvrelsinf}Equivalence Relations of Infinite
Rank}
\begin{prop}[{Folklore, {\cites[Theorem 9.1.4]{cst_downey_melnikov}}}]
\label{prop:eqvinf}Let $(E,\eqvrel)$ be an equivalence relation
with infinitely many equivalence classes, and assume that the structure
$(E,\eqvrel,\{P_{\geq n}\}_{n\in\N})$ admits a $0'$-computable copy,
where $P_{\geq n}(x)$ is the predicate ``the equivalence class of
$x$ has size at least $n$''.\footnote{This includes the case where the equivalence class of $x$ is infinite.}
Then, $(E,\eqvrel)$ admits a computable copy.
\end{prop}
\begin{proof}
First, we note that if $(E,\eqvrel,\{P_{\geq n}\}_{n\in\N})$ admits
a $0'$-computable copy then $(E,\eqvrel,\{P_{\geq n}\}_{n\in\N\cup\{\infty\}})$
admits a $0'$-computable c.e.-typed copy, where $P_{\geq\infty}(x)$
is the predicate ``the equivalence class of $x$ is infinite''.
We construct this copy as follows: With an oracle for $0'$, run the
construction of $(E,\eqvrel,\{P_{\geq n}\}_{n\in\N})$, and every
time you see an element in a new finite equivalence class of size
exactly $n$ (i.e.\ satisfies $P_{\geq n}$ but not $P_{\geq n+1}$),
output an equivalence class of size exactly $n$. In the limit, you
will output a copy of the substructure of $E$ consisting of the elements
in finite equivalence classes. In parallel, build countably many infinite
equivalence classes. Since we know everything about the equivalence
class of any given element when we are outputting it, it is trivial
to provide a c.e.\ (or even computable) index for the type of the
tuple of elements output up to a given point.

Next, we show that $E$ satisfies the QETP relative to the languages:
\begin{itemize}
\item $\Lang$ is the language containing $\eqvrel$, $\not\eqvrel$, $=$,
and $\neq$,
\item $\Lang'$ is the language containing $\Lang$, plus the countably
many predicates $P_{\geq n}$.
\end{itemize}
It is obvious that $\Lang$ is finite and closed under negation, so
it suffices to construct $\chi$, $\tau$, and $\etau$.
\begin{itemize}
\item To construct $\chi_{q}(\vec{x})$ given $q(\vec{x},\vec{y})$: Without
loss of generality (see \rmkref{assumeq}) we assume that $q(\vec{x},\vec{y})$
is a maximal conjunction of atomic formulas of the types:
\begin{itemize}
\item $y_{i}\eqvrel x_{j}$ or $y_{i}\not\eqvrel x_{j}$, and one of these
is included for every $i$ and $j$,
\item $y_{i}\eqvrel y_{j}$ or $y_{i}\not\eqvrel y_{j}$, and one of these
is included for every $i$ and $j$,
\item $x_{i}\eqvrel x_{j}$ or $x_{i}\not\eqvrel x_{j}$, and one of these
is included for every $i$ and $j$,
\item $y_{i}\neq x_{j}$, and all of these are included for every $i$ and
$j$,
\item $y_{i}\neq y_{j}$, and all of these are included for $i\neq j$,
\item $x_{i}\neq x_{j}$, and all of these are included for $i\neq j$.
\end{itemize}
Thus, $q(\vec{x},\vec{y})$ includes all the information about these
elements in terms of what equivalence classes they inhabit. If the
information is inconsistent, (namely, if it fails transitivity, reflexivity,
or symmetry), simply set $\chi_{q}(\vec{x})\equiv\bot$. On the other
hand, if the information is consistent, for such a tuple $\vec{y}$
to exist given the tuple $\vec{x}$ requires, and indeed is equivalent,
to the demand that the equivalence class of each $x_{i}$ is large
enough -- namely, the equivalence class of $x_{i}$ must be of size
$n_{i}$ or above, where $n_{i}$ is the number of distinct variables
$z$ such that $z\eqvrel x_{i}$ is part of the information in $q$.
As such, we set $\chi_{q}(\vec{x})\equiv q_{0}(\vec{x})\land\bigwedge_{i}P_{\geq n_{i}}(x_{i})$,
where $q_{0}(\vec{x})$ is the information in $q$ that relates the
elements of $\vec{x}$. It will satisfy the requirements about $\chi$
by construction.

Note that we are using the assumption that there exist infinitely
many arbitrarily large equivalence classes to handle variables $y_{i}$
that are not equivalent to any $x_{j}$.
\item To construct $\tau_{q\varphi}(\vec{x},y)$: The domain of $\tau$
shall be the set of pairs $(q,\top)$. Given $q(\vec{x},y,\vec{z})$,
set $\tau_{q\top}(\vec{x},y)\equiv\chi_{q}(\vec{x},y)\land\left(P_{\geq\infty}(y)\lor\bigvee_{i}y\eqvrel x_{i}\right)$.
\item To construct $\etau_{q\varphi Q}(\vec{x})$: Simply set $\etau_{q\varphi Q}(\vec{x})\equiv\chi_{Q}(\vec{x})$
in all circumstances.
\end{itemize}
It is a tedious but straight-forward exercise in model theory to verify
that the above assignments serve as witnesses to the QETP. We use
the fact that there are infinitely many infinite equivalence classes
in an essential way when verifying Item \ref{enu:qetp2} of the QETP.
By application of \thmref{main_final}, the proof is complete.
\end{proof}
There are a few properties of the above proof that will come up again,
so we remark on them now.
\begin{remark}
We have used a very weak version of \thmref{main_final}. For example:
We used no infinitary formulas, our $\tau$ was computable with computable
domain, instead of merely $0'$-computable with $0'$-c.e.\ domain,
etc. This reflects the fact that, as far as jump inversion theorems
go, \propref{eqvinf} is rather simple! We will see some other natural
(though not necessarily much more complicated) examples that use more
of the power of \thmref{main_final}.
\end{remark}
\begin{remark}
Even though \thmref{main_final} provides a $0'$-computable isomorphism
between the reduct of the original structure and the resulting structure,
\propref{eqvinf} does not provide such a guarantee. Indeed, one can
show that the best we can do in general is to find a $0''$-computable
isomorphism between the original $0'$-computable equivalence relation
-- call it $E_{0}$ --, and the computable copy -- call it $E_{1}$.
This stems from the fact that we have built an intermediate auxiliary
structure to apply \thmref{main_final} to -- say $E_{+}$ -- and
even though \thmref{main_final} guarantees that there is a $0'$-computable
isomorphism between $E_{+}$ and $E_{1}$, an isomorphism between
$E_{0}$ and $E_{+}$ may be computably more complicated. The best
that we can guarantee, and indeed the best that can be guaranteed,
is that such an isomorphism is $0''$-computable, as \propref{eqvcomplexiso}
below shows.
\end{remark}
\begin{lemma}
\label{lem:enueqv}There is a computable enumeration of all computable
equivalence relations, finite or infinite.
\end{lemma}
\begin{proof}
\global\long\def\braket#1{\langle#1\rangle}%
Let $\varphi_{n}(x)$ denote the $n$-th $\{0,1\}$-valued partial
computable function. Then, for each $n$, define the $n$-th computable
equivalence relation on an initial segment of $\N$ as follows: If
we've added all numbers below a given $k$ to the equivalence relation,
then we attempt to calculate $\varphi_{n}\braket{0,k}$, $\varphi_{n}\braket{1,k}$,
$\dots$, $\varphi_{n}\braket{k,k}$ and $\varphi_{n}\braket{k,0}$,
$\varphi_{n}\braket{k,1}$, $\dots$, $\varphi_{n}\braket{k,k}$.
If any of these computations do not halt, neither $k$ nor any number
greater than $k$ is considered to be part of the $n$-th computable
equivalence relation, which is therefore finite with exactly $k$
elements. Otherwise, we check if the resulting relation on the set
$\{0,\dots,k\}$ is an equivalence relation. If it is not, we again
consider the $n$-th equivalence relation to have exactly $k$ elements;
if it is, we add $k$ to the equivalence relation and repeat with
$k+1$. This construction guarantees that any computable equivalence
relation has an isomorphic copy corresponding to some value of $n$.
\end{proof}
\begin{prop}
\label{prop:eqvcomplexiso}There is a $0'$-computable equivalence
relation $(E_{0},\eqvrel,\{P_{\geq n}\}_{n\in\N})$ with infinitely
many infinite equivalence classes such that any isomorphism $f$ between
$E_{0}\upto\eqvrel$ and a computable equivalence relation $(E_{1},\eqvrel)$
can be used to compute $0''$.
\end{prop}
\begin{proof}
We shall build $E_{0}=E_{00}\amalg E_{01}$, where $E_{00}$ shall
be used to encode $0'$ (into any isomorphism between $E_{0}$ and
a computable equivalence relation), and $E_{01}$ shall be used to
encode $0''$, via encoding the set of total Turing indices -- though
the decoding process will require the usage of $0'$, which is why
$E_{00}$ is necessary.

To build $E_{00}$: Consider designated elements $a_{0}$, $a_{1}$,
$a_{2}$, etc. where $a_{2n}$ is defined to be equivalent to $a_{2n+1}$
if, and only if, $n\in0'$, and no other equivalences exist between
any $a_{i}$ and any other element of $E_{0}$. If we have an isomorphism
$f$ between $E_{0}$ and a computable equivalence relation, we can
therefore tell if $n\in0'$ by checking whether $f(a_{2n})$ is equivalent
to $f(a_{2n+1})$.

To build $E_{01}$: Consider designated elements $\{b_{nx}\}_{(n,x)\in\N^{2}}$.
For each $n\in\N$, iterate through the values of $x\in\N$ until
such an $x$ is found with $\varphi_{n}(x)\isnotwd$. If no such $x$
is found, make $b_{ni}\eqvrel b_{nj}$ for all $i$, $j$. Otherwise,
suppose that $x$ is the first value for which $\varphi_{n}(x)\isnotwd$.
Then, spend $x$ steps enumerating elements of each of the first $n$
computable equivalence relations (as per \lemref{enueqv}), and using
$0'$ find a number $k\geq x$ such that none of these at most $xn$
elements are in an equivalence class of size exactly $k$. Then, make
$b_{n0}\eqvrel\cdots\eqvrel b_{nk}$, and no other $b_{ni}$ equivalent
to any other. The essential part of this choice is that it ensures
that no isomorphism between $E_{0}$ and one of these equivalence
classes will take $b_{n0}$ to one of their first $x$ elements.

Now, let us show that any isomorphism $f$ between $E_{0}$ and a
computable equivalence relation $E_{1}$ will compute $0''$. As argued
before, $f$ computes $0'$, so we use an oracle for $0'$ with impunity.
Suppose we wish to know whether some $n\in\N$ is a total Turing index.
Without loss of generality -- because this happens for all but finitely
many $n$ -- assume that $E_{1}$ is one of the first $n$ computable
equivalence classes. Then, find the index of $f(b_{n0})$ in the equivalence
relation $E_{1}$ -- say this index is $t$ -- and check whether
$\varphi_{n}(0)$ up to $\varphi_{n}(t)$ all terminate. If not, then
clearly $\varphi_{n}$ is nontotal. On the other hand, if $\varphi_{n}$
is nontotal, with $x$ the first element with $\varphi_{n}(x)\isnotwd$,
we know that $f(b_{n0})$ must have index $t$ greater than $x$,
in which case one of the computations $\varphi_{n}(0)$ to $\varphi_{n}(t)$
will not have terminated. Thus, checking these first $t$ elements
suffices to tell whether $\varphi_{n}$ is total.
\end{proof}
\begin{remark}
Our choice of $\tau_{q\varphi}$ included $\chi_{q}$ in its definition,
and our choice of $\etau_{q\varphi Q}$ included (and, in this case,
was equal to) $\chi_{Q}$. Inspecting the definition of QETP, one
finds that we can always do so -- that is, if $\chi_{q}$, $\tau_{q\varphi}$,
and $\etau_{q\varphi Q}$ serve as witnesses to the QETP, so will
$\chi_{q}$, $\tau_{q\varphi}\land\chi_{q}$, and $\etau_{q\varphi Q}\land\chi_{Q}$,
though the converse need not be true.

If we had wanted to, we could have modified the definition of QETP
to make our formulas simpler in practice, by replacing e.g.\ every
instance of $\tau$ with $\tau\land\chi$, and likewise for $\etau$.
For instance, Item \ref{enu:qetp3} currently requires
\[
D\vDash\forall_{\vec{x},y}\left[\etau_{q\varphi Q}(\vec{x})\land\tau_{q\varphi}(\vec{x},y)\rightarrow\chi_{Q}(\vec{x},y)\right],
\]
and we might consider making the weaker requirement that
\[
D\vDash\forall_{\vec{x},y}\left[\etau_{q\varphi Q}(\vec{x})\land\chi_{Q}(\vec{x})\land\tau_{q\varphi}(\vec{x},y)\land\chi_{q}(\vec{x},y)\rightarrow\chi_{Q}(\vec{x},y)\right].
\]

Such a change would not alter which structures satisfy the QETP, only
what their witnessing formulas are, and if we were optimizing for
the writing of applications, we could have used the weaker version
instead. For example, in the case of \propref{eqvinf}, we currently
have
\[
\tau_{q}(\vec{x},y)\equiv\chi_{q}(\vec{x},y)\land\left(P_{\geq\infty}(y)\lor\bigvee_{i}y\eqvrel x_{i}\right)\text{ and }\etau_{q\varphi Q}(\vec{x})\equiv\chi_{Q}(\vec{x}),
\]
and with the alternate version of the QETP this could be replaced
by the simpler and arguably more elegant
\[
\tau_{q}(\vec{x},y)\equiv P_{\geq\infty}(y)\lor\bigvee_{i}y\eqvrel x_{i}\text{ and }\etau_{q\varphi Q}(\vec{x})\equiv\top.
\]
We opted against using this sleeker version of the QETP because the
definition was already technical and complicated enough as it is,
and we gave priority to the understandability of the definition over
a minor gain in characters in the proof of applications.
\end{remark}

\subsection{Equivalence Relations of Finite Rank}

\propref{eqvfin} below was one of the main motivators for the current
form of the QETP. We need to state a couple of important definitions
first.
\begin{definition}[{Khisamiev, {\cites[Definition 9.1.3]{cst_downey_melnikov}}}]
\label{def:limitwisemonotonic}A set $X\subseteq\N$ is said to be
\emph{limitwise monotonic} if there is a total computable function
$f\colon\N\times\N\to\N$ such that, for every $x\in\N$, the sequence
$\{f(x,y)\}_{y\in\N}$ is weakly increasing (i.e.\ nondecreasing)
and bounded (and therefore convergent), and $X$ is the image of the
pointwise limit function $g(x)=\lim_{y}f(x,y)$.
\end{definition}
\begin{definition}[{{\cites[Definition 9.1.1]{cst_downey_melnikov}}}]
Let $(E,\eqvrel)$ be an equivalence relation. The \emph{characteristic
set of $E$} is the set
\[
\#E=\{\,n\in\N\mid\text{\ensuremath{E} admits at least one equivalence class of size \ensuremath{n}}\,\}.
\]
\end{definition}
\begin{prop}[{Folklore, {\cites[Theorem 9.1.4]{cst_downey_melnikov}}}]
\label{prop:eqvfin}Let $(E,\eqvrel)$ be an equivalence relation
with finitely many infinite equivalence classes, and assume that the
structure $(E,\eqvrel,\{P_{\geq n}\}_{n\in\N})$ admits a $0'$-computable
copy, where $P_{\geq n}(x)$ is the predicate ``the equivalence class
of $x$ has size at least $n$''. Assume moreover that the characteristic
set of $E$ is limitwise monotonic. Then, $(E,\eqvrel)$ admits a
computable copy.
\end{prop}
\begin{proof}
Similarly but not exactly like in proof of \propref{eqvinf}, our
first step is to construct a $0'$-computable c.e.-typed copy of $(E,\eqvrel,\{P_{\geq n}\}_{n\in\N\cup\{\infty\}},\{P_{=n}\}_{n\in\N\cup\{\infty\}})$,
where $P_{=n}(x)$ denotes ``the equivalence class of $x$ has size
exactly $n$''. The same construction as in the proof of \propref{eqvinf}
can be used with very minor modifications. The added predicates are
no issue, because when we add an element we already know its equivalence
class size exactly, and instead of manually adding infinitely many
equivalence classes, we add as many as $E$ has (a finite number which
has to be obtained non-uniformly). It then remains to show that $(E,\eqvrel,\{P_{\geq n}\}_{n\leq\infty},\{P_{=n}\}_{n\leq\infty})$
has the QETP over the language containing only $\eqvrel$ (and its
negation, and equality and its negation), though there are some added
technicalities.

Throughout, we shall assume that $E$ admits arbitrarily large finite
equivalence classes, or in other words that $\#E$ is an infinite
set. The case where $\#E$ is finite is relatively trivial and is
handled separately: If $\#E$ is finite, then it takes finitely much
space to encode how many equivalence classes $E$ admits in each size,
and we simply make a computable equivalence relation with this specification.
We will see in \rmkref{cardefinitehard} that this is necessary, and
in \rmkref{cardefinitewhy} the root cause and how to get around it.

In the following, we let $f$ be a computable function witnessing
that $\#E$ is limitwise monotonic, as in \defref{limitwisemonotonic}.
\begin{itemize}
\item To construct $\chi_{q}(\vec{x})$ given $q(\vec{x},\vec{y})$: The
same construction as in \propref{eqvinf} holds, so long as, as therein
observed, $E$ is assumed to have arbitrarily large finite equivalence
classes.
\item To construct $\tau$: Let $q(\vec{x},y,\vec{z})$ range over the quantifier-free
$\Lang$-types realized in $E$, and let $\varphi(\vec{x})$ range
over the possible assignments of equivalence class sizes for the elements
of the tuple $\vec{x}$. For such a pair $(q,\varphi)$, set $N_{0}$
to be the largest finite size of equivalence class imposed by $\varphi(\vec{x})$,
set $N_{1}$ to be the smallest size of the equivalence class of a
$y$ that satisfies $\chi_{q}(\vec{x},y)$, and find the smallest
pair $(i,j)$ in some fixed well-ordering of $\N^{2}$ that satisfies
$f(i,j)>\max\{N_{0},N_{1}\}$. Then, set $M(i)=\lim_{k}f(i,k)$. Finally,
we set 
\[
\tau_{q\varphi}(\vec{x},y)\equiv\chi_{q}(\vec{x},y)\land\left(P_{=M(i)}(y)\lor\bigvee_{i}y\eqvrel x_{i}\right).
\]
.
\item To construct $\etau_{q\varphi Q}(\vec{x})$, let $N$ be the least
equivalence class size that $y$ would need to have in order to satisfy
$\chi_{Q}(y)$, and $i$ the value obtained from the previous bullet
point, and set
\[
\etau_{q\varphi Q}(\vec{x})\equiv\llor_{k\in\N}\chi_{Q}(\vec{x})\land\left(\left\llbracket f(i,k)\geq N\right\rrbracket \lor\bigvee_{i}y\eqvrel x_{i}\right),
\]
where $\llbracket f(i,k)\geq N\rrbracket$ denotes
\[
\llbracket f(i,k)\geq N\rrbracket\equiv\begin{cases}
\top & \text{if \ensuremath{f(i,k)\geq N},}\\
\bot & \text{otherwise.}
\end{cases}
\]
\end{itemize}

As in \propref{eqvinf}, checking the QETP is a straight-forward,
if tedious, affair. The most interesting part of the proof consists
of noticing that $\etau_{q\varphi Q}(\vec{x})$ is logically equivalent
to $\chi_{Q}(\vec{x})\land\left(\llbracket M(i)\geq N\rrbracket\lor\bigvee_{i}y\eqvrel x_{i}\right)$,
which uses the fact that $f(i,k)$ is monotonic in $k$. By application
of \thmref{main_final}, the proof is complete in the case where $\#E$
is infinite.
\end{proof}

\begin{remark}
Unlike in \propref{eqvinf}, for \propref{eqvfin} we can, in fact,
guarantee the existence of a $0'$-computable isomorphism, albeit
non-uniformly so. To do this, one picks one element of each infinite
equivalence class of the original equivalence relation, say $E_{0}$,
and hard-codes this collection of representatives into our algorithm
(which is fine, because there are finitely many of them). Then, in
the case where $\#E$ is infinite (the case where $\#E$ is finite
is straightforward and left to the reader), one defines $E_{+}$ as
outlined in the proof, and defines an isomorphism between $E_{0}$
and $E_{+}$ as follows: The elements of $E_{0}$ that belong to a
finite equivalence class are mapped in an obvious way to the equivalence
class they're responsible for creating, and the elements of $E_{0}$
that belong to an infinite class are detected (by checking against
equivalence with one of the finitely many representatives), and to
each infinite class of $E_{0}$ one remembers an infinite class of
$E_{+}$ to biject it to, and does so. This isomorphism is $0'$-computable,
and since there is also a $0'$-computable isomorphism between $E_{+}$
and the final computable equivalence relation $E_{1}$, \propref{eqvfin}
can in fact be made to guarantee the existence of a $0'$-computable
isomorphism.

To show that this cannot be done uniformly, not even for the case
where the equivalence relation $E_{0}$ is guaranteed to have equivalence
classes of every finite size (which immediately makes $\#E$ uniformly
limitwise monotonic) and have at most -- or even exactly -- one
infinite equivalence class, one can piggyback on the proof of \propref{eqvcomplexiso},
by noticing that the equivalence relation therein defined is a computable
disjoint union $\amalg_{i}E_{(i)}$ of infinite equivalence relations
with either zero or one infinite classes. These equivalence relations
also do not have equivalence classes of every finite size, but since
the factor that makes the isomorphism computably complicated is the
designated elements of each $E_{(i)}$, we can simply append one equivalence
class of each size to every $E_{(i)}$. We can also ensure that each
$E_{(i)}$ has exactly one infinite class by, in the case of (in the
language of the proof of \propref{eqvcomplexiso}) $E_{00}$, simply
appending one infinite class, and in the case of each component of
$E_{01}$, making all the $\{b_{ni}\}_{i}$ equivalent for large values
of $i$ (instead of making them all singleton classes). This shows
that \propref{eqvfin} cannot be made to uniformly produce a $0'$-computable
isomorphism, even restricted to the case where the equivalence relation
has exactly one equivalence class of every size in $\N\cup\{\infty\}$.

On the flip side, if $E_{0}$ is guaranteed to have \emph{no} infinite
classes whilst still having $\#E$ infinite limitwise monotonic, we
do in fact obtain uniformity\emph{ in $E_{0}$ together with a witness
to the fact that $\#E$ is limitwise monotonic}, and this is direct
from the proof of \propref{eqvfin}: If $E_{0}$ has no equivalence
classes, the isomorphism between $E_{0}$ and $E_{+}$ is easily and
uniformly constructed, and we need the witness that $\#E$ is limitwise
monotonic to construct $\etau$.

Finally, here is a sketch of proof that you really do need a witness
to $\#E$ being limitwise monotonic as part of the input, to guarantee
uniformity. Suppose there were a computable process $F$ that took
as input a $0'$-index for an equivalence relation $(E,\eqvrel,\{P_{\geq n}\}_{n\in\N})$
with no infinite classes and with $\#E$ limitwise monotonic, and
produced as output a computable index for $(E,\eqvrel)$. Then, using
the Recursion Theorem relativized to $0'$, we may build an equivalence
relation $E$ as follows: First, look at $F(E)$. We may assume WLOG
that $F(E)$ is a well-defined index for a relational structure by
appealing to $0'$, and we may assume that it represents a valid equivalence
relation because there is a computable process that takes an index
for a relational structure $R$ and outputs an index for an equivalence
relation $E(R)$ such that, if $R$ was already an equivalence relation
then $E(R)\cong R$. Now, using $0'$, ask whether $F(E)$ contains
any elements. If not, set $E$ to consist of one equivalence class
of every size. If yes, pick an element and repeatedly ask: Does the
equivalence class of this element have size at least $n+1$? Whenever
the answer is `yes', add an equivalence class of size $n$ to $E$.
If the answer is ever seen to be `no', this implies that the equivalence
class of this element has size exactly $n$, and so we \emph{do not}
add an equivalence class of this size, but add an equivalence class
of every other size. It should be clear that this ensures that $E$
is not isomorphic to $F(E)$ despite the fact that $E$ admits no
infinite equivalence classes and $\#E$ is either $\N$ or $\N$ minus
a point, and hence limitwise monotonic. This is a contradiction, which
shows that such a computable $F$ cannot exist.
\end{remark}
\begin{remark}
\label{rmk:cardefinitehard}The fact that we have handled the ``$\#E$
is finite'' case separately (and non-uniformly) is unavoidable, and
here is why. We claim that there is no uniform way to turn an index
for a $0'$-computable equivalence relation $(E,\eqvrel,\{P_{\geq n}\}_{n\in\N})$,
even one with no infinite equivalence classes, into a computable index
for $(E,\eqvrel)$. In fact, we prove the even stronger fact: Any
oracle $X$ that is capable of uniformly turning a $0'$-computable
equivalence relation $(E,\eqvrel,\{P_{\geq n}\}_{n\in\N})$ with infinitely
many singleton classes, and either one or two size-two equivalence
classes\footnote{Note that this means $\#E=\{1,2\}$, and in particular is uniformly
limitwise monotonic.}, and no classes of any other size, into a computable index for $(E,\eqvrel)$,
may be used to compute $0''$.

Suppose we wish to determine if a computable function $f(x)$ is total.
First, note that we can uniformly (and computably) create an index
for a $0'$-computable equivalence relation $(E,\eqvrel,\{P_{\geq n}\}_{n\in\N})$
that is composed entirely of infinitely many singletons and one size-two
class, unless a value of $x$ is exists for which $f(x)\isnotwd$,
in which case we add a second equivalence class with two elements.
Then, using the oracle $X$, we create a computable index for $(E,\eqvrel)$.
Computably, we can search in $E$ for two distinct pairs of distinct
but equivalent elements. We will find two if, and only if, $f$ is
nontotal. This shows that $X$ enumerates the set of nontotal indices,
which straightforwardly implies $X\geq0'$. Thus, instead of computably
searching in $E$ for two pairs of distinct elements, $X$ can check
in finite time if such a elements exist, and thus determine whether
$f$ is total.
\end{remark}
\begin{remark}
\label{rmk:cardefinitewhy}A way to understand what makes equivalence
classes with finite characteristic set so computationally hard to
handle is the following intuition: A computable copy of $(E,\eqvrel)$
contains more descriptive power, from the perspective of $0'$, than
a $0'$-computable copy of $(E,\eqvrel,\{P_{\geq n}\}_{n\in\N})$.
In other words, there is computational content that $0'$ can read
off from a computable copy of $(E,\eqvrel)$ that it could not read
off from a $0'$-computable copy of $(E,\eqvrel,\{P_{\geq n}\}_{n})$,
and an example of such content is the answer to the questions: ``Is
there an equivalence relation in $E$ of size greater than or equal
to $n$?'' or, more generally, ``Are there at least $k$ equivalence
relations in $E$ of size $\geq n$?'' As such, knowledge of the
answer to these questions is a necessary condition to be able to turn
a $0'$-computable copy of $E$ into a computable copy thereof, because
the answer to these questions \emph{can} be computed using $0'$ from
the computable copy. This explains what makes the case with infinitely
many equivalence relations, or with arbitrarily large finite equivalence
classes, computationally easy to handle: The answer to these questions
is always `yes'. Moreover, the answer to these questions is always
simple enough that it can be encoded in finitely many bits of information,
and hence could be non-uniformly hardcoded into an algorithm to obtain
a computable copy of $(E,\eqvrel)$, but the fact that this cannot
be done uniformly explains why \propref{eqvfin} cannot be made uniform
as it is.
\end{remark}

\subsection{Khisamiev's Theorem}

The following theorem is in the literature, see e.g. Chapter 5 of
\cite{cst_downey_melnikov}. We obtain it as a corollary of \thmref{main_final}.
It was the main motivator for the definition of $0'$-computable c.e.-typed
structure.

In the following, we interpret a c.e.\ presentation of a group as
being a description of the group as a list of generators and relations,
such that there is a computer program that enumerates the list of
generators and relations. In the present case, this is equivalent
to the more general notion of c.e.\ presentation of a structure in
computability theory.
\begin{thm}[Khisamiev]
\label{thm:khisamiev}Every c.e.\ presented torsion-free abelian
group is isomorphic to a computable group. Furthermore, if the group
is non-trivial, then this computable copy can be built uniformly in
the index of the c.e.\ presentation.
\end{thm}
\begin{proof}
Given a c.e.\ presentation of a group $G$, we first build a $0'$-computable
c.e.-typed copy of $G$ over an appropriate signature $\Lang'$, and
we will then show that any nontrivial torsion-free abelian group has
the QETP (uniformly) over $\Lang\subseteq\Lang'$, where $\Lang$
is the signature containing only the ternary predicate $S(x,y,z)\equiv\text{``\ensuremath{x+y=z}''}$(and
its negation, and equality and its negation).

The signature $\Lang'$ shall contain $\Lang$, plus:
\begin{itemize}
\item The countably many predicates $Q_{n,\vec{c}}(\vec{x})$, parametrized
over $n\in\N$ and finite tuples $c_{1},\dots,c_{k}\in\Z$:
\[
Q_{n,\vec{c}}(\vec{x})\equiv\text{``\ensuremath{n\mid c_{1}x_{1}+\dots+c_{k}x_{k}}''}.
\]
Note that we allow $n$ to equal zero, in which case (since $0\mid y$
iff $y=0$) $Q_{0,\vec{c}}(\vec{x})$ corresponds to the predicate
$\vec{c}\cdot\vec{x}=0$.
\item The countably many predicates $\neg Q_{0,\vec{c}}(\vec{x})$, parametrized
over finite tuples $c_{1},\dots,c_{k}\in\Z$:
\[
\neg Q_{0,\vec{c}}(\vec{x})\equiv\text{``\ensuremath{c_{1}x_{1}+\dots+c_{k}x_{k}\neq0}''}.
\]
\item Countably many predicates $\{n\Delta(x)\}_{n\in\Z}$, whose meaning
is as follows: We will carefully pick a single designated nonzero
element of $G$, which we refer to as $\delta$, and $n\Delta(x)$
should be interpreted as ``$x=n\delta$''.
\end{itemize}
We shall require some linear algebra of $\Z$-modules to show that,
from a c.e.\ presentation of $G$, we can uniformly obtain a $0'$-computable
c.e.-typed copy of $G$. Then, we will explain how we choose the element
of $G$ that will be tagged with $\Delta$. To obtain the $0'$-computable
c.e.-typed copy of $G$, here is the basic idea: Consider the computable
process that outputs all possible linear combinations of generators.
We want to, with each new linear combination of generators, say $d$,
first check whether it is equal to any of the combinations previously
produced (this is an easy check with $0'$), and if it is not, we
want to output a c.e.\ index for its type relative to the elements
previously output. This is easy for the positive predicates -- to
check if, for example, $5\mid4d+6c$ (where $c$ is a previously output
combination of generators), it suffices to check if there is a further
combination of generators $u$ that satisfies the relation $5u-4d-6c=0$.
This is a c.e.\ check. The difficulty lies in the negative predicates,
of which $\neg Q_{0,\vec{c}}(\vec{x})$ is the general case. In other
words, we need to encode \emph{into a single c.e.\ index} the full
information about what linear combinations of elements-output-thus-far
equal zero.

Here is a fact from commutative algebra, whose proof can be found
in any standard book that covers modules over principal ideal domains
(e.g.\ \cite{dummit_foote}):
\begin{enumerate}[label={Fact \Alph*}]
\item Any submodule $M$ of $\Z^{k}$ admits a finite basis.
\end{enumerate}
We will need an effective version of this fact. A stepping stone,
which we will not directly use but displays the main idea, is the
following:
\begin{enumerate}[resume, resume*]
\item \label{enu:gauss}Given a finite set $X\subseteq\Z^{k}$, we can
effectively find a basis for the submodule of $\Z^{k}$ generated
by $X$.
\end{enumerate}
Proof idea: A slight modification of the usual Gaussian elimination
algorithm, applied to the matrix whose rows are the elements of $X$.
To zero out entries below a pivot, instead of dividing to turn the
pivot element into $1$, use the Euclidean algorithm to turn the pivot
into the GCD of itself with all elements beneath it, whereupon all
elements beneath it may be zeroed out.

Now we are ready to state and prove the version we will use:
\begin{enumerate}[resume, resume*]
\item \label{enu:cemodulebasis}Given a c.e.\ index for a submodule $M$
of $\Z^{k}$, we may $0'$-effectively find a basis of $M$.
\end{enumerate}
Proof idea: Consider the $k$-width and possibly-infinite-height matrix
whose rows are the elements of $M$. Use $0'$ to find the GCD of
the elements of the first column, call it $d$. If $d=0$ (i.e.\ all
elements of the first column are zero), skip to the next step. If
$d\neq0$, find a combination of rows of $M$ that has $d$ in its
first coordinate, and add this combination to the basis. Then, if
$k\geq1$, recursively apply this algorithm to the $(k-1)$-width
and possibly-infinite-height matrix whose elements consist of the
rows of the previous matrix that have a $0$ as a first entry. Upon
finding a combination of rows that has the new GCD as its first coordinates,
add the corresponding combination of the rows of the original matrix
to the basis. After $k$ steps, one shall have a basis of $M$.

Returning to our original problem: We have a collection of combinations
of generators, say $d_{0}$, $d_{1}$, $\dots$, $d_{k}$, and we
wish to encode all possible linear combinations of these elements
that yield zero. We do this by applying \ref{enu:cemodulebasis} to
the c.e.\ submodule of $\Z^{k+1}$ given by
\[
M=\{\,\vec{c}\in\Z^{k+1}\mid\vec{c}\cdot\vec{d}=0\,\}.
\]
With access to a basis of $M$, say $b_{1},\dots,b_{\ell}$ (encoded
into the columns of a matrix $B$), we can tell whether any set of
coefficients $\vec{c}$ is in $M$ (and hence, whether we should say
yes or no to $Q_{0,\vec{c}}(\vec{d})$) by using standard algorithms
to solve the linear equation $B\vec{x}=\vec{c}$ over $\Q$, and conclude
$\vec{c}\in M$ if and only if a solution exists and its coefficients
are all integers. By hardcoding this basis into a c.e.\ index, we
can therefore obtain a c.e.\ index for the atomic $\Lang'$-type
of a tuple $\vec{d}$ of linear combinations of generators.

\smallskip{}

Now, let us explain how to choose the element $\delta$, which we
shall use to decide the predicates $n\Delta$. Of course, this has
to have been done in advance, so that we can hardcode the combination
of generators corresponding to $\delta$ into the c.e.\ indices of
the types of the elements of $G$ in advance.

There is an easy way and a hard way to choose $\delta$. The easy
way is simply to pick an arbitrary nonzero element of $G$. The issue
is that to make this choice requires the usage of an oracle for $0'$,
which will lead to a non-uniform presentation. Nevertheless, if one
is not interested in uniformity, one may move on and ignore the difficult
but uniform way to choose the element that we will now propose.

Consider the following schema: Let $\{C_{i}\}_{i\in\N}$ enumerate
all possible combinations of generators. To start, we will guess $\delta=C_{0}$,
and look through the relations until (if ever) we see enough relations
to conclude $C_{0}=0$. When this happens, we will guess $\delta=N_{1}C_{1}$,
where $N_{1}$ is a nonzero number divisible by enough integers. This
choice is made so that, for every number $k$ for which we'd been
able to conclude $k\mid C_{0}$ before we found that $C_{0}=0$, we
have $k\mid N_{1}$ and so $k\mid N_{1}C_{1}$. We proceed in this
manner, enumerating relations until we are able to conclude that $N_{1}C_{1}=0$,
or equivalently that $C_{1}=0$. If this happens, we change our mind
to guess $\delta=N_{2}C_{2}$, where $N_{2}$ has enough divisibilities
that, again, every divisibility we'd already concluded about $N_{1}C_{1}$
also holds for $N_{2}$ and hence $N_{2}C_{2}$. Since, by assumption,
$G$ is a nontrivial group, eventually this procedure will run into
a nonzero combination of generators, and $0'$ is able to figure out
which element is $\delta$ in advance.

The reader may be wondering in what respect this procedure is more
uniform. After all, the limiting combination of generators $N_{M}C_{M}$
cannot be computed uniformly, only $0'$-uniformly. However, and crucially,
we can uniformly compute a c.e.\ index for the type of the limiting
combination of generators $\delta$ as follows: follow along the procedure
used to find $\delta$, and every time we conclude a new divisibility
for (the current guess for) $\delta$, we add this divisibility to
its type. We took care so that any divisibility we add never needs
to be removed. The predicates $\neg Q_{0,c}(x)$ are all true except
for $\neg Q_{0,0}(x)$ because the group is torsion-free, and thus,
adding those to the type, we obtain an enumeration of the type of
$\delta$.

Below, we shall be using the type of $\delta$ in an essential way
(namely, to compute $\etau$), which is why having a uniform way to
compute this type is essential.

\medskip{}
Now that we have a $0'$-computable c.e.-typed copy of our group over
the signature $\Lang'$, we shall show that every nontrivial torsion-free
abelian group with a designated nonzero element with c.e.\ atomic
type satisfies the QETP.
\begin{itemize}
\item First, we construct $\chi_{q}$. This is an adaptation of the standard
proof that the theory of torsion-free abelian groups with divisibility-by-integer
predicates admits quantifier elimination, which can be found e.g.\ in
\cites[Appendix A.2]{hodges_model_theory}. Very briefly:
\begin{itemize}
\item Assume without loss of generality that $q(\vec{x},\vec{y})$ is a
disjunction of atomic formulas (see \rmkref{assumeq}), and hence
can be expressed as a system of integer coefficient equations and
inequations
\[
q(\vec{x},\vec{y})\equiv\left\{ \begin{array}{l}
M\vec{x}=N\vec{y},\\
K\vec{x}\neq L\vec{y}.
\end{array}\right.
\]

\textbf{Important:} The notation $K\vec{x}\neq L\vec{y}$ is \emph{not}
being used two mean ``the vectors $K\vec{x}$ and $L\vec{y}$''
are distinct, but rather the stronger fact: ``every entry of the
vector $K\vec{x}$ is distinct from the corresponding entry in $L\vec{y}$''.
This means that one must not be careless in handling these inequalities
-- one cannot, for example, apply any kind of Gaussian elimination
thereto.
\item Now, write $N=ADB$, where $D$ is a diagonal matrix padded with zeros
(i.e.\ of the form $\left[\begin{smallmatrix}D_{0} & 0\\
0 & 0
\end{smallmatrix}\right]$ with $D_{0}$ diagonal) and $A$ and $B$ are invertible, all over
the integers -- this can always be done using a strategy like the
one from \ref{enu:gauss}, see e.g.\ \cites[Theorem 2.12]{bapat_graphs_matrices}.
Then, $q(\vec{x},\vec{y})$ is logically equivalent to a new system
of equations and inequations
\[
q(\vec{x},\vec{y})\Leftrightarrow\left\{ \begin{array}{l}
M'\vec{x}=DB\vec{y},\\
K'\vec{x}\neq LB^{-1}B\vec{y}.
\end{array}\right.
\]
\item The existence of a solution to this system is, since $B$ is invertible,
equivalent to the existence of a solution to the following other system
\[
\exists_{\vec{y}}q(\vec{x},\vec{y})\Leftrightarrow\exists_{\vec{z}}\left\{ \begin{array}{l}
M'\vec{x}=D\vec{z},\\
K'\vec{x}\neq L'\vec{z}.
\end{array}\right.
\]
\item The rows of the equation $M'\vec{x}=D\vec{z}$ that have nonzero entries
of $D$ yield equations of the type $nz_{i}=\vec{a}\cdot\vec{x}$.
This allows us to remove (whilst preserving logical equivalence) every
instance of $z_{i}$ from every other (in)equation: Any other (in)equation
can be multiplied by $n$ preserving logical equivalence (this uses
$n\neq0$ and the fact that we're working with torsion-free groups),
resulting in an instance of $\text{(coefficient)\ensuremath{\times}}nz_{i}$,
which can be substituted for $\text{(coefficient)}\times(\sum a_{i}x_{i})$.
Thus, we may eliminate all $z$ variables corresponding to nonzero
rows of $D$, and our system now becomes
\[
\exists_{\vec{y}}q(\vec{x},\vec{y})\Leftrightarrow\exists_{\vec{z}}\exists_{\vec{w}}\left\{ \begin{array}{l}
M_{0}\vec{x}=0,\\
n_{1}z_{1}=\vec{c}_{1}\cdot\vec{x},\\
\vdots\\
n_{\ell}z_{\ell}=\vec{c}_{\ell}\cdot\vec{x},\\
K'\vec{x}\neq L'\vec{w}.
\end{array}\right.
\]
\item The equations including the $z$ variables become divisibility statements,
i.e.\ $\exists_{z_{i}}nz_{i}=\vec{c}\cdot\vec{x}\Leftrightarrow n\mid\vec{c}\cdot\vec{x}\equiv Q_{n,\vec{c}}(\vec{x})$,
so we now need only deal with the inequations $K'\vec{x}\neq L'\vec{w}$.
\item Now, let us stratify the equations $K'\vec{x}\neq L'\vec{w}$ as follows:
First, consider the equations that use no $w$-variable. Then, consider
the ones that use $w_{1}$ and no other. Then, consider the ones that
use only $w_{1}$ and $w_{2}$, and so on. The equations that use
no $w$-variable are important -- keep them. These are a requirement
on $\vec{x}$ that we cannot get rid of. Let us refer to these equations
as $K_{0}\vec{x}\neq0$. On the other hand, at every new stage of
the strata, the resulting equations will doubtlessly be satisfied
by some $w_{i}$. For example, if the first stratum has $e$ equations,
this is forbidding at most $e$ possible values of $w_{1}$. But the
group is nontrivial and hence infinite, whereby there is at least
one valid choice for $w_{1}$. Then, in the next stratum, contingent
on this choice of $w_{1}$, there are once again at most (number of
equations) forbidden values for $w_{2}$. Simply choose $w_{2}$ to
not be one of those. Continuing in this manner, we see that we can
ensure that all equations that include a $w$ variable can be automatically
satisfied and are therefore adding no information.
\item In conclusion, we've shown that any existential statement is equivalent
to a system of the type
\begin{equation}
\exists_{\vec{y}}q(\vec{x},\vec{y})\Leftrightarrow\left\{ \begin{array}{l}
M_{0}\vec{x}=0,\\
n_{1}\mid\vec{c}_{1}\cdot\vec{x},\\
\vdots\\
n_{\ell}\mid\vec{c}_{\ell}\cdot\vec{x},\\
K_{0}\vec{x}\neq0,
\end{array}\right.\label{eq:syseqchi}
\end{equation}
with the system found effectively, and thus it is in this manner that
we define $\chi_{q}(\vec{x})$.
\item If $q$ were not a disjunction of atomic formulas, as initially assumed,
write it in conjunctive normal form as $q\equiv q_{1}\lor\dots\lor q_{k}$
and set $\chi_{q}\equiv\chi_{q_{1}}\lor\dots\lor\chi_{q_{k}}$.
\end{itemize}
\item Next, we construct $\tau$. Intuitively, given $q(\vec{x},y,\vec{z})$,
we consider all possible assignments of $y$ of the form $y=\frac{1}{a}\vec{c}\cdot\vec{x}$
(if such an element exists). We will prove below, when establishing
the QETP, that assuming not every element of $\vec{x}$ is null there
is such a combination that satisfies $\chi_{q}(\vec{x},\frac{1}{a}\vec{c}\cdot\vec{x})$.
We will choose one such combination, and set $\tau_{q\varphi}(\vec{x},y)\equiv(ay=\vec{c}\cdot\vec{x}$).
We also need to handle the case where every element of the tuple $\vec{x}$
is zero (which will generally happen if the tuple is empty, or if
it contains a single element that happens to be zero), which is why
we have the designated element $\delta$ -- if every element of $\vec{x}$
is null, we set $\tau(\vec{x},y)\equiv N\Delta(y)$ where $N$ is
large enough to encompass all divisibilities demanded by the tuple
of elements $\vec{z}$ (e.g.\ if $q(y,z)\equiv(z+z=y)$, we would
set $y=2\delta$).

Let's discuss this process more formally. Given $q(\vec{x},y,\vec{z})$:
\begin{itemize}
\item Set $\varphi_{0}(\vec{x})\equiv\land_{i}(x_{i}=0)$ (or, if we are
being unnecessarily precise, $\varphi_{0}(\vec{x})\equiv\land_{i}Q_{0,1}(x_{i})$),
and set $\tau_{q\varphi_{0}}(\vec{x},y)\equiv\varphi_{0}(\vec{x})\land N\Delta(y)$,
where $N$ is large enough to ensure $\forall_{d}\chi_{q}(\vec{0},Nd)$.
\item For every nonzero integer $a$ and tuple of integers $\vec{c}$, set
$\varphi_{a\vec{c}}(\vec{x})\equiv\text{``\ensuremath{\chi_{q}(\vec{x},\frac{1}{a}\vec{c}\cdot\vec{x})}''}$
and $\tau_{q\varphi_{a\vec{c}}}(\vec{x},y)\equiv(ay=\vec{c}\cdot\vec{x})$.
It remains to explain what we mean by $\chi_{q}(\vec{x},\frac{1}{a}\vec{c}\cdot\vec{x})$.

Since $q$ is assumed to be a type, we have by definition that $\chi_{q}(\vec{x},y)$
is a system of equations as in (\ref{eq:syseqchi}). By multiplying
all equations by $a$, we obtain an equivalent system of equations
where every instance of $y$ is being multiplied by $a$. In other
words, we can rewrite $\chi_{q}(\vec{x},y)\Leftrightarrow\psi(\vec{x},ay)$,
and thus the formula ``$\chi_{q}(\vec{x},\frac{1}{a}\vec{c}\cdot\vec{x})$''
may be taken to mean
\[
\text{``\ensuremath{\chi_{q}(\vec{x},\frac{1}{a}\vec{c}\cdot\vec{x})}''}\equiv(a\mid\vec{c}\cdot\vec{x})\land\psi(\vec{x},\vec{c}\cdot\vec{x}).
\]

\end{itemize}
\item Finally, we construct $\etau$. Given $q(\vec{x},y,\vec{z})$ and
$\varphi(\vec{x})$ as one of the above cases, as well as $Q(\vec{x},y,\vec{z},\vec{w})$
extending $q$, we wish to find a quantifier-free expression that
is equivalent to $\exists_{y}\left[\chi_{Q}(\vec{x},y)\land\tau_{q\varphi}(\vec{x},y)\right]$.
Intuitively, the idea is that in either case our formula $\tau$ determines
$y$ exactly (as a function of $\vec{x}$), so we can just ``plug
that $y$ into $\chi_{Q}(\vec{x},y)$''. To do this more precisely,
we split into cases:
\begin{itemize}
\item If $\tau_{q\varphi}(\vec{x},y)\equiv(ay=\vec{c}\cdot\vec{x})$, we
apply the same ``multiply every equation by $a$'' trick to write
a formula that corresponds to ``$\left(a\mid\vec{c}\cdot\vec{x}\right)\land\chi_{Q}(\vec{x},\frac{1}{a}\vec{c}\cdot\vec{x})$''.
The resulting formula is the desired $\etau$.
\item If $\tau_{q\varphi}(\vec{x},y)\equiv\varphi_{0}(\vec{x})\land N\Delta(y)$,
then, intuitively, we know for a tuple satisfying $\tau$ we have
that all $x_{i}$ are zero, and that $y=N\delta$. Thus, the existential
quantifier may be eliminated, and we wish to determine whether $\chi_{Q}(\vec{0},N\delta)$.
This is where it is important that we have a c.e.\ index for the
type of $\delta$ -- and, if we desire a uniform result, that we
have this index without recourse to any oracle. Knowing the c.e.\ index
for the type of $\delta$, we can just consult whether $\chi_{Q}(\vec{0},N\delta)$
holds -- more precisely, there is a computable sequence $\{B_{n}\}_{n\in\N}$
such that
\[
\left\{ \begin{aligned} & G\nvDash\chi_{Q}(\vec{0},N\delta) & \Rightarrow & B_{n}=\bot\text{ for all \ensuremath{n},}\\
 & G\vDash\chi_{Q}(\vec{0},N\delta) & \Rightarrow & B_{n}=\top\text{ for all but finitely many \ensuremath{n}.}
\end{aligned}
\right.
\]
Thus, in this scenario, we set 
\[
\etau_{q\varphi Q}(\vec{x})\equiv\llor_{n}B_{n}.
\]
\end{itemize}
\end{itemize}
Now we verify that the above definitions are witnesses to the QETP.
\begin{enumerate}[label=(\alph*)]
\item ($\Lang$ is finite and closed under negation) Obvious.
\item ($G\vDash\chi_{q}(\vec{x})\leftrightarrow\exists_{\vec{y}}q(\vec{x},\vec{y})$)
By construction.
\item There are two bullet points here. In reverse order:
\begin{itemize}
\item ($G\vDash\chi_{q}(\vec{x})\land\varphi(\vec{x})\rightarrow\exists_{y}\tau_{q\varphi}(\vec{x},y)$)
Direct from construction.
\item (Every tuple $\vec{b}$ of elements of $G$ that satisfies $\chi_{q}(\vec{b})$
satisfies some $\varphi(\vec{b})$ with $(q,\varphi)$ in domain of
$\tau$)
\begin{itemize}
\item If $\vec{b}=\vec{0}$ this is obvious.
\item If some \textbf{$b_{i}\neq0$}: We wish to show that there are $a\in\Z_{\neq0}$
and $\vec{c}\subseteq\Z$ such that $a\mid\vec{c}\cdot\vec{b}$ and
$\chi_{q}(\vec{b},\frac{1}{a}\vec{c}\cdot\vec{b})$. To this effect,
we prove the following related lemma, which obviously implies the
desired claim:
\begin{lemma}
If $q(\vec{x},\vec{y})$ is a quantifier-free type in the smaller
language $\Lang$, and $\vec{b}$ is a tuple in $G$ such that $G\vDash\exists_{\vec{y}}q(\vec{b},\vec{y})$,
and some element of $\vec{b}$ is nonzero, then there exist integer
vectors $\vec{c}_{1},\dots,\vec{c}_{\ell}$ and nonzero integers $a_{1},\dots,a_{\ell}$
such that, for every $i$, $a_{i}\mid\vec{c}_{i}\cdot\vec{b}$, and
$G\vDash q(\vec{b},\frac{1}{a_{1}}\vec{c}_{1}\cdot\vec{b},\dots,\frac{1}{a_{\ell}}\vec{c}_{\ell}\cdot\vec{b})$.
\end{lemma}
\textit{Proof of lemma:} We start by writing $q(\vec{b},\vec{y})$
as a linear system of equations
\[
q(\vec{x},\vec{y})\equiv\left\{ \begin{aligned}M\vec{b} & =N\vec{y},\\
K\vec{b} & \neq L\vec{y}.
\end{aligned}
\right.
\]
Then, as in the construction of $\chi_{q}$, we rewrite this system
in terms of alternate variables $\vec{z}$, which are related to the
variables $\vec{y}$ by the relation $\vec{z}=B\vec{y}$ for an invertible
matrix $B$:
\[
\left\{ \begin{aligned} & M'\vec{b}=D\vec{z},\\
 & K'\vec{b}\neq L'\vec{z},
\end{aligned}
\right.
\]
with $D$ a diagonal matrix padded with zeros. Now, let us look at
the rows of the equation $M'\vec{b}=D\vec{z}$. Some rows correspond
to a zero element in the diagonal of $D$, i.e.\ are of the form
$\vec{m}\cdot\vec{b}=0$. These can be ignored -- since it is known
that $\exists_{y}q(\vec{b},\vec{y})$, these rows simply consist of
true linear relations between the $\vec{b}$, and the corresponding
$z$-variables are free. Let us call them $\vec{z}_{\mathrm{free}}$.
The remaining $z$-variables are determined, corresponding to equations
of the form $\vec{c}_{i}\cdot\vec{b}=a_{i}z_{i}$. As in the construction
of $\chi_{q}$, this means that the corresponding $z$ variables are
determined as $z_{i}=\frac{1}{a_{i}}\vec{c}_{i}\cdot\vec{b}$ (which
in particular precludes the fact that $a_{i}\mid\vec{c}_{i}\cdot\vec{b}$).
Without loss of generality, we substitute the determined variables
by their corresponding expansion in terms of $\vec{b}$ (which may
require clearing denominators), leaving us only to satisfy the relation
$K''\vec{b}\neq L''\vec{z}_{\mathrm{free}}$. Here, one again applies
a similar reasoning to the one used in the construction of $\chi_{q}$,
with a minor twist. Stratifying the equations so that we first determine
$z_{1}$, then $z_{2}$, and so on, we see that at each step $z_{i}$
is only forbidden to take a finite number of values. Then, we take
some \emph{nonzero} $b_{j}$, and set $z_{i}=kb_{j}$ for some large
enough coefficient $k$.

In the end, one concludes that there is a solution to the system of
equations $q(\vec{b},B\vec{z})$ whereby every entry of $\vec{z}$
is a valid rational combination of the entries of $\vec{b}$. Multiplying
this by the integer matrix $B$ yields the desired vector $\vec{y}$,
also consisting solely of valid rational combinations, which serves
as a solution to $q(\vec{b},\vec{y})$.

\hfill{}\textit{(End proof of lemma.)}

\smallskip{}
This completes the proof of Item \ref{enu:qetp2}.
\end{itemize}
\end{itemize}
\item ($G\vDash\exists_{y}\left[\chi_{Q}(\vec{x},y)\land\tau_{q\varphi}(\vec{x},y)\right]\leftrightarrow\etau_{q\varphi Q}(\vec{x})$)
True by construction.
\item ($G\vDash\etau_{q\varphi Q}(\vec{x})\land\tau_{q\varphi}(\vec{x},y)\rightarrow\chi_{Q}(\vec{x},y)$)
Essentially, this is true because $\tau_{q\varphi}(\vec{x},y)$ \emph{always}
determines $y$ uniquely in terms of $\vec{x}$. As such, if it is
the case that there is some $y_{0}$ that satisfies $\chi_{Q}(\vec{x},y_{0})\land\tau_{q\varphi}(\vec{x},y_{0})$
(that is, $\etau_{q\varphi Q}(\vec{x})$ holds), and it is the case
that some specific $y$ satisfies $\tau_{q\varphi}(\vec{x},y)$, then
it must be the case that $y=y_{0}$ and therefore $\chi_{Q}(\vec{x},y)$.
\end{enumerate}

Thus, we are in condition to apply \thmref{main_final} to the $0'$-computable
c.e.-typed copy of $G$ (plus designated element). Since this copy
was obtained uniformly from the original presentation, the result
is uniform.
\end{proof}

\begin{remark}
A useful trick came up in the course of the prior proof: If, perchance,
$\tau_{q\varphi}(\vec{x},y)$ has the property that $y$ is determined
uniquely in terms of $\vec{x}$, then Item \ref{enu:qetp3} of the
QETP comes for free. This will not apply to any of the other examples
present in this article, but it does point at something noteworthy:
It is the job of $\tau$ to determine $y$ as uniquely as possible,
so that \thmref{main_final} ``knows'' what is OK to match $y$
with and what is not. The case where there is a single well-determined
element it can be matched with is optimal, but even in other cases,
we want $\tau$ to pin down the $\Lang'$-type of $y$ as much as
possible (and, indeed, this is in some sense the essence of Item \ref{enu:qetp3}
of the QETP).
\end{remark}

\subsection{\protect\label{subsec:linord}Linear Orders with Successors and Predecessors}
\begin{thm}[{{\cites[VII.26]{montalban_cst}}}]
\label{thm:JIlinear-1}Suppose that $L$ is a linear order for which
every element admits a successor and predecessor, and suppose that
the structure $(L,<,S)$ admits a $0'$-computable copy, where $S(x,y)$
is the successor relation. Then, the structure $(L,<)$ admits a computable
copy.
\end{thm}
\begin{proof}
The most important step is to make a $0'$-computable c.e.-typed copy
of $(L,<,\not<,=,\neq,\{D_{\geq i}\}_{i\in\N})$, where $D_{\geq i}(x,y)$
is the relation ``there exist $x=x_{0}<x_{1}<\dots<x_{i}=y$'',
as this is the bare minimum language necessary for quantifier elimination.
The main difficulty in doing so is that $0'$ does not know enough
about two arbitrary elements to create, in finite time, a c.e.\ index
for which distance relations hold true of them. Thus, we shall need
to create an intermediary structure in which some sort of distance
relations are uniformly c.e.

We know that $(L,<,S)$ is $0'$-computable, and hence there is a
computable process that produces a $(<,S)$-structure, and is allowed
to delete previously added elements, such that the set of elements
that are never deleted make up an isomorphic copy of $(L,<,S)$. We
create an auxiliary computable linear order which shall act as a computable
shadow of $(L,<,S)$, from which we will create a $0'$-computable
c.e.-typed copy of $L$.

Assume without loss of generality that, whenever a real element (i.e.\ one
that will never be deleted) is added to $L$, all elements of $L$
at that moment are themselves real (because when an element is deleted,
you can just delete all elements added since its addition, and re-add
them afterward). We create a new computable linear order, $(L_{+},<)$,
whose elements consist of, for some (but not all) elements $x$ that
were ever added to $L$ (even if later removed), a corresponding element
$x_{+}$ in $L_{+}$, in such a way that if $x$ and $y$ both coexist
in $L$ at the same time, $x<y$ iff $x_{+}<y_{+}$. Here is the process
for the creation of $L_{+}$: Each time an element $x$ is added to
$L$, we imagine an element $x_{+}$in $L_{+}$ such that, for every
$y$ in $L$ at this moment, $x_{+}$ compares to $y_{+}$ as $x$
compares to $y$. Then, we need to determine how $x_{+}$ compares
to the current elements of $L_{+}$ that are not of this form. Transitivity
will force some comparisons. The ones that are not forced can be decided
arbitrarily, e.g.\ say that $x_{+}<z$ for every element with which
the comparison is not forced. Finally, check whether adding this element
will break the requirement that $(L_{+},<)$ is a linear order (e.g.\ because
$x$ is actually a fake element of $L$ that compares to the remaining
elements of $L$ in an obviously stupid way), and importantly, whether
it stands between the image of two elements of $L$ that are currently
thought to be successors. If it does not break these requirement,
we add $x_{+}$ to $L_{+}$. Otherwise, proceed with the next element
added to $L$.

We make a few observations:
\begin{itemize}
\item Note that $L$ embeds into $L_{+}$ via the map $x\mapsto x_{+}$.
The main content of this statement is that for every real $x$ we
do add $x_{+}$ to $L_{+}$. This requires our assumption that, every
time a real element $x$ is added to $L$, all elements currently
in $L$ are also real: This ensures that $x_{+}$ compares to all
other elements of the form $y_{+}$ in a coherent way, and it is a
standard fact about linear orders (amalgamation property) that we
can then add $x_{+}$ to $L_{+}$ in such a way that the resulting
order is a linear order, and indeed the rule ``$x_{+}$ is added
as the leftmost element of the interval it's forced to be in'' is
a uniform way to do so.
\item If $xSy$ are elements of $L$, there are finitely many elements of
$L_{+}$ between $x_{+}$ and $y_{+}$. This is because no element
is added to $L_{+}$ after these two are.
\item In any computable linear order (such as $L_{+}$), the distance relations
$\{D_{\geq i}\}_{i\in\N}$ are uniformly c.e. as they are $\Sigma_{1}$-defined.
As such, there is a procedure that takes as input a finite tuple of
elements from the order $L_{+}$ and outputs a c.e.\ index for their
atomic type in the signature $(<,\not<,=,\neq,\{D_{\geq i}\}_{i\in\N})$.
\end{itemize}
We are now ready to create the $0'$-computable c.e.-typed copy of
$L$. First, we define the $0'$-c.e. subset $L_{+-}\subseteq L_{+}$
as
\[
L_{+-}=\{\,z\in L_{+}\mid\text{there exist real elements \ensuremath{xSy\in L} such that \ensuremath{x_{+}\leq z\leq y_{+}}}\,\}.
\]
Then, we $0'$-enumerate the elements of $L_{+-}$, and with each
new element we also produce a c.e.\ index for the tuple of elements
enumerated up to now, \emph{from the order $(L_{+},<)$.} We claim
that the resulting order is isomorphic (as a linear order) to $L$
and that the predicates $D_{\geq i}$ are interpreted in the correct
way.
\begin{itemize}
\item (The resulting order is isomorphic to $L$): Consider the embedding
$f\colon L\to L_{+-}$ given by $x\mapsto x_{+}$. (We've already
seen that this map is well-defined and embeds $L$ into $L_{+}$,
and from the definition of $L_{+-}$ and the fact that every element
of $L$ admits a successor (or a predecessor) it is clear that for
every $x\in L$ we have $x_{+}\in L_{+-}$.) This is an order-preserving
map with the properties:
\begin{itemize}
\item If $xSy$ then there are finitely many elements between $f(x)$ and
$f(y)$,
\item Every element of $L_{+-}$ is a finite distance away from an element
in the image of $f$.
\end{itemize}
One can argue (and indeed, the standard proof of \thmref{JIlinear-1}
also goes through this fact) that in this case $(L,<)$ and $(L_{+-},<)$
are isomorphic: Pick one representative $x$ from every finite-distance
block from $L$, and map $S^{i}x$ ($i\in\Z)$ to $S^{i}(f(x))$.
The fact that this map -- call it $g$ -- is well-defined requires
some details\footnote{One must verify that every element of $L_{+-}$ has a successor and
a predecessor. This is because every element $z$ of $L_{+-}$ satisfies
$\exists_{x}x_{+}<z<(S^{3}x)_{+}$, and there are finitely many elements
between $x_{+}$ and $(S^{3}x)_{+}$, whereby there must be a smallest
element greater (resp. smaller) than $z$ as it suffices to take the
least (resp. greatest) element in the finite interval $\left]z,(S^{3}x)_{+}\right]$
(resp. $\left[x_{+},z\right[$).}, the fact that $g$ is an embedding is obvious, and the fact that
$g$ is surjective consists of: Every element $z$ of $L_{+-}$ is
a finite distance away from some $y_{+}$, and if $x$ is the representative
of the block of $y$ in $L$, we must have that $z$ is a finite distance
away from $x_{+}$. If $i$ is the signed distance between $z$ and
$x_{+}$, we have $z=g(S^{i}x)$.
\item (The predicates $D_{\geq i}$ are correctly interpreted): For the
sake of this argument, let us denote the predicates induced from $L_{+}$
(and thus the ones we have a c.e.\ index for) by $\{D_{\geq i}^{\mathrm{ind}}\}_{i\in\N}$
and the ``true'' predicates by $\{D_{\geq i}\}_{i\in\N}.$We would
like to show that $L_{+-}\vDash D_{\geq i}^{\mathrm{ind}}(x,y)\leftrightarrow D_{\geq i}(x,y)$.
The implication $\leftarrow$ is obvious, so we focus on the implication
$\rightarrow$, which we prove by contrapositive. Suppose that $x$
and $y$ are elements of $L_{+-}$ such that $\neg D_{\geq i}(x,y)$.
Let $x^{\leftarrow}$ be an element of $L$ such that $(x^{\leftarrow})_{+}\leq x\leq(Sx^{\leftarrow})_{+}$,
and likewise define $y^{\rightarrow}$ such that $(S^{-1}y^{\rightarrow})_{+}\leq y\leq(y^{\rightarrow})_{+}$.
Then, all of the elements in the open interval $\left]x^{\leftarrow},y^{\rightarrow}\right[$
are taken to (distinct) real elements between $x$ and $y$, whence
there must be finitely (and indeed, less than $i$ many) of them.
Thus, the entirety of the interval $\left](x^{\leftarrow})_{+},(y^{\rightarrow})_{+}\right[\subseteq L_{+}$
actually lies in $L_{+-}$, and thus distances are preserved between
any two elements thereof when considering only elements of $L_{+-}$.
Since the distance between $x$ and $y$ is less than $i$ in $L_{+-}$,
it must also be so in $L_{+}$, i.e.\ $\neg D_{\geq i}^{\mathrm{ind}}(x,y)$.
\end{itemize}
Now that we have a $0'$-computable c.e.-typed copy of $(L,<,\not<,=,\neq,\{D_{\geq i}\}_{i\in\N})$,
it suffices to establish that it has the QETP, whereby $(L,<)$ has
a computable copy by \thmref{main_final}:
\begin{itemize}
\item To construct $\chi_{q}$: Given a type $q(\vec{x},\vec{y})$ (see
\rmkref{assumeq}), the demand that $\exists_{\vec{y}}q(\vec{x},\vec{y})$
corresponds to requiring that (the ordering on variables induced by
$q$ is coherent, and that) the $x_{i}$ are in a specific order,
and that the adjacent (in this order) $x_{i}$ admit enough space
between them to add the necessary $y_{j}$. This can obviously be
(computably) written as a positive combination of $D_{\geq n_{i}}(x_{i},x_{j})$.
\item To construct $\tau$: Given a type $q(\vec{x},y,\vec{z})$, we define
$\tau_{q\top}(\vec{x},y)$ as follows. Consider the linear order induced
on the variables $\vec{x}$, $y$, and $\vec{z}$ by $q$.
\begin{itemize}
\item If, in this order, there is $i$ such that $x_{i}<y$, let $i$ be
the largest such WLOG, and let $k$ be the number of variables $z_{j}$
such that $x_{i}<z_{j}<y$. Then, set
\[
\tau_{q\top}(\vec{x},y)\equiv D_{\geq k}(x_{i},y)\land\neg D_{\geq(k+1)}(x_{i},y)\text{ (which is to say \ensuremath{S^{k}(x_{i},y)})}.
\]
\item If there is no such $i$, but there is at least one $x$-variable,
let $x_{i}$ be the least $x$-variable, let $k$ be the number of
$z_{j}$ such that $y<z_{j}<x_{i}$, and set
\[
\tau_{q\top}(\vec{x},y)\equiv D_{\geq k}(y,x_{i})\land\neg D_{\geq(k+1)}(y,x_{i})\text{ (which is to say \ensuremath{S^{k}(y,x_{i})})}.
\]
\item If there are no $x$-variables, set $\tau_{q\top}(y)\equiv\top$.
\end{itemize}
\item To construct $\etau_{q\top Q}$: Depending on the case of the definition
of $\tau_{q\top}$, respectively, and using the notation therein,
\begin{itemize}
\item $\etau_{q\top Q}(\vec{x})$ says ``$\chi_{Q}(\vec{x})$ and $Q$
didn't add any new variable between $y$ and $x_{i}$'',
\item Likewise,
\item $\etau_{q\top Q}\equiv\chi_{Q}$ (applied to the empty tuple).
\end{itemize}
\end{itemize}
The verification of all the relevant properties is straight-forward
and left to the reader.
\end{proof}

\subsection{\protect\label{subsec:trees}A Simple Class of Trees}

We now present a more complete version of a proof previously sketched
in the introduction. As a reminder: If $T$ is a tree and $\sigma$
is a node, we use the term ``parent of $\sigma$'' to mean the unique,
if it exists, node that has $\sigma$ as a child. We use the term
``$a$-th parent of $\sigma$'' to mean, if it exists, the parent
of the parent of ... (repeated $a$ times) of the parent of $\sigma$.
In particular, the $0$-th parent of $\sigma$ is $\sigma$ itself.
\begin{prop}[{{\cites[Proposition 3.8]{strong_jump_inversion}}}]
\label{prop:trees}Let $T$ be an infinite tree in $\omega^{<\omega}$
(i.e.\ a nonempty subset of $\omega^{<\omega}$ closed under prefixes)
such that every node of $T$ is either a leaf or admits infinitely
many children, of which finitely many are leaves. Define on $T$ the
predicates:
\begin{itemize}
\item $\leaf(x)$, meaning $x$ has no children,
\item $\Shat ab(x,y)$, meaning that $x$ admits an $a$-th parent, and
$y$ admits a $b$-th parent, and that these two are the same,
\item $R_{a}(x)$, meaning that $x$ admits the root of the tree as an $a$-th
parent,
\end{itemize}
which have as particular cases the predicates $S(x,y)\equiv\Shat 01(x,y)$,
meaning that $y$ is a child of $x$, and $R_{0}(x)$, meaning that
$x$ is the root.

Then, if the structure $(T,\leaf,S,R_{0})$ admits a $0'$-computable
copy, the structure $(T,S,R_{0})$ admits a computable copy.

\end{prop}
\begin{proof}
First, we note that from a $0'$-computable copy of $(T,\leaf,S,R_{0})$
we can uniformly compute a $0'$-computable c.e.-typed copy of $(T,\neg\leaf,\{\Shat ab,\neg\Shat ab\}_{a,b\in\N},\{R_{a},\neg R_{a}\}_{a\in\N})$,
by enumerating the elements of the copy of $T$ and, for every new
element, finding all its predecessors up to the root and comparing
them to the predecessors of all the elements enumerated up to now,
which allows us to compute (and, crucially, encode as a finite amount
of information) which $\Shat ab$ relations hold between this new
element and the others already enumerated thus far, and the distance
of this new element from the root (to compute $R_{a}$). Thus, it
suffices to show that the structure $T$ satisfies the QETP, over
the signatures
\begin{itemize}
\item $\Lang=(S,\neg S,R_{0},\neg R_{0},=,\neq)$, and
\item $\Lang'$ equals $\Lang$, plus $\neg\leaf$ and all predicates $\Shat ab$,
$\neg\Shat ab$, $R_{a}$, and $\neg R_{a}$.
\end{itemize}
To this effect, we first need to establish a (uniform) quantifier-elimination
result. Thus, suppose that $q(\vec{x},\vec{y})$ is a quantifier-free
formula in the signature $\Lang$, and let us calculate a formula
$\chi_{q}(\vec{x})$ in the signature $\Lang'$ such that $T\vDash\forall_{\vec{x}}(\chi_{q}(\vec{x})\leftrightarrow\exists_{\vec{y}}q(\vec{x},\vec{y}))$.

Suppose that $q(\vec{x},\vec{y})$ is a quantifier-free $\Lang$-type
all of whose elements are distinct (see \rmkref{assumeq}). We wish
to rewrite $\exists_{\vec{y}}q(\vec{x},\vec{y})$ in a quantifier-free
manner. Let us momentarily consider the language augmented with a
partial function symbol $P(x)$, which returns the unique predecessor
of $x$ if it exists. Then, every variable $y_{a}$ such that there
is an instance of $S(y_{a},*)$ in $q$ may be replaced by $P(*)$,
with the added requirement $\neg R_{0}(*)$. If any instance of $S(y_{a},P^{n}(y_{a}))$
is found, abort and set $\chi_{q}\equiv\bot$. Otherwise, we thus
replace $\exists_{\vec{y}}q(\vec{x},\vec{y})$ by a logically equivalent
formula, using the partial function symbol $P$, where the only variables
$y_{a}$ that remain are those that are not demanded to be the predecessor
of any other variable. For each such $y_{a}$, one of two cases occurs:
\begin{enumerate}[label={Case \arabic*.}]
\item  There is some $x_{b}$ such that we now have $x_{b}=P^{m}(y_{a})$,
for some integer $m\geq1$. In this case, we can get rid of the variable
$y_{a}$ while retaining logical equivalence, as long as we add the
demand $\neg\leaf(x_{b})$, as if $x_{b}$ is not a leaf it will have
infinitely many infinite children, which will have infinitely many
infinite children, and so on, so such a $y_{a}$ always exists.
\item The type $q$ does not demand that any predecessor of $y_{a}$ is
any $x_{b}$. In this case, we can get rid of the variable $y_{a}$,
because any (self-consistent) demand we make of $y_{a}$ will always
be satisfied by some deep enough child of the root.
\end{enumerate}
Once this process terminates, we've replaced $\exists_{\vec{y}}q(\vec{x},\vec{y})$
by a logically equivalent positive Boolean combination of: $\neg\leaf(x_{i})$,
$P^{a}(x_{i})=P^{b}(x_{j})$, $P^{a}(x_{i})\neq P^{b}(x_{j})$ (incl.
the case where one of the sides is not well-defined), $R_{0}(P^{a}(x_{i}))$,
and $\neg R_{0}(P^{a}(x_{i}))$. These are, respectively, equivalent
to: $\neg\leaf(x_{i})$, $\Shat ab(x_{i},x_{j})$, $\neg\Shat ab(x_{i},x_{j})$,
$R_{a}(x_{i})$, and $\neg R_{a}(x_{i})$, which yields the desired
formula $\chi_{q}$.

\medskip{}

Now, let us construct $\tau$. Suppose that a quantifier-free $\Lang$-type
$q(\vec{x},y,\vec{z})$ is given. Consider, as in the construction
of $\chi$, eliminating all variables $\vec{z}$ that appear as the
predecessor of another variable. Now, there are two cases:
\begin{enumerate}[label={Case \arabic*.}]
\item There are $x_{a}$, $m$, and $n$ such that $P^{m}(x_{a})=P^{n}(y)$
is part of the data of $q$. In this case, set $\tau_{q\top}(\vec{x},y)\equiv\chi_{q}(\vec{x},y)\land\neg\leaf(y)$.
\item The data of $q$ does not imply that any particular predecessor of
$y$ is equal to any particular predecessor of any $x_{a}$. There
are two sub-cases: Either the data of $q$ allows us to determine
the depth of $y$ precisely, or it only gives us a lower bound for
the depth (because none of the predecessors of $y$ is prescribed
to be the root). Let $N$ be the depth of $y$ or its lower bound,
as the case may be, and set $\tau_{q\top}(\vec{x},y)$ to mean ``$\chi_{q}(\vec{x},y)$
and $\varphi_{n}(\vec{x})$ and $\neg\leaf(y)$ and $y$ is at depth
$N$ and $y$ has no common predecessor with any $x_{a}$ other than
the root''. (As explained in the introduction, this can be expressed
as a $\llorc$ formula.).
\end{enumerate}
Finally, we construct $\etau$. Given quantifier-free $\Lang$-types
$q(\vec{x},y,\vec{z})$ and $Q(\vec{x},y,\vec{z},\vec{w})$, we define
$\etau_{q\varphi Q}(\vec{x})$ depending on which case the definition
of $\tau_{q\varphi}$ falls into:
\begin{enumerate}[label={Case \arabic*.}]
\item If $\tau_{q\top}(\vec{x},y)\equiv\chi_{q}(\vec{x},y)\land\neg\leaf(y)$,
set $\etau_{q\top Q}(\vec{x})\equiv\chi_{Q}(\vec{x})$.
\item If $\tau_{q\top}(\vec{x},y)$ is set to mean ``$\chi_{q}(\vec{x},y)$
and $\neg\leaf(y)$ and $y$ is at depth $N$ and $y$ has no common
predecessor with any $x_{a}$ other than the root'', check whether
$Q$ demands that $y$ has more than $N$ predecessors. If so, set
$\etau_{q\top Q}(\vec{x})\equiv\bot$. If it includes coherent data
about exactly $N$ distinct predecessors of $y$, check if it demands
that the $N$th predecessor is the root, and that none of the other
predecessors are related to any $x_{a}$ other than via the root.
If this is the case, set $\etau_{q\top Q}(\vec{x})\equiv\chi_{Q}(\vec{x})$,
and if not, set $\etau_{q\top Q}(\vec{x})\equiv\bot$. Finally, if
it includes data about less than $N$ predecessors of $y$, check
that none of them are the root and that none of them are related to
any $x_{a}$ other than via the root, and set $\etau_{q\top Q}(\vec{x})\equiv\chi_{Q}(\vec{x})$
(or $\bot$ if the check fails).
\end{enumerate}
This concludes the construction. A tedious model-theoretic argument
will show that these formulas satisfy all the conditions of the QETP,
and thus \thmref{main_final} applies, whence we obtain a computable
copy of the structure $(T,S,R_{0})$.
\end{proof}

\section{Escaping Tennenbaum's Theorem with More Truths}

In this section, we answer the question posed by Pakhomov in \cite{pakhomov}
that was outlined in the introduction: Do there exist theories definitionally
equivalent to ``$\mathsf{PA}$ + all $\Pi_{n}$ truths'' that admit
computable nonstandard models? The answer is yes, but in order to
motivate our construction, we will first briefly sketch Pakhomov's
original construction, so as to highlight the ways in which our construction
builds upon it.

As described in the introduction, we will henceforth be working primarily
with the theory $\ZFniptc$, which consists of the axioms of $\mathsf{ZF}$,
with the removal of the axiom of infinity, and with the addition of
the axiom of transitive closure, which (modulo the remaining axioms)
is equivalent to the axiom stating that every set is in some level
$V_{\alpha}$ of the von~Neumann hierarchy.

\subsection{\protect\label{subsec:pakhomov_original}Pakhomov's Original Construction}

As explained in the introduction, Pakhomov defined in $\ZFniptc$
a ternary predicate $S(x,y,z)$ by transfinite recursion on the von~Neumann
hierarchy, which provably satisfies the rule $x\in y\leftrightarrow\forall_{z}S(x,y,z)$,
whence the theory of $S$ is definitionally equivalent to $\ZFniptc$,
and which satisfies certain properties that make the theory of $(\in,S)$
amenable to a jump-inversion-type result, whereby a computable nonstandard
model of the theory of $S$ follows. Let us briefly sketch Pakhomov's
definition of $S$ within $\ZFniptc$. Our definition is not exactly
the same as Pakhomov's, but differs only in inessential ways.

We define an ascending sequence of relations $S_{\alpha}$ on $V_{q\alpha}$,\footnote{Using the convention on ordinal multiplication that satisfies $\beta\sup\alpha_{i}=\sup(\beta\alpha_{i})$.}
where $q$ is a large enough finite ordinal to be determined shortly
(Pakhomov's construction sets $q=6$, while our choice of construction,
which is slightly more inefficient in an attempt to make it easier
to parse, shall set $q=10$), in such a way that each $S_{\alpha}$
agrees with every other $S_{\beta}$ where mutually defined. This
relation is uniquely determined on limit ordinals, as $V_{q\sup\alpha_{i}}=\cup V_{q\alpha_{i}}$,
so it suffices to describe the successor step. As such, let us define
$S_{\alpha+1}$ in terms of $S_{\alpha}$.

Given $\alpha$, a set $A\subseteq V_{q\alpha}$, and $SA\subseteq(A\cup\{q\alpha\})^{3}$,
define an element $w(\alpha,A,SA)$ satisfying the following conditions:
\begin{itemize}
\item $w(\alpha,A,SA)\in V_{q(\alpha+1)}\setminus V_{q\alpha}$,
\item The operation $(\alpha,A,SA)\mapsto w(\alpha,A,SA)$ is injective
as a class function $V\times V\times V\to V$, and
\item $w(\alpha,A,SA)$ is not $\in$-related to any element of $A$.
\end{itemize}
An example of such a function is $w(\alpha,A,SA):=(q\alpha,(A,SA))$,
assuming the usage of Kuratowski pairs $(x,y):=\{\{x\},\{x,y\}\}$,
which will immediately satisfy the second and third bullet points.
To ensure the first bullet point is where we are compelled to choose
$q=10$, as follows:
\begin{itemize}
\item Suppose $A\subseteq V_{q\alpha}$. Then, $A\in V_{q\alpha+1}$.
\item Thus, $A\cup\{q\alpha\}\subseteq V_{q\alpha+1}$, and any pair of
its elements is thereby in $V_{q\alpha+3}$, whence any triplet is
in $V_{q\alpha+5}$.
\item As a consequence, we have $SA\subseteq V_{q\alpha+5}$, and so $SA\in V_{q\alpha+6}$.
Thus, we conclude $w(\alpha,A,SA)\in V_{q\alpha+10}$, and so we set
$q=10$ as to obtain $w(\alpha,A,SA)\in V_{q(\alpha+1)}$.
\item We also have that $w=w(\alpha,A,SA)$ is not in $V_{q\alpha}$ because
$V_{q\alpha}$ is transitive and does not contain $q\alpha$, while
$w$ does.
\end{itemize}
The motivation for defining $w(\alpha,A,SA)$ is as follows: In the
course of constructing a computable nonstandard model of the theory
of $S$, one could consider starting with a $0'$-computable nonstandard
model $D$ of $(\in,S)$, and, when made to remove an element that
was added by mistake, we instead change our mind about which element
of $D$ it represents. This requires that any commitments we've already
made about $S$ be preserved, and since we have no \textit{a priori}
control about what those wrong commitments look like, we need to ensure:
Given a few elements $a_{1},\dots,a_{n}$ of $D$, and given a hypothetical
element which is $S$-related to those elements in some hard-to-control
manner, there exists a true element within $D$ that has those relations.
This element is precisely the one obtained from the $w$ function,
where $A$ is the set $\{a_{1,}\dots,a_{n}\}$, and $SA$ encodes
the $S$-relations between our hypothetical element and the elements
of $A$.

Now, we could not simply allow any and all possible $S$-relations
to hold, or otherwise the theory of $S$ would simply be the theory
of the random $3$-hypergraph, which is far too homogeneous to be
definitionally equivalent to set theory. This is where the following
definition (which takes place inside $\ZFniptc$) comes in.

We say that a pair $(A,SA)$ is \emph{good} (for the ordinal $\alpha$)
if $SA$ is a set of triples of elements from $A\cup\{q\alpha\}$
and the following two conditions hold:
\begin{itemize}
\item $SA$ agrees with $S_{\alpha}$ on $A$, in the sense that $SA\cap A^{3}=\{(a,b,c)\in A^{3}\mid S_{\alpha}(a,b,c)\}$,
and
\item For every $a,b\in A$ with $a\in b$, we have $(a,b,q\alpha)\in SA$.
By induction on $\alpha$, this turns out to be equivalent to the
stronger demand that $\forall_{a,b\in A}a\in b\rightarrow\forall_{z\in A\cup\{q\alpha\}}(a,b,z)\in SA$.
\end{itemize}
We are now ready to define $S_{\alpha+1}$. Given a triplet $(a,b,c)\in V_{q(\alpha+1)}^{3}$,
\begin{itemize}
\item If all elements of the triplet are in $V_{q\alpha}$, set $S_{\alpha+1}(a,b,c)\equiv S_{\alpha}(a,b,c)$,
\item If, for some choice of $A$ and $SA$ that is good for $\alpha$,
every element of the triplet is either in $A$ or is equal to $w=w(\alpha,A,SA)$,
set $S_{\alpha+1}(a,b,c)$ to hold true if, and only if, upon replacing
every instance of $w$ in the triplet $(a,b,c)$ by $q\alpha$, the
resulting triplet is in $SA$. For example, if $a,b\in A$ and $c=w(\alpha,A,SA)$,
set $S_{\alpha+1}(a,b,c)\equiv[(a,b,q\alpha)\in SA]$.

In the sequence, when such a choice of $\alpha$, $A$, and $SA$
is clear from context, denote by $a^{*}$ the operation of replacing
$q\alpha$ by $w(\alpha,A,SA)$, and $a_{*}$ the inverse operation.
Thus, the definition of $S_{\alpha+1}$ in this case can be reworded
as: $S_{\alpha+1}(a,b,c)\equiv[(a_{*},b_{*},c_{*})\in SA]$ or, equivalently,
$S_{\alpha+1}(a^{*},b^{*},c^{*})\equiv[(a,b,c)\in SA]$.
\item For every other triplet from $V_{q(\alpha+1)}$, set $S_{\alpha+1}(a,b,c)$
to hold true.
\end{itemize}
With this, we define the $S$ relation within $\ZFniptc$ as follows:
Given $a$, $b$, and $c$, it holds by the axiom of transitive closure
that all three lie in some common $V_{\alpha}$. Then, evaluate $S(a,b,c)\equiv S_{\alpha}(a,b,c)$.

As originally pointed out by Pakhomov, the main defining properties
of the relation $S$ are the following two lemmas:
\begin{lemma}
Provably in $\ZFniptc$: Given sets $x$ and $y$, $x\in y$ iff $\forall_{z}S(x,y,z)$.
\end{lemma}
\begin{proof}
Suppose $x\in y$, but also that there is $z$ such that $\neg S(x,y,z)$.
Let $\alpha$ be the least ordinal for which there exist $x,y,z\in V_{q\alpha}$
with $x\in y$ and $\neg S_{\alpha}(x,y,z)$. The only way for this
to happen is if, for this value of $\alpha$, there is a good choice
of $A$ and $SA$ for which $(x_{*},y_{*},z_{*})\notin SA$. Because
$\alpha$ is minimal, at least one of $x_{*}$, $y_{*}$, or $z_{*}$
must be $q\alpha$, and since $x\in y$ it cannot be $x_{*}$ nor
$y_{*}$. However, since $x_{*}\in y_{*}$ and $z_{*}=q\alpha$ we
must then have $(x_{*},y_{*},z_{*})\in SA$, a contradiction. This
proves that $x\in y$ implies $\forall_{z}S(x,y,z)$.

For the other direction, assume $x\notin y$. Pick a large enough
value of $\alpha$ such that $A=\{x,y\}$ is contained in $V_{q\alpha}$,
and consider the set $SA=\{(y,x,q\alpha)\}\cup(S_{\alpha}\cap A^{3})$.
It can be checked that the pair $(A,SA)$ is good (in fact, if $y\notin x$
we could even set $SA=S_{\alpha}\cap A^{3}$). As such, for $w=w(\alpha,A,SA)$,
we shall have $S_{\alpha+1}(x,y,w)$ iff $(x,y,q\alpha)\in SA$, which
is not the case. Thus, $\neg S_{\alpha}(x,y,w)$, and so $\exists_{z}\neg S(x,y,z)$
as desired.
\end{proof}
\begin{lemma}
\label{lem:sflex}Let $M$ be a model of $\ZFniptc$. Let $A$ be
a finite\footnote{Here we mean finite from the perspective of the metatheory, but indeed
this result holds also from within the theory, basically by definition.} subset of $M$, and suppose that we wish to find an element $w$
that $S$-relates to all elements of $A$ in a prescribed manner.
So long as the prescription satisfies the rule: ``For all $a,b\in A$
with $a\in b$, we have $S(a,b,w)$'', there is in fact some $w\in M\setminus A$,
not $\in$-related to any element of $A$, satisfying this prescription.
\end{lemma}
\begin{proof}
This follows basically from the definition of the $S$ relation, together
with the fact that every finite set (from the perspective of the metatheory)
is represented in $M$.
\end{proof}
The above two properties suffice to establish the following jump inversion
theorem, though we'll forego the proof, as \thmref{gen_pakhomov_jump_inversion}
completely subsumes it.
\begin{thm}
\label{thm:basic_pakhomov_jump_inversion}If $D$ is a $0'$-computable
model of all true facts about $(\in,S)$ in $\ZFniptc$, the reduct
$D\upto S$ admits a computable copy.
\end{thm}
\begin{remark}
\label{rmk:sjipakhomov1}\thmref{basic_pakhomov_jump_inversion} is
(up to trivial modification) an example of a Strong Jump Inversion
result in the strong sense of \rmkref{jumpinversion}, because the
structure $(D,\notin,S)$ is a structural jump of $(D,S)$. Indeed,
from the fact that $x\in y\leftrightarrow\forall_{z}S(x,y,z)$ we
conclude that $\notin$ is r.i.c.e.\ in $(D,S)$, and we will establish
in the proof of \thmref{gen_pakhomov_jump_inversion} the following
quantifier elimination result: Any formula $\exists_{\vec{y}}q(\vec{x},\vec{y})$
in the signature containing only $S$ is equivalent in $D$ to a first-order
quantifier-free formula $Q(\vec{x})$ in the signature $(\notin,S)$
containing only positive instances of $\notin$. Thus, with access
to an oracle from $0'$ we can uniformly turn any $\Sigma_{1}^{c}(S)$
formula into an equivalent finitary quantifier-free formula in $(\notin,S)$,
which suffices to show that $(D,\notin,S)$ is a structural jump of
$(D,S)$.
\end{remark}
\thmref{basic_pakhomov_jump_inversion} is enough to establish the
following results of Pakhomov's original paper:
\begin{corollary}[Pakhomov \cite{pakhomov}]
\label{cor:cet0}Every consistent c.e.\  extension $T$ of the theory
of $S$ in $\ZFniptc$ admits a computable model.
\end{corollary}
\begin{proof}
The theory $\ZFniptc$ is c.e., and so the theory $T'$ consisting
of $T$, $\ZFniptc$, and the definition of $S$ in terms of $\in$,
is also c.e. Thus, there is a $0'$-computable complete and consistent
extension $T''$ of $T'$, and by applying the relativized computable
completeness theorem (see \cite{harizanov}), there is a $0'$-computable
model of $T''$ and hence of the theory of $(\in,S)$. We then apply
\thmref{basic_pakhomov_jump_inversion} to this model to obtain a
computable copy of its reduct to $S$, which is a model of $T$.
\end{proof}
\begin{corollary}[Pakhomov \cite{pakhomov}]
\label{cor:nonstd}There is a computable nonstandard model of a theory
definitionally equivalent to $\mathsf{PA}$.
\end{corollary}
\begin{proof}
Recall that $\ZFfinptc$ (that is, $\ZFniptc$ plus the negation of
the axiom of infinity) is definitionally equivalent to $\mathsf{PA}$
(see \cite{kaye_wong} for details), and hence the theory of $S$
in $\ZFfinptc$ is as well. Thus, let $T$ be a theory containing
the theory of $S$, plus the translation to the $S$-language of a
statement consistent with $\mathsf{PA}$ not in true arithmetic. By
\corref{cet0}, there is a computable model of $T$.
\end{proof}
The proof of \corref{nonstd} presented above is very close to Pakhomov's
original proof. We present an alternate proof, which is not directly
possible as stated using only Pakhomov's results.
\begin{proof}[\corref{nonstd}]
Let $T$ be the theory in the language of arithmetic, plus an added
constant $c$, containing the theory of $\mathsf{PA}$, plus sentences
saying that $c$ is nonstandard: $c\neq0$, $c\neq S0$, $c\neq SS0$,
and so on. Let $T'$ be the theory $T$, plus predicates $S$ and
$\in$, and axioms defining $S$ and $\in$ in terms of the arithmetical
symbols. Since $T'$ is a c.e.\ theory, it admits a $0'$-computable
model $M$, whose reduct to $(\mathord{\in},S)$ satisfies the assumptions
of \corref{cet0}, and so its reduct to $S$ admits a computable copy.
\end{proof}

\subsection{Generalizing Pakhomov's Construction}

In Pakhomov's original construction, as explained in Section \ref{subsec:pakhomov_original},
the predicate $\in$ is performing two distinct tasks. On the one
hand, $S$ is defined to mean ``set inclusion, with a third argument
that represents a witness in the case where inclusion does not hold'',
with the `witness' part playing a crucial role by being flexible
enough to turn $0'$-computable models into computable models. On
the other hand, the predicate $\in$ is providing the ``definitional
flexibility'' required to allow the entire construction of $S$ to
take place, as we need the theory we're working in to be powerful
enough to quantify over possible prescriptions of $S$ for a hypothetical
extra element, in order to ensure that an element satisfying this
prescription does exist. Upon understanding these roles as separate,
one gets the idea that the construction could be applied to other
predicates, or even applied to itself recursively.

There are a few technical hitches in the way of this idea, most notably
the fact that an implicit ingredient in the construction of $S$ is
that the $\in$-relation is ``rare'', making it easy to construct
elements that are $\in$-incomparable to previous elements, while
conversely the $S$ relation is ``common''. This foreshadows the
idea that, in order to generalize to higher-order ``$S$ relations'',
we will need to introduce a few negations (see \lemref{sneg}).

Another technical matter regards the coefficient $q$ described in
Section \ref{subsec:pakhomov_original}. If we adopt the same exact
strategy described therein, we would need to increase our coefficient
$q$ the more ``iterations of $S$'' we would like to introduce.
By itself, this is not necessarily an issue, but we've found that
it simplifies the definition to adopt a slightly different perspective.
In Section \ref{subsec:pakhomov_original}, we imagine that, for a
given stage $q\alpha$ of the von~Neumann hierarchy, we look at all
possible ways that a new element relates via $S$ to the elements
of $V_{q\alpha}$, and designate someone realizing these possible
relations in $V_{q(\alpha+1)}$. If, instead, we exercise some patience,
we can instead consider that at every stage $V_{\alpha}$ of the von~Neumann
hierarchy, for every possible way for an element to $S$-relate to
the elements of $V_{\alpha}$ we ``schedule'' an element to relate
to them in this manner, though this element might only be actually
added, say, in $V_{\alpha+100}$. In our view, this adds some versatility
and simplifies the construction.

We are now ready to state and prove:
\begin{thm}
\label{thm:nestedpakhomov}There is a nested sequence of consistent
c.e.\ theories $\ZFniptc\subseteq T^{0}\subseteq T^{1}\subseteq T^{2}\subseteq\cdots$
satisfying the following properties:
\begin{itemize}
\item Each $T^{n}$ is in the language containing $\in$ and predicates
$S^{0}$, $\dots$, $S^{n}$, with each $S^{i}$ being an $(i+2)$-ary
predicate symbol,
\item All of these extensions are conservative, in the sense that they contain
no additional theorems in the predicate $\in$,
\item The predicates $\in$ and $S^{n}$ are definable in terms of the other
within $T^{n}$, and
\item Given a $0'$-computable model $D$ of $T^{n}$ restricted to the
language containing $S^{i}$, $\dots$, and $S^{n}$, there is a computable
copy $M$ of $D\upto(S^{i+1},\dots,S^{n})$.
\end{itemize}
The last item relativizes: Given an $X'$-computable model of $T^{n}$
in the last $k$ predicates, there is an $X$-computable copy of its
reduct to the last $k-1$ predicates.
\end{thm}
We dedicate the remainder of this section to the proof of \thmref{nestedpakhomov}.
We start by defining the sequence $\{T_{n}\}_{n\in\N}$ inductively.
For the base case, we set $T^{0}$ to be $\ZFniptc$, plus the axiom
$\forall_{x,y}[S^{0}(x,y)\leftrightarrow(x\notin y)]$. It trivially
(and, in the case of the last bullet point, almost vacuously) satisfies
all required conditions. Now, we suppose that $T^{n}$ has already
been defined, and set out to define $T^{n+1}$.

Within $T^{n}$, we define the ordinal-indexed sequence of predicates
$S_{\alpha}^{n+1}(\vec{x},y)$, where $\vec{x}$ is an $(n+1)$-uple
of variables, in a way similar to the definition of $S_{\alpha}(x,y,z)$
in Section \ref{subsec:pakhomov_original}. The big picture is the
same: We define $S_{\alpha}^{n+1}$ inductively as a compatible sequence
of relations on $V_{\alpha}$ (note the absence of the constant $q$)
with $S_{\lambda}^{n+1}$ having an obvious definition for limit ordinals
$\lambda$, so it suffices to describe the successor step. To assist
us in that regard, we define the following sequence of class functions:

Given sets $A$, $mA$,\footnote{The set $mA$ is not present in the original construction of $S$,
and it is technically not necessary here, but it will save us trouble.
It should be seen as encoding some ``metadata'' on the element $w$
we are about to define, which is essentially used inductively to give
flexibility to higher-order $S^{N}$ predicates.} an ordinal $\alpha$, and sets $S^{i}A\subseteq(A\cup\{\alpha\})^{i+2}$
for $i=0,\dots,n+1$,\footnote{Here, the expression $S^{i}A$ is merely notation (where $i$ is a
finite ordinal and $A$ is a set), intended to evoke the idea that
$S^{i}A$ is defining the relation $S^{i}$ between the elements of
$A$ and a new element. Note that $S^{0}A$ is not the same as $A$.} define an element $w^{n}(\alpha,mA,A,SA,\dots,S^{n}A)$ inductively
in $n$ as follows. We assume the usage of Kuratowski pairs throughout.
\begin{itemize}
\item $w^{0}(\alpha,mA,A)=(\alpha,(A,mA))$, and
\item $w^{n+1}(\alpha,mA,A,S^{0}A,S^{1}A,\dots,S^{n}A,S^{n+1}A)=w^{n}(\alpha,(S^{n+1}A,mA),A,S^{0}A,\dots,S^{n}A)$.
\end{itemize}
The main defining features of this sequence of class functions are:
\begin{itemize}
\item For every $n$, $(\alpha,mA,A,S^{0}A,\dots,S^{n}A)\mapsto w^{n}(\alpha,mA,A,S^{0}A,\dots,S^{n}A)$
is an injective definable class function, and
\item $w^{n}(\alpha,mA,A,S^{0}A,\dots,S^{n}A)$ is \emph{not} in $V_{\alpha}$.
\end{itemize}
Let us go back to defining $S^{n+1}$ inductively, assuming that we
have already defined $S^{0}$ up to $S^{n}$, for $n\geq0$. Assuming
that $S_{\alpha}^{n+1}$ is already defined on tuples of elements
of $V_{\alpha}$, we define a ``good sequence'' $(\alpha,mA,A,S^{0}A,\dots,S^{n+1}A)$
as one that satisfies the following properties:
\begin{itemize}
\item $A$ is a subset of $V_{\alpha}$,
\item For $i=0,\dots,n+1$, the set $S^{i}A$ is a subset of $(A\cup\{\alpha\})^{i+2}$,
\item For $i=0,\dots,n$, the set $S^{i}A\cap A^{i+2}$ coincides with the
relation $S^{i}$ on $A$,
\item The set $S^{n+1}A\cap A^{n+2}$ coincides with the relation $S_{\alpha}^{n+1}$
on $A$,
\item The set $S^{0}A$ contains all pairs $(a,\alpha)$, $(\alpha,a)$,
and $(\alpha,\alpha$), with $a\in A$.
\item For $i=0,\dots,n$, the following condition holds: For every $(i+3)$-uple
$(\vec{a},b)$ of elements of $A\cup\{\alpha\}$, if $\vec{a}\notin S^{i}A$,
then $(\vec{a},b)\in S^{i+1}A$.
\end{itemize}
We are now ready to define $S_{\alpha+1}^{n+1}$. Given an $(n+3)$-uple
$\vec{x}$ of elements in $V_{\alpha+1}$,
\begin{itemize}
\item If all elements of the tuple are in $V_{\alpha}$, set $S_{\alpha+1}^{n+1}(\vec{x})\equiv S_{\alpha}^{n+1}(\vec{x})$,
\item If there is a good sequence $(\beta,mA,A,S^{0}A,\dots,S^{n+1}A)$,
for some $\beta\leq\alpha$, such that the element $w=w^{n+1}(\beta,mA,A,S^{0}A,\dots,S^{n+1}A)$
is in $V_{\alpha+1}$, and every entry of $\vec{x}$ is an element
of $A\cup\{w\}$, set $S_{\alpha+1}^{n+1}(\vec{x})$ to hold true
if, and only if, upon replacing every instance of $w$ in the tuple
$\vec{x}$ by $\beta$, the resulting tuple is in $S^{n+1}A$.

In the sequence, when such a choice of good sequence is clear from
context, denote by $\vec{x}^{*}$ the operation of replacing every
entry of $\vec{x}$ that equals $\beta$ by $w$, and let $\vec{x}_{*}$
denote the inverse operation. Thus, the definition of $S_{\alpha+1}^{n+1}$
in this case can be reworded as: $S_{\alpha+1}^{n+1}(\vec{x})\equiv[\vec{x}^{*}\in S^{n+1}A]$.

Remark: Note that, by definition, if $(\beta,mA,A,S^{0}A,\dots,S^{n+1}A)$
is a good sequence and $n\geq0$, we also have that $(\beta,(S^{n+1}A,mA),A,S^{0}A,\dots,S^{n}A)$
is also a good sequence (for a smaller value of $n$), and thus the
element $w=w^{n+1}(\beta,mA,A,S^{0}A,\dots,S^{n+1})=w^{n}(\beta,(S^{n+1}A,mA),A,S^{0}A,\dots,S^{n}A)$
will also satisfy: $S^{n}(\vec{x})\equiv[\vec{x}^{*}\in S^{n}A]$
for all tuples of elements of $A\cup\{w\}$, and inductively for lower
indices. The case of $S^{0}$ is an edge case, but the definition
of good sequence still ensures that $S^{0}(x,y)\equiv[(x,y)^{*}\in S^{0}A]$.
\item For every other triplet of elements from $V_{\alpha+1}$, set $S_{\alpha+1}^{n+1}(\vec{x})$
to hold true.
\end{itemize}
Note that the first and second item in the definition do not contradict
each other: If a tuple $\vec{x}$ fits both bullet points, then either
all its elements are in $A$, in which case the definition of good
sequence ensures that there is agreement between both possible definitions
of $S_{\alpha+1}^{n+1}(\vec{x})$, or one of its elements is $w=w^{n+1}(\beta,mA,A,\dots,S^{n+1}A)$,
in which case $S_{\gamma}^{n+1}(\vec{x})$ would have been defined
to agree with the second bullet point for $\gamma$ the smallest ordinal
such that $w\in V_{\gamma}$.

This induces a well-defined relation $S^{n+1}$, whose main defining
properties are as follows:
\begin{lemma}
\label{lem:sneg}Provably in $\ZFniptc$, $\neg S^{n}(\vec{x})$ iff
$\forall_{y}S^{n+1}(\vec{x},y)$.
\end{lemma}
\begin{lemma}
\label{lem:sflex_general}Let $M$ be a model of $\ZFniptc$. Let
$A$ be a finite subset of $M$, and suppose that we wish to find
an element $w$ that relates to all elements of $A$ with regard to
the predicates $S^{0},\dots,S^{n+1}$ in a prescribed manner. The
following is a sufficient condition for there to exist an element
$w$ of $M$ satisfying this prescription: For every $0\leq i\leq n$,
the prescription satisfies the rules
\begin{itemize}
\item The prescription agrees with the pre-existing relations on tuples
that don't contain any instance of $w$,
\item For all $x\in A\cup\{w\}$, it is prescribed that $S^{0}(x,w)$ and
$S^{0}(w,x)$.
\item For all $(\vec{x},y)\in(A\cup\{w\})^{i+2}$ such that $\neg S^{i}(\vec{x})$
is prescribed, it is also prescribed that $S^{i+1}(\vec{x},y)$.
\end{itemize}
These two properties are enough to establish:
\end{lemma}
\begin{thm}
\label{thm:gen_pakhomov_jump_inversion}Let $D$ be an $X'$-computable
model of $T^{n}$ in the predicates $S^{i},\dots,S^{n}$, with $i<n$.
Then, the reduct $D\upto(S^{i+1},\dots,S^{n})$ admits an $X$-computable
copy.
\end{thm}
\begin{remark}
The same reasoning used in \rmkref{sjipakhomov1} will show that \thmref{gen_pakhomov_jump_inversion}
is an example of a Strong Jump Inversion result in the strong sense
of \rmkref{jumpinversion}. More precisely, that reasoning shows that
the structure $(D,S^{i},\dots,S^{n})$ is a structural jump of $(D,S^{i+1},\dots,S^{n})$.
\end{remark}
\begin{proof}[\thmref{gen_pakhomov_jump_inversion}]
We will apply \thmref{main_final} to $D$ which, as an $X'$-computable
model over a finite language, is also an $X'$-computable $X$-c.e.-typed
model. We need only establish that $D$ has the QETP over $\Lang=(S^{i+1},\dots,S^{n})$
(plus negations) and $\Lang'=(S^{i},\dots,S^{n})$ (plus negations).
\global\long\def\abs#1{\left|#1\right|}%

We construct $\chi$, $\tau$, and $\etau$:
\begin{itemize}
\item First, we construct $\chi_{q}$. We assume without loss of generality
that $q$ is a quantifier-free type all whose variables are distinct
(see \rmkref{assumeq}).

We start by considering the case where $q(\vec{x},y)$ only has one
$y$-variable. In this case, the formula $\exists_{y}q(\vec{x},y)$
is true if, and only if, there exists an element $y$ satisfying a
certain prescription of predicates $S^{i+1},\dots,S^{n}$. We show
that the following three statements are equivalent:
\begin{itemize}
\item $P(\vec{x})\equiv\exists_{y}q(\vec{x},y)$
\item $Q(\vec{x})$: For every instance of $\neg S^{j+1}(\vec{z}_{0},z_{1})$
in $q(\vec{x},y)$:
\begin{itemize}
\item If $j>i$, we require that $S^{j}(\vec{z}_{0})$ be in $q$, and\footnote{Note: This demand is happening outside of first-order logic. Rather,
it's a check that we metatheoretically make of $q(\vec{x},y)$, and
if this check is failed, we set $Q(\vec{x})\equiv\bot$.}
\item If $j=i$, $z_{1}\equiv y$, and $\vec{z}_{0}$ is a subtuple of $\vec{x}$,
we require $S^{i}(\vec{z}_{0})$,
\end{itemize}
Moreover, we also demand $(q\upto\vec{x})(\vec{x})$, i.e.\ whatever
demand $q$ makes of the tuple $\vec{x}$, we also demand it.
\item $R(\vec{x})\equiv$ There is a $y$ such that $q(\vec{x},y)$ and
the relations $S^{0}$ to $S^{i}$ hold for all tuples including elements
from $(\vec{x},y)$ and at least one instance of $y$$)$.
\end{itemize}
Here is the proof. First, $P(\vec{x})\rightarrow Q(\vec{x})$ follows
from the fact that $\forall_{\vec{z}_{0}}(\exists_{z_{1}}\neg S^{j+1}(\vec{z}_{0},z_{1}))\rightarrow S^{j}(\vec{z}_{0})$
(and the assumption that $q$ is a type). Second, $Q(\vec{x})\rightarrow R(\vec{x})$
follows by \lemref{sflex_general} as follows: $R(\vec{x})$ is true
if we can find an element $w$ that relates to the elements of $\vec{x}$
in a certain manner. This manner consists of using $q$ to dictate
$S^{i+1}$ to $S^{n}$, and setting all $S^{0}$ to $S^{i}$ as true
for tuples containing $w$ and whatever-it-already-is for tuples that
do not. Finally, $R(\vec{x})$ obviously implies $P(\vec{x})$.

Now, still in the case where there is only one $y$-variable, we set
$\chi_{q}(\vec{x})\equiv Q(\vec{x})$ as above. We will now handle
the case where there are two $y$-variables, say $y_{1}$ and $y_{2}$,
and an obvious induction will handle the general case. We will use
the formulas $P$, $Q$, and $R$ constructed above for different
formulas $q$. We will use a subscript to indicate which formula $q$
is being used.

Let $q(\vec{x},y_{1},y_{2})$ be a quantifier-free $\Lang$-type.
Then, let us look at $\chi_{q}(\vec{x},y_{1})$. As above, it's either
$\bot$ (if the first bullet point of the definition of $Q_{q}(\vec{x},y_{1})$
fails), in which case the quantifier elimination is trivial, or otherwise
it consists of $q\upto(\vec{x},y_{1})$, together with a few \emph{positive}
requirements of type $S^{i}(\vec{z}_{0})$ with $\vec{z}_{0}$ a subtuple
of $(\vec{x},y_{1})$. Let $q'(\vec{x},y_{1})$ be $q\upto(\vec{x},y_{1})$
together with those requirements that include $y_{1}$, and $q^{*}(\vec{x})$
be the remainder of these requirements. Note that
\[
\exists_{y_{1}}\exists_{y_{2}}q(\vec{x},y_{1},y_{2})\leftrightarrow\exists_{y_{1}}\chi_{q}(\vec{x},y_{1})\leftrightarrow q^{*}(\vec{x})\land\exists_{y_{1}}q'(\vec{x},y_{1}),
\]
and crucially that, since all $S^{i}$-requirements made by $q'$
are positive, $\exists_{y_{1}}q'(\vec{x},y_{1})$ is logically between
$P_{q'}(\vec{x})$ and $R_{q'}(\vec{x})$, and hence equivalent to
both, and hence equivalent to $Q_{q'}(\vec{x})$. Thus, we conclude
$\exists_{\vec{y}}q(\vec{x},\vec{y})\leftrightarrow q^{*}(\vec{x})\land Q_{q'}(\vec{x})$,
and so we define the latter as $\chi_{q}$ in this scenario. The induction
continues in a similar fashion, in the event that there are more $y$-variables,
with the essential point being that every $S^{i}$-demand being made
at any step is positive.
\item Now, we construct $\tau$. Given $q(\vec{x},y,\vec{z})$, we set $\tau_{q\top}(\vec{x},y)$
to be the demand that every $S^{i}$-relation including $y$ at least
once holds, and that $\chi_{q}(\vec{x})$ holds.
\item Finally, $\etau_{q\top Q}(\vec{x})$ is simply $\chi_{Q}(\vec{x})$.
\end{itemize}
The proof that these formulas are witnesses to the QETP is simple
and follows the same $P$, $Q$, $R$ reasoning used to establish
that $\chi_{q}(\vec{x})\leftrightarrow\exists_{y}q(\vec{x},y)$.
\end{proof}
\thmref{nestedpakhomov} immediately follows.

As explained in the introduction, we obtain as a corollary:
\begin{thm}
\label{thm:pakhomov_answer}For every $n$, there is a theory definitionally
equivalent to ``$\mathsf{PA}$ plus all true $\Pi_{n}$ sentences''
that admits a computable non-standard model.
\end{thm}

\section{Further Questions}

In \thmref{gen_pakhomov_jump_inversion}, we were able to get the
complexity of models down by ``one jump'', turning a $0'$-computable
model of a c.e.\ theory into a computable model. However, a c.e.\ theory
admits models with complexity far below $0'$. We wonder if this could
be exploited to obtain stronger results.
\begin{question}
Is there something else that can be said beyond \thmref{gen_pakhomov_jump_inversion}
and \thmref{main_final} in the assumption that we are given, say,
a low model instead of a merely $0'$-computable model?
\end{question}
\medskip{}

As always with this kind of work, one may wonder whether there exist
less artificial theories or predicates that could have filled the
role of our predicates $S^{1}$, $S^{2}$, etc. In keeping with \cite{lutz_walsh_tennenbaum},
we define a `natural' theory or predicate as one that has been studied
by mathematicians who are not logicians.
\begin{question}
\label{question:natural1}Does there exist a natural theory that is
definitionally equivalent to $\mathsf{PA}$ and admits a non-standard
computable model?
\end{question}
As a dual to \questionref{natural1}, once one has been made aware
of the role of the signature in Tennenbaum's theorem, one may think
to study more deeply the impact that the signature has on the complexity
of nonstandard models of $\mathsf{PA}$. For example: It is a standard
theorem (or definition) that a $\mathrm{PA}$ degree is the same as
one which computes a nonstandard model of $\mathsf{PA}$. On the other
hand, we saw that there is a computable nonstandard model of the definitionally
equivalent theory $T_{0}$ constructed by Pakhomov -- and hence,
by a theorem of Knight \cite{knight_degrees_coded_in_jumps}, there
exists one in every degree. Let us define, for a theory $T$ that
is definitionally equivalent to $\mathsf{PA}$, its \emph{nonstandard
spectrum} as the set of degrees of nonstandard models of $T$.
\begin{question}
If $T$ is a c.e.\ theory that is definitionally equivalent to $\mathsf{PA}$,
we can see that its nonstandard spectrum is at most the set of all
Turing degrees, and at least the set of all $\mathrm{PA}$ degrees.
Does there exist such a theory whose nonstandard spectrum is strictly
in-between these two cases? More generally, what can the nonstandard
spectrum of a c.e.\ theory definitionally equivalent to $\mathsf{PA}$
look like?
\end{question}
The methods used in this paper are too coarse to address this question;
they would have to be cleverly modified to stop the nonstandard spectrum
from including $0$.

As a dual question, albeit unrelated to the contents of this paper:
We mentioned in the introduction that Lutz and Walsh \cite{lutz_walsh_tennenbaum}
produced a c.e.\ consistent theory such that no definitionally equivalent
theory admits a computable model. To be more precise, they actually
constructed an entire family of c.e.\ consistent theories $T_{\mathrm{LW}}(R)$,
parametrized by an infinite computable binary tree $R$ none of whose
paths are ``guessable'' (a term defined in \cite{lutz_walsh_tennenbaum}),
such that no model of any theory definitionally equivalent to $T_{\mathrm{LW}}(R)$
admits a computable model. We think it would be interesting to ask
how close we can get such models to computable.
\begin{question}
Define the Lutz--Walsh spectrum of an infinite computable binary
tree $R$ none of whose paths are guessable, say $\mathrm{LWSpec}(R)$,
as the set of Turing degrees that compute a model of a theory definitionally
equivalent to $T_{\mathrm{LW}}(R)$. It is clear that this $\mathrm{LWSpec}(R)$
always contains all $\mathrm{PA}$ degrees, and in \cite{lutz_walsh_tennenbaum}
it is proven that $\mathrm{LWSpec}(R)$ does not contain $0$ (if,
again, none of the paths of $R$ are guessable). Does $\mathrm{LWSpec}(R)$
ever contain any degrees that are not $\mathrm{PA}$? If so, how far
from $\mathrm{PA}$ can we get? Is there such a tree $R$ for which
$\mathrm{LWSpec}(R)$ consists of all non-zero degrees? Is there any
nonzero degree that can never be in $\mathrm{LWSpec}(R)$?
\end{question}
\medskip{}

In the statement of \thmref{nestedpakhomov}, the arity of the predicate
used to replace set inclusion increased by $1$ for every jump in
complexity. We wonder if this is a logical necessity or merely an
artifact of our construction.
\begin{question}
For a given value of $n$, does there exist a theory $T_{Q}^{n}$,
axiomatizing a single ternary predicate $Q(x,y,z)$, such that $T_{Q}^{n}$
is definitionally equivalent to $\mathsf{PA}$ plus `all $\Pi_{n}$
truths' which admits a computable nonstandard model? What if we require
the predicate to be binary, instead?
\end{question}
\medskip{}

Regarding our jump inversion theorem, \thmref{main_final}, the following
question remains open:
\begin{question}
Does there exist a proof of \thmref{JIbool} (jump inversion for Boolean
algebras) as a corollary of \thmref{main_final}?
\end{question}
While \thmref{main_final} guarantees the existence of a $0'$-computable
isomorphism between the original structure and its computable copy,
most results in the literature can only guarantee (out of necessity)
a $0''$ or $0'''$-computable isomorphism. This is explained by the
fact that, in our proofs, we must first create a $0'$-computable
c.e.-typed copy which is isomorphic to the original structure, though
this isomorphism may itself need to be complex. Hirschfeldt asked
whether one could obtain a more direct result, for example by weakening
the assumptions of the QETP.
\begin{question}[Hirschfeldt]
Does there exist a version of \thmref{main_final} that requires
weaker assumptions, possibly using an infinite injury method instead
of finite injury, from which theorems such as \thmref{JIlinear-1}
(jump inversion for linear orders) or \propref{eqvinf} (jump inversion
for equivalence relations with infinitely many infinite classes) may
be obtained more easily?
\end{question}
Many general results in computable structure theory necessitate the
use of a finite relational signature, oftentimes for reasons similar
the ones discussed at the start of Section \ref{subsec:cetypedstructures}.
It was for these reasons that we introduced the notion of $0'$-computable
c.e.-typed structure. We wonder if our definition would be applicable
to other scenarios.
\begin{question}
Are there results in the literature, applicable to finite relational
signatures, which could be further (and usefully) generalized using
the notion of $0'$-computable c.e.-typed structure?
\end{question}
\medskip{}

In Section \ref{subsec:qetp}, we found it useful to work with $\llorc$
formulas, which are related to but not exactly the same as $\Sigma_{1}^{c}$
formulas. Ash, Knight, Manasse and Slaman showed that the definability
of a relation by a $\Sigma_{1}^{c}$ formula is related to computability-theoretic
properties, namely the notion of being a r.i.c.e. (relatively intrinsically
c.e.) relation; see \cites[Theorem II.16]{montalban_cst} for details.
We wonder if $\llorc$ formulas admit a similar treatment.
\begin{question}
Is there a computability-theoretic characterization of those relations
in a structure that are definable by a $\llorc$ formula?
\end{question}
\printbibliography

\end{document}